\documentclass[11pt,reqno]{amsart}
\usepackage{amssymb,amsmath,amsfonts,amsthm,enumerate,stmaryrd,tensor,mathtools,dsfont,upgreek,bbm,mathrsfs,nameref,microtype,xcolor,bm,tikz}
\usetikzlibrary{arrows}
\usepackage[left=0.65 in, right=0.65 in,top=0.65 in, bottom=0.65 in]{geometry}
\usepackage{hyperref,cleveref}

\numberwithin{equation}{subsection}
\definecolor{popblue}{RGB}{55,115,255}
\definecolor{lightbl}{RGB}{155,205,255}
\definecolor{depthbl}{RGB}{145,215,255}
\definecolor{fancyre}{RGB}{225,55,115}
\definecolor{darkblu}{RGB}{15,75,185}
\definecolor{mellowy}{RGB}{225,225,35}
\renewcommand{\j}[1]{\tjump{#1}}
\renewcommand{\tilde}[1]{\widetilde{#1}}
\renewcommand{\Bar}{\overline}

\renewcommand{\S}{\mathbb{S}}
\newcommand{\R}{\mathbb{R}}
\newcommand{\N}{\mathbb{N}}

\newcommand{\C}{\mathbb{C}}
\newcommand{\X}{\mathbb{X}}
\newcommand{\Xs}{\mathbb{X}^{\mathrm{sub}}}
\newcommand{\Y}{\mathbb{Y}}
\newcommand{\W}{\mathbb{W}}
\newcommand{\E}{\mathbb{E}}
\newcommand{\Es}{\mathbb{E}^{\mathrm{sub}}}
\newcommand{\F}{\mathbb{F}}
\newcommand{\Fs}{\mathbb{F}^{\mathrm{sub}}}
\renewcommand{\P}{\mathbb{P}}

\newcommand{\m}{\mathrm}

\newcommand{\lv}{\lVert}
\newcommand{\rv}{\rVert}

\newcommand{\al}{\alpha}
\newcommand{\be}{\beta}
\newcommand{\es}{\varnothing}
\newcommand{\lra}{\;\Leftrightarrow\;}
\newcommand{\ep}{\varepsilon}
\newcommand{\f}{\frac}
\newcommand{\sig}{\sigma}
\newcommand{\gam}{\gamma}
\newcommand{\del}{\delta}

\newcommand{\bn}{\binom}

\newcommand{\pd}{\partial}

\newcommand{\grad}{\nabla}
\newcommand{\bpm}{\begin{pmatrix}}
\newcommand{\epm}{\end{pmatrix}}

\newcommand{\loc}{\m{loc}}
\renewcommand{\bar}{\overline}

\newcommand{\emb}{\hookrightarrow}
\newcommand{\res}{\restriction}
\renewcommand{\le}{\leqslant}
\renewcommand{\ge}{\geqslant}
\newcommand{\tjump}[1]{\llbracket#1\rrbracket}

\newcommand{\norm}[1]{\left\lv#1\right\rv}
\newcommand{\bnorm}[1]{\Big\lv#1\Big\rv}

\newcommand{\tnorm}[1]{\lv#1\rv}
\newcommand{\bp}[1]{\Big(#1\Big)}
\renewcommand{\sp}[1]{\big(#1\big)}
\newcommand{\tp}[1]{(#1)}
\newcommand{\tfloor}[1]{\lfloor #1\rfloor}
\newcommand{\sfloor}[1]{\big\lfloor#1\big\rfloor}

\newcommand{\abs}[1]{\left|#1\right|}
\newcommand{\babs}[1]{\Big|#1\Big|}

\newcommand{\tabs}[1]{|#1|}
\newcommand{\ssb}[1]{\big[{#1}\big]}
\newcommand{\tsb}[1]{[{#1}]}
\newcommand{\bcb}[1]{\Big\{{#1}\Big\}}
\newcommand{\tcb}[1]{\{{#1}\}}
\newcommand{\br}[1]{\left\langle #1 \right\rangle}
\providecommand{\bbr}[1]{\Big\langle #1 \Big\rangle}

\providecommand{\tbr}[1]{\langle #1 \rangle}
\renewcommand{\bf}[1]{\mathbf{#1}}
\newcommand{\ii}{\m{i}}
\DeclareMathOperator{\Div}{div}

\newtheorem{prop}{\color{popblue}{Proposition}}[section]
\newtheorem{thm}[prop]{\color{popblue}{Theorem}}
\newtheorem{defn}[prop]{\color{popblue}{Definition}}
\newtheorem{lem}[prop]{\color{popblue}{Lemma}}
\newtheorem{coro}[prop]{\color{popblue}{Corollary}}
\newtheorem{rmk}[prop]{\color{popblue}{Remark}}
\newtheorem{exa}[prop]{\color{popblue}{Example}}

\newtheorem{innercustomthm}{\color{popblue}{Theorem}}
\newenvironment{customthm}[1]
{\renewcommand\theinnercustomthm{#1}\innercustomthm}
{\endinnercustomthm}
\author{Noah Stevenson}
\address{
	Department of Mathematics\\
	Princeton University\\
	Princeton, NJ 08544, USA
}
\email[N. Stevenson]{stevenson@princeton.edu}
\thanks{N. Stevenson was supported by an NSF Graduate Research Fellowship}
\author{Ian Tice}
\address{
	Department of Mathematical Sciences\\
	Carnegie Mellon University\\
	Pittsburgh, PA 15213, USA
}
\email[I. Tice]{iantice@andrew.cmu.edu}
\thanks{I. Tice was supported by an NSF Grant (DMS \#2204912). }

\title[Shallow water traveling waves]{
        The traveling wave problem for the shallow water equations: well-posedness and the limits of vanishing viscosity and surface tension
 }

\begin{document}
% _+__+_ -_+__+_ -_+__+_ -_+__+_ -_+__+_ -_+__+_ -_+__+_ -_+__+_ -_+__+_ -_+__+_ -_+__+_ -_+__+_ -_+__+_ -
\maketitle
% _+__+_ -_+__+_ -_+__+_ -_+__+_ -_+__+_ -_+__+_ -_+__+_ -_+__+_ -_+__+_ -_+__+_ -_+__+_ -_+__+_ -_+__+_ -
\begin{abstract}
In this paper we study solitary traveling wave solutions to a damped shallow water system, which is in general quasilinear and of mixed type. We develop a small data well-posedness theory and prove that traveling wave solutions are a generic phenomenon that persist with and without viscosity or surface tension and for all nontrivial traveling wave speeds, even when the parameters dictate that the equations are hyperbolic and have a sound speed. This theory is developed by way of a Nash-Moser implicit function theorem, which allows us to prove strong norm continuity of solutions with respect to the data as well as the parameters, even in the vanishing limits of viscosity and surface tension.   
\end{abstract}
% _+__+_ -_+__+_ -_+__+_ -_+__+_ -_+__+_ -_+__+_ -_+__+_ -_+__+_ -_+__+_ -_+__+_ -_+__+_ -_+__+_ -_+__+_ -

% _+__+_ -_+__+_ -_+__+_ -_+__+_ -_+__+_ -_+__+_ -_+__+_ -_+__+_ -_+__+_ -_+__+_ -_+__+_ -_+__+_ -_+__+_ -
\section{Introduction}
% _+__+_ -_+__+_ -_+__+_ -_+__+_ -_+__+_ -_+__+_ -_+__+_ -_+__+_ -_+__+_ -_+__+_ -_+__+_ -_+__+_ -_+__+_ -

% _+__+_ -_+__+_ -_+__+_ -_+__+_ -_+__+_ -_+__+_ -_+__+_ -_+__+_ -_+__+_ -_+__+_ -_+__+_ -_+__+_ -_+__+_ -
\subsection{The shallow water equations, traveling formulation, and the role of applied force}
% _+__+_ -_+__+_ -_+__+_ -_+__+_ -_+__+_ -_+__+_ -_+__+_ -_+__+_ -_+__+_ -_+__+_ -_+__+_ -_+__+_ -_+__+_ -

We begin by stating the time-dependent formulation of the damped shallow water model under consideration in this work. Let $d\in\N^+$ denote the spatial dimension, the most physically relevant options for which are $1$ and $2$ (but the analysis here works more generally). The shallow water system with applied forcing consists of equations coupling the velocity $\Bar{\upupsilon}:\R^+\times\R^d\to\R^d$ and the free surface $\Bar{\upeta}:\R^+\times\R^d\to\R$ to the applied forcing $\Bar{\bf{f}}:\R^+\times\R^d\to\R^d$. The system reads
\begin{equation}\label{the shallow water system}
    \begin{cases}
        \pd_t\Bar{\upeta}+\grad\cdot(\Bar{\upeta} \Bar{\upupsilon})=0,\\
        \Bar{\upeta}\tp{\pd_t\Bar{\upupsilon}+\Bar{\upupsilon}\cdot\grad\Bar{\upupsilon}}+\al \Bar{\upupsilon}-\Bar{\mu}\grad\cdot(\Bar{\upeta}\S\Bar{\upupsilon})+\Bar{\upeta}\grad(g\Bar{\upeta}-\Bar{\sig}\Delta\Bar{\upeta})+\Bar{\bf{f}}=0,
    \end{cases}
\end{equation}
where $\S\Bar{\upupsilon}$ is a viscous stress tensor-like object obeying the formula
\begin{equation}\label{viscous stress tensor like object}
    \S\Bar{\upupsilon}=\grad\Bar{\upupsilon}+\grad\Bar{\upupsilon}^{\m{t}}+2(\grad\cdot\Bar{\upupsilon})I_{d\times d},
\end{equation}
and $\Bar{\mu},\Bar{\sig}\in[0,\infty)$ are the viscosity and the surface tension, respectively. The parameter $\al\in\R^+$ is the characteristic slip speed, which controls the strength of the zeroth order damping effect. The gravitational field strength is $g\in\R^+$.  

The shallow water system~\eqref{the shallow water system}, which is often called the Saint-Venant system when $\bar{\mu} =\bar{\sigma}=0$, is an important model approximating the incompressible free boundary Navier-Stokes  system in a regime in which the fluid depth is much smaller than its characteristic horizontal scale, and as such is both physically and practically relevant.  A derivation of the system~\eqref{the shallow water system} from the free boundary Navier-Stokes  system proceeds through an asymptotic expansion and vertical averaging procedure; for full details we refer, for instance, to the surveys of Bresch~\cite{MR2562163} or Mascia~\cite{mascia_2010}.  

One can see that system~\eqref{the shallow water system} has formal similarities to the barotropic compressible Navier-Stokes equations, with $\Bar{\upupsilon}$ and $\Bar{\upeta}$ playing the roles of the fluid velocity and density, respectively, and with a quadratic pressure law (corresponding to $\Bar{\upeta}\grad g\Bar{\upeta} = \grad (\frac{g}{2} \Bar{\upeta}^2)$) and viscosity coefficients proportional to density (corresponding to $\Bar{\mu}\Bar{\upeta}$).  The first equation in~\eqref{the shallow water system} is thus analogously dubbed the continuity equation; it dictates how the free surface is stretched and transported by the velocity.  Note that this serves as a replacement for incompressibility, which is lost in the shallow limit.

The second equation in the system, which one could call the momentum equation, has several interesting features. The advective derivative $\Bar{\upeta}\tp{\pd_t\Bar{\upupsilon}+\Bar{\upupsilon}\cdot\grad\Bar{\upupsilon}}$ is balanced by both dissipative and hyperbolic terms in addition to the forcing $\Bar{\bf{f}}$. The dissipation structure, which is $\al\Bar{\upupsilon}-\Bar{\mu}\grad\cdot\tp{\Bar{\upeta}\mathbb{S}\Bar{\upupsilon}}$, has two parts. The former, zeroth order contribution is the manifestation of the Navier slip boundary condition in the shallow limit; it represents a sort of frictional drag induced by tangential fluid motion at the bottom boundary.  The latter,  second order contribution is a manifestation of the  viscous stress tensor from the incompressible Navier-Stokes system.  Analogous to how incompressibility is lost in the shallow water limit, the shallow water viscous stress tensor more closely resembles the viscous stress tensor from compressible Navier-Stokes.  On the other hand, the hyperbolic structure, which is $\Bar{\upeta}\grad\tp{g\Bar{\upeta}-\Bar{\sig}\Delta\Bar{\upeta}}$, is directly descended from the gravity-capillary operator in the free boundary Navier-Stokes system's dynamic boundary condition.

When  $\Bar{\bf{f}}=0$ the system~\eqref{the shallow water system} admits an equilibrium solution $\Bar{\upupsilon}=0$ and $\Bar{\upeta}=H$ for any $H\in\R$.  After fixing some $H\in\R^+$ (it is only physically meaningful to consider positive equilibrium height), we can perform a rescaling and perturbative reformulation to obtain a non-dimensional version of system~\eqref{the shallow water system} in which $(\al,g,H)\mapsto(1,1,1)$. Indeed, upon defining the characteristic length $\m{L}$ and time $\m{T}$
\begin{equation}\label{characteristic length time speed}
    \m{L}=\f{\sqrt{gH}}{\al}H\text{ and }\m{T}=\f{H}{\al},
\end{equation}
along with the unitless quantities $\mu,\sig\in[0,\infty)$ given by
\begin{equation}\label{unitless quantities}
    \mu^2=\f{\Bar{\mu}\m{T}}{\m{L}^2}=\f{\Bar{\mu}\al}{gH^2} \text{ and } \sig^2=\f{\Bar{\sig}}{g\m{L}^2}=\f{\Bar{\sig}\al^2}{g^2H^3},
\end{equation}
the system~\eqref{the shallow water system} transforms to the nondimensional perturbative form
\begin{equation}\label{the shallow water system nondimensional version}
    \begin{cases}
        \pd_t\upeta+\grad\cdot((1+\upeta)\upupsilon)=0,\\
        (1+\upeta)\tp{\pd_t\upupsilon+\upupsilon\cdot\grad \upupsilon}+ \upupsilon-\mu^2\grad\cdot((1+\upeta)\S\upupsilon)+(1+\upeta)\grad(\upeta-\sig^2\Delta\upeta)+\bf{f}=0,
    \end{cases}
\end{equation}
where $\upupsilon$, $\upeta$, and $\bf{f}$ in~\eqref{the shallow water system nondimensional version} are obtained from $\Bar{\upupsilon}$, $\Bar{\upeta}$, and $\Bar{\bf{f}}$ in~\eqref{the shallow water system} via the rescalings
\begin{multline}
    \Bar{\upupsilon}(t,x)=\f{\m{L}}{\m{T}}\cdot\upupsilon\tp{t/\m{T}, x/\m{L}},\quad\Bar{\upeta}(t,x)=\f{\al \m{L}}{\sqrt{gH}}\cdot\tp{1+\upeta\tp{t/\m{T}, x/\m{L}}},\\\text{and }\Bar{\bf{f}}(t,x)=\al\sqrt{gH}\cdot\bf{f}\tp{t/\m{T}, x/\m{L}},\quad\text{for }(t,x)\in[0,\infty)\times\R^d.
\end{multline}
We emphasize that in the system~\eqref{the shallow water system nondimensional version} the nondimensionalized viscosity $\mu^2$ and surface tension $\sig^2$ appear as squares only for the sake of convenience in our subsequent analysis. Furthermore, we see from~\eqref{unitless quantities} that $1/\mu^2$ is the Reynolds number of system~\eqref{the shallow water system nondimensional version}, while $1/\sig^2$ is the Bond number.

In the formulation~\eqref{the shallow water system nondimensional version}, one further interesting aspect of the equations is revealed: there is either a characteristic wave speed (the `speed of sound') or a more complicated dispersion relation, depending on whether $\sig=0$ or $\sig\neq0$. Indeed, if we  take the divergence of the second equation, and substitute in the first, we find that
\begin{equation}\label{hyperbolic-parabolic competition}
    \pd_t^2\eta  +(1-4\mu^2\Delta)\pd_t\eta - \Delta\tp{1-\sig^2\Delta}\eta=\text{forcing and nonlinear terms}.
\end{equation}
This reveals that $\eta$ obeys a dispersive or wave-like equation with a damping effect (the middle term on the left), driven by applied forces and nonlinear interactions. If we ignore the damping term and focus on the part $\pd_t^2\eta-\Delta\tp{1-\sig^2\Delta}\eta$, then we arrive at the dispersion relation $\omega(\xi)^2=  |\xi|^2  \tp{1+4\pi^2\sig^2|\xi|^2}$.  This reveals two notable features of $\sigma$:  first, if $\sigma =0$ then the group and phase velocities agree and have unit magnitude; second, if $\sigma >0$ then the group and phase velocities are colinear and have magnitudes of order $O(\abs{\xi})$ as $\abs{\xi} \to \infty$, but the ratio of the group speed to the phase speed converges to $2$,  which is the reciprocal of the ratio one obtains from the standard deep water dispersion relation, $\omega(\xi)^2 =  \abs{\xi}$.

In this work we are interested in traveling wave solutions to~\eqref{the shallow water system nondimensional version}, which are solitary perturbations of the equilibrium $(\upupsilon,\upeta)=(0,0)$ that are stationary when viewed in a coordinate system moving at a fixed velocity.  By the rotational invariance of the equations, we lose no generality in assuming that the velocity of the traveling wave is colinear with $e_1\in \R^d$, the unit vector in the first coordinate direction.  We thus fix $\gam\in\R^+$ to be the speed of our traveling wave and posit the traveling ansatz
\begin{equation}\label{the traveling ansatz}
    \upupsilon(t,x)=v(x-t\gam e_1),\quad\upeta(t,x)=\eta(x-t\gam e_1),\quad\bf{f}(t,x)=\mathcal{F}(x-t\gam e_1), \text{ for } (t,x)\in[0,\infty)\times\R^d,
\end{equation}
for new unknowns $v:\R^d\to\R^d$, $\eta:\R^d\to\R$ and traveling applied forcing $\mathcal{F}:\R^d\to\R^d$.  Under this ansatz, the system~\eqref{the shallow water system nondimensional version} becomes
\begin{equation}\label{traveling wave formulation of the equation}
    \begin{cases}
    \grad\cdot((1+\eta)(v-\gam e_1))=0,\\
    (1+\eta)(v-\gam e_1)\cdot\grad v+v-\mu^2\grad\cdot((1+\eta)\S v)+(1+\eta)(I-\sig^2\Delta)\grad\eta+\mathcal{F}=0.
    \end{cases}
\end{equation}
Due to the fact that the hyperbolic part of~\eqref{the shallow water system nondimensional version}, which is described in~\eqref{hyperbolic-parabolic competition}, has speed of sound equal to unity when neglecting capillary effects, we expect that the traveling problem~\eqref{traveling wave formulation of the equation} behaves quite differently depending on whether or not the speed of the applied forcing is subsonic $(\gam<1)$ or supersonic ($\gam\ge 1$).

We now turn to a discussion of the precise form of the forcing $\mathcal{F}$ appearing in the second equation of~\eqref{traveling wave formulation of the equation}, endeavoring to make a physically meaningful choice.  For this one needs consider the derivation of the traveling shallow water equations from the traveling free boundary incompressible Navier-Stokes equations.  We perform the relevant calculations in Appendix~\ref{Ian's appendix post appendectomy}. The idea is that for the latter Navier-Stokes system in traveling form, a forcing term $\phi:\R^{d+1}\to\R^{d+1}$ acts on the bulk and a symmetric tensor $\Xi:\R^{d+1}\to\R_{\m{sym}}^{(d+1)\times(d+1)}$ generates a stress that acts on the free surface. For instance, $\Xi$ can model wind or other localized surface stress, and $\phi$ can be a local gravity perturbation as in a primitive model of an ocean-moon system.

As we show in \eqref{derived traveling shallow water system}, in the shallow water limit the applied force and stress dictate that $\mathcal{F}$ has the form
\begin{equation}\label{physically relevant form of the forcing}
    \mathcal{F}=\int_{0}^{1+\eta}f(\cdot,y)\;\m{d}y+\Upxi(\cdot,1+\eta)\grad\eta-\xi(\cdot,1+\eta)-\upxi(\cdot,1+\eta)\grad\eta-\tp{1+\eta}\grad\tp{\upxi(\cdot,1+\eta)}, %there was a typo in the last term here
\end{equation}
where
\begin{equation}
    \phi=(f,\phi\cdot e_n),\;\Xi=\bpm\Upxi&\xi\\\xi&\upxi\epm,\;\Upxi:\R^{d+1}\to\R^{d\times d}_{\m{sym}},\;\xi:\R^{d+1}\to\R^d,\;\upxi:\R^{d+1}\to\R,\; f:\R^{d+1}\to\R^d.
\end{equation}
While~\eqref{physically relevant form of the forcing} gives the correct physically relevant form of the forcing, it clearly possesses some structural redundancies.  Thus, for the purposes of our analysis it is convenient to consider a class of forcing that is symbolically less cumbersome, yet of the type~\eqref{physically relevant form of the forcing}.  We discuss this now.

As we mentioned previously, the relationship between the wave speed $\gam$ and the speed of sound is important. We elect to distinguish two different parameter regimes, depending on $\gam\in\R^+$, in which the regularity counting schemes for the solutions differ.  The first is the \emph{omnisonic} regime, which is wave speed agnostic, and the second is the \emph{subsonic} regime, which is determined via $\gam\in(0,1)$. Motivated by~\eqref{physically relevant form of the forcing} and the anticipation of differing quantitative regularity, we are thus led us to formulate the forcing terms in the following two different ways, depending on which regime we are in: 
\begin{equation}\label{conditions on the forcing}
    \mathcal{F}=\begin{cases}
        \m{Tr}\tau[\eta]\grad\eta+(\mu+\sig)\tau[\eta]\grad\eta+\varphi[\eta]&\text{when~\eqref{traveling wave formulation of the equation} is thought of as omnisonic},\\
        \tau[\eta]\grad\eta+\varphi[\eta]&\text{when~\eqref{traveling wave formulation of the equation} is thought of as subsonic},
    \end{cases}
\end{equation}
where $\tau[\eta]:\R^d\to\R^{d\times d}$, $\varphi[\eta]:\R^d\to\R^d$ are assumed to have the following form
\begin{equation}\label{special form of the superposition nonlinearities}
    \tau[\eta]=\sum_{i=0}^\ell \tau_i\eta^i,\quad\varphi[\eta]=\sum_{i=0}^\ell \varphi_i\eta^i,
\end{equation}
with $\tau_i:\R^d\to\R^{d\times d}$ and $\varphi_i:\R^d\to\R^d$ for $i\in\tcb{0,\dotsc,\ell}$ and fixed $\ell\in\N$. These are meant to be an $\ell^{\m{th}}$ order approximation of superposition type nonlinearities
\begin{equation}\label{superposition approximation}
    \tau[\eta]\approx\tau(\cdot,\eta),\quad\varphi[\eta]\approx\varphi(\cdot,\eta),\quad\tau:\R^d\times\R\to\R^{d\times d},\quad\varphi:\R^d\times\R\to\R^d
\end{equation}
where for $i\in\tcb{0,1,\dots,\ell}$ we make the identification
\begin{equation}
    \tau_i=(i!)^{-1}(\pd_{d+1}^i\tau)(\cdot,0),\quad\varphi_i=(i!)^{-1}(\pd_{d+1}^{i}\varphi)(\cdot,0).
\end{equation}
We elect to have the generators $\tau[\eta]$ and $\varphi[\eta]$ of the forcing term $\mathcal{F}$ in~\eqref{traveling wave formulation of the equation} have the form~\eqref{special form of the superposition nonlinearities} for the following reasons. First, they are strictly more general than the special case of $\ell=0$, which is perhaps the most physically relevant. The flexibility of $\ell\in\N$ allows for a very large class of superposition type nonlinearities as in~\eqref{superposition approximation} to be considered exactly or approximately. Finally, the appearance of, at worst, polynomial nonlinearities in~\eqref{special form of the superposition nonlinearities} allows us to avoid a lengthy and unenlightening digression in which we would handle the fully nonlinear compositions.   We emphasize that the reader interested in including such compositions in the analysis of~\eqref{traveling wave formulation of the equation} could do so  by importing techniques from our previous work on the free boundary compressible Navier-Stokes system~\cite{stevenson2023wellposedness}.

The difference between the two forms of $\mathcal{F}$ in~\eqref{conditions on the forcing} is that in the omnisonic regime we are asking that tensorial part of $\tau$ acting on the gradient of $\eta$ obeys an extra smallness condition in the limit $\mu$ and $\sig$ tending to zero, as quantified by the $(\mu+\sig)$ prefactor in~\eqref{conditions on the forcing}. This is purely for technical reasons, and the authors are ignorant of whether or not removing this condition presents a genuine obstruction.

Finally, we discuss the necessity of including the forcing term $\mathcal{F}$. An elementary power-dissipation calculation, which we carry out in Appendix~\ref{appendix on dissipation calculation}, shows that solutions to~\eqref{traveling wave formulation of the equation} belonging to an $L^2$-based Sobolev space framework satisfy the identity   
\begin{equation}\label{power dissipation identity}
    \int_{\R^d}|v|^2+{\mu^2}(1+\eta)\bp{\f12|\mathbb{D}^0v|^2+2\bp{1+\f{1}{d}}|\grad\cdot v|^2}+\mathcal{F}\cdot v=0,
\end{equation}
where $\mathbb{D}^0v=\grad v+\grad v^{\m{t}}-(2/d)\tp{\grad\cdot v}I_{d\times d}$. Therefore, if $1+\eta\ge0$ and $\mathcal{F}=0$, then~\eqref{power dissipation identity} implies that $v=0$. We can then return to the second equation in~\eqref{traveling wave formulation of the equation} to deduce that $\grad\eta=0$, and hence $\eta$ is a constant. This tells us that nontrivial solitary traveling wave solutions to the shallow water system studied here cannot exist unless they are generated by traveling applied stress or force. We emphasize that in this identity we are crucially using that the system~\eqref{traveling wave formulation of the equation} has a drag term, which yields the $L^2$ control on the left.

We thus arrive at our main goal in this paper: for every choice of physical parameters $(\gam,\mu,\sig)\in\R^+\times[0,\infty)^2$ we wish to identify an open set of stress and forcing data that determine $\mathcal{F}$ and give rise to locally unique and nontrivial solutions to the traveling shallow water system. Moreover, we wish to establish well-posedness in the sense of continuity of the solution as a function of the data and physical parameters, even in the vanishing limit of viscosity or surface tension. See Figure~\ref{figure that is our goal} for various plots of some of the traveling wave solutions considered in this work.

\begin{figure}[!h]
    \centering
    \scalebox{0.27}{\includegraphics{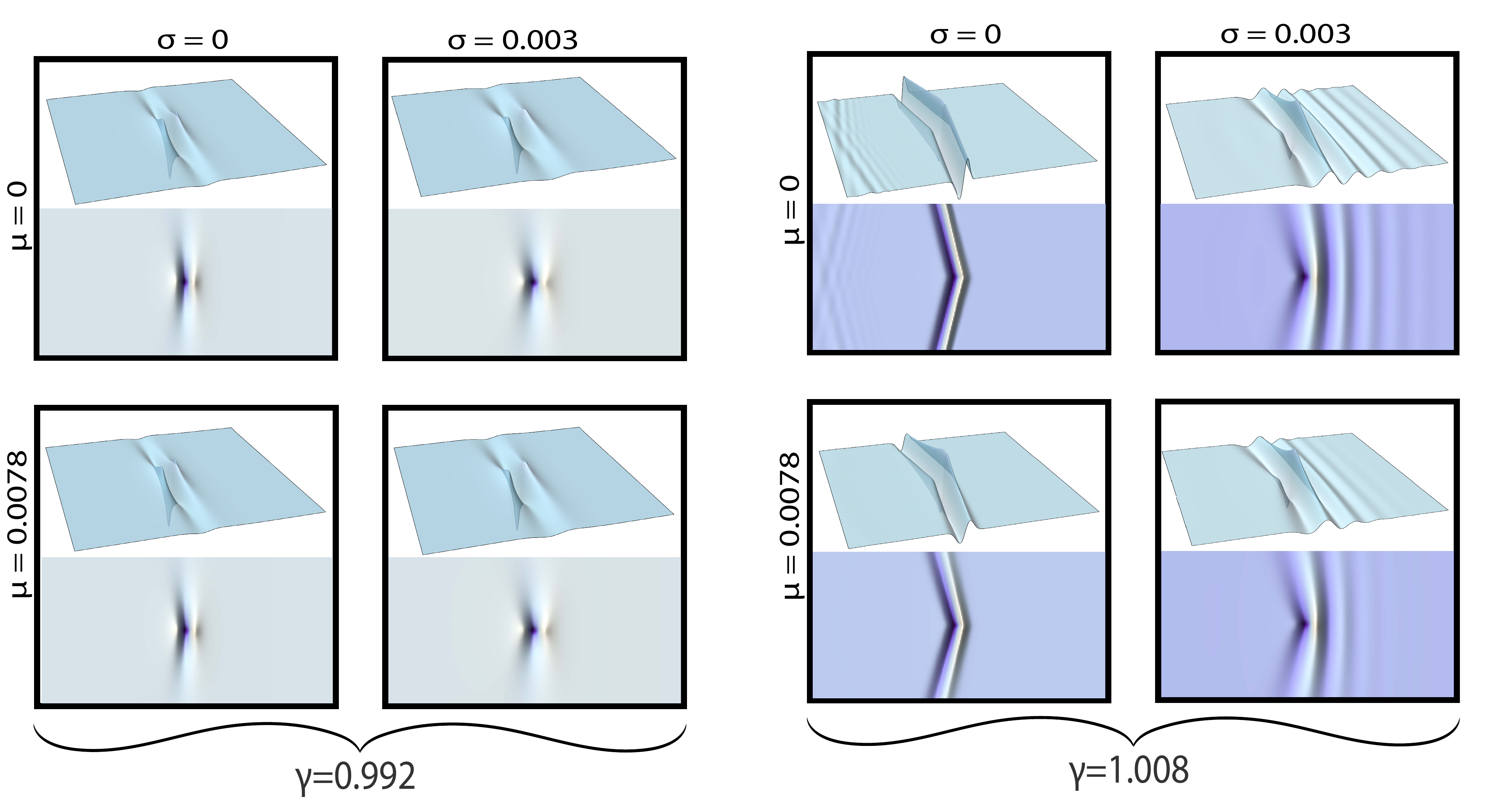}}
    \caption{
    Depicted here are 3D and relief plots of numerical simulations of the free surface component of the solution to~\eqref{traveling wave formulation of the equation} generated by a fixed applied Gaussian stress moving to the right for various values of $\gam\in\tcb{0.992,1.008}$, $\mu\in\tcb{0,0.0078}$, and $\sig\in\tcb{0,0.003}$. Notice that the wave speed $\gam$ transitioning from subsonic to supersonic leads to a shock wave-like formation in the free surface, though no discontinuity actually forms.}
    \label{figure that is our goal}
\end{figure}

% _+__+_ -_+__+_ -_+__+_ -_+__+_ -_+__+_ -_+__+_ -_+__+_ -_+__+_ -_+__+_ -_+__+_ -_+__+_ -_+__+_ -_+__+_ -
\subsection{Previous work}
% _+__+_ -_+__+_ -_+__+_ -_+__+_ -_+__+_ -_+__+_ -_+__+_ -_+__+_ -_+__+_ -_+__+_ -_+__+_ -_+__+_ -_+__+_ -

We now turn to a brief survey of the literature associated with the shallow water systems~\eqref{the shallow water system} and~\eqref{traveling wave formulation of the equation} and their variants.  There is a large body of work,  so we shall restrict our attention to those results most closely related to ours. 

We emphasize from the start that there are numerous versions of the shallow water equations, all of which are derived from the free boundary Euler or Navier-Stokes systems in one way or another; for various examples, we refer to~\cite{MR1324142,MR1821555,MR1975092,MR2118849,MR2281291,MR2397996,MR4105349}, and in particular the articles of Bresch~\cite{MR2562163} and Mascia~\cite{mascia_2010}. As such, it is important that we give at least a rapid review of literature associated with the derived models' parents.

For a thorough review of the fully dynamic free boundary incompressible Navier-Stokes equations in various geometries we refer to the surveys of Zadrzy\'{n}ska~\cite{Zadrynska_2004} and Shibata and Shimizu~\cite{SS_2007}.  While neglecting surface tension, Beale~\cite{Beale_1981} established local well-posedness. Beale~\cite{Beale_1983} then, with surface tension accounted for, established the existence of global solutions and derived their decay properties with Nishida~\cite{BN_1985}. In other settings, solutions with surface tension were also constructed by Allain~\cite{allain_1987}, Tani~\cite{tani_1996}, Bae~\cite{Bae_2011}, and Shibata and Shimizu~\cite{SS_2011}.  Similarly, solutions without surface tension were produced by Abels~\cite{Abels_2005_3},  Guo and Tice~\cite{GT_2013_inf,GT_2013_lwp},  and Wu~\cite{Wu_2014}. 

The inviscid analog of the aforementioned free boundary Navier-Stokes in traveling wave form, which is also known as the traveling water wave problem, has been the subject of major work for more than a century. For a thorough review, see the survey articles of Toland~\cite{Toland_1996}, Groves~\cite{Groves_2004}, Strauss~\cite{Strauss_2010}, and Haziot, Hur, Strauss, Toland, Wahl\'en, Walsh, and Wheeler~\cite{MR4406719}.

In contrast, work on the traveling wave problem for the free boundary Navier-Stokes equations has only recently appeared in the literature. Leoni and Tice~\cite{leoni2019traveling} built a well-posedness theory for small forcing and stress data, provided $\gamma \neq 0$.  This was generalized to multi-layer and stationary ($\gam=0$) settings  by Stevenson and Tice~\cite{MR4337506,stevensontice2023wellposedness} and to inclined and periodic configurations by Koganemaru and Tice~\cite{koganemaru2022traveling}. The corresponding well-posedness theory for the free boundary compressible Navier-Stokes equations was developed by Stevenson and Tice~\cite{stevenson2023wellposedness}.  Traveling waves for the Muskat problem were constructed with similar techniques by Nguyen and Tice~\cite{nguyen_tice_2022}. There are also experimental studies of viscous traveling waves; for details, we refer to the work of Akylas, Cho, Diorio, and Duncan~\cite{CDAD_2011,DCDA_2011}, Masnadi and Duncan~\cite{MD_2017}, and Park and Cho~\cite{PC_2016,PC_2018}.

To reiterate the goal of this work, we wish to extend the previously mentioned theory on forced traveling wave solutions for viscous or dissipative fluid models to the system of shallow water equations. This then brings us to the next body of literature to be surveyed, which are results on shallow water systems similar to the one considered in this work. For the sake of brevity, we shall only consider the following two families of results: 1) the local and global well-posedness for variations on the time dependent shallow water system~\eqref{the shallow water system} and 2) the roll-waves phenomena for a tilted version of the equations.

Ton~\cite{MR0605432} established local well-posedness for a version of the viscous equations without surface tension, with small initial data in 2D bounded domains. This result was extended by Kloeden~\cite{MR0777469} to give global in time existence for data near equilibrium.  The inviscid, undamped, version of the equations without surface tension, also known as the Saint-Venant equations, fits within the realm of hyperbolic conservation laws when restricted to one dimension; consequently, the result of Lions, Perthame, and Souganidis~\cite{MR1383202}  on the existence and stability of solutions to 1-dimensional hyperbolic conservation laws in particular establishes existence for 1D Saint-Venant. Sundbye~\cite{MR1402600} studied global existence for near equilibrium data when considering a Coriolis force and external force in the viscous equations without surface tension in 2D. Bresch and Desjardins~\cite{MR1955592} proved the existence of global weak solutions to a viscous shallow water model in 2D and studied the quasi-geostrophic limit. Wang and Xu~\cite{MR2155011} obtained local and global solutions for any and small, respectively, initial data for the equations in 2D while neglecting both the drag term and surface tension. Hao~\cite{MR2601082} studied global in time existence for near equilibrium data for the Cauchy problem including the effects of surface tension in 2D. Bresch and Nobel~\cite{MR2975338} rigorously justified the shallow water limit of the dynamic problem in 2D without surface tension by utilizing `higher order' versions of the equations. Chen and Perepelitsa~\cite{MR3000715} studied the inviscid limit in a weak topology without a drag term and without surface tension in 1D. For a two dimensional rotating shallow water model, Wang~\cite{MR3708163} analyzed the inviscid limit. Haspot~\cite{MR3612532} proved global in time existence for a certain class of large initial data to a related variation on the shallow water equations. Yang, Liu, and Hao~\cite{MR4280587} study the simultaneous limit of vanishing viscosity and Froude number for a two dimension system without surface tension.

We now briefly discuss  the literature on roll waves, which are a sort of unforced traveling wave phenomenon that can appear in variants on the shallow water system derived with a non-trivially tilted bottom (equivalently, derived with a tangential component of gravity). Dressler~\cite{MR0033717} proved the existence of discontinuous roll wave solutions to an inviscid shallow water model in 1D. Barker, Johnson, Rodrigues, Zumbrun and Barker, Johnson, Noble, Rodrigues, Zumbrun~\cite{MR2813829,MR3600832} performed a stability analysis for solitary roll wave solutions to the viscous equations in 1D with various forms of the drag term.  Also in 1D, Balmforth and Madre~\cite{MR2259999} studied the dynamics and stability of roll waves, including the effects with bottom drag and viscosity while doing so. The linear stability of 1D viscous roll waves was further studied by Noble~\cite{MR2372484}.  Periodic roll wave trains were studied numerically in 1D, and also for a reduced model in 2D, by Ivanova, Nkonga, Richard~\cite{MR3721925}. Johnson, Noble, Rodrigues, Yang, and Zumbrun~\cite{MR3933411} studied the stability of discontinuous roll wave solutions to the 1D inviscid Saint-Venant equations with a certain drag term.

Our results for the shallow water system are the first to establish well-posedness of the traveling equations with small forcing; moreover, we find that this can be done regardless whether or not one accounts for the effects of surface tension and viscosity. In particular, we rule out the possibility of unforced roll waves with a normal-pointing gravitational field and establish the generality of traveling waves with forcing, which are a special type of global in time solutions to~\eqref{the shallow water system}.

We emphasize that, while our analysis is capable of handling viscosity and surface tension either on or off, the drag term is essential and cannot be neglected.  Without viscosity or surface tension, this results in a hyperbolic problem with frictional damping.  There is a large literature related to the dynamics of such problems, which we now briefly survey to conclude our discussion of prior work. Matsumura~\cite{MR0470507} studied a damped quasilinear wave equation and employed the frictional damping term to achieve global existence.  MacCamy~\cite{MR0452184} proved global well-posedness for an integrodifferential equation related to a damped wave equation. Nohel~\cite{MR0535352} studied a forced quasilinear wave equation with a frictional damping term.   Greenberg and Li~\cite{MR0737964} studied the effects of boundary damping on a quasilinear wave equation.  Dafermos~\cite{MR1359325} studied a system of damped conservation laws.  Levine, Park, and Serrin~\cite{MR1659905} studied the Cauchy problem for a nonlinearly damped wave equation.

% _+__+_ -_+__+_ -_+__+_ -_+__+_ -_+__+_ -_+__+_ -_+__+_ -_+__+_ -_+__+_ -_+__+_ -_+__+_ -_+__+_ -_+__+_ -
\subsection{Main results and discussion}\label{section on main results and discussion}
% _+__+_ -_+__+_ -_+__+_ -_+__+_ -_+__+_ -_+__+_ -_+__+_ -_+__+_ -_+__+_ -_+__+_ -_+__+_ -_+__+_ -_+__+_ -

In this paper we present three principal results on the system~\eqref{traveling wave formulation of the equation}, which is the shallow water equations in traveling wave formulation. In order to properly state these results, we must first embark on a brief discussion of a certain nonstandard function space that will serve as our container for the perturbative free surface unknown, as in the prior works \cite{koganemaru2022traveling, leoni2019traveling,nguyen_tice_2022,MR4337506,stevenson2023wellposedness}.  Given $\R\ni s\ge 0$ we define the anisotropic Sobolev space
\begin{equation}\label{some definition of the anisobros}
    \mathcal{H}^s(\R^d)=\tcb{\eta\in\mathscr{S}^\ast(\R^d;\R)\;:\;\mathscr{F}[f]\in L^1_{\loc}(\R^d;\C) \text{ and }\tnorm{\eta}_{\mathcal{H}^s}<\infty}
\end{equation}
equipped with the norm
\begin{equation}\label{some norm on the anisobros}
    \tnorm{f}_{\mathcal{H}^s}=\bp{\int_{\R^d}\sp{|\xi|^{-2}\tp{|\xi|^4+\xi_1^2}\mathds{1}_{B(0,1)}(\xi)+\tbr{\xi}^{2s}\mathds{1}_{\R^d\setminus B(0,1)}(\xi)}\tabs{\mathscr{F}[\eta](\xi)}^2\;\m{d}\xi}^{1/2}.
\end{equation}
We refer to Appendix~\ref{appendix on the anisobros} for more information on these function spaces. For the purposes of stating our main results, we only mention here that we have the embeddings $H^s(\R^d)\emb\mathcal{H}^s(\R^d)\emb H^s(\R^d)+C^\infty_0(\R^d)$ with the former being strict if and only if $d\ge 2$.

Our first two main theorems handle small data well-posedness in the `easy cases' of system~\eqref{traveling wave formulation of the equation} in which the equations obey an advantageous elliptic derivative counting and do not exhibit derivative loss. The first of these works in any dimension and requires that both the viscosity and surface tension parameters are different from zero.

\begin{customthm}{1}[Proved in Theorem~\ref{thm on well-posedness II}]\label{main theorem no. 1}
    Let $\N\ni s>d/2$. There exist open sets
    \begin{equation}
        (\R^+)^3\times\tcb{0}^{\ell+1} \times\tcb{0}^{\ell+1} \subset\mathscr{U}_s\subset(\R^+)^3 \times (H^s(\R^d;\R^{d\times d}))^{\ell+1} \times ( H^s(\R^d;\R^d))^{\ell+1}
    \end{equation}
    and
    \begin{equation}
        0\in \mathscr{V}_s(\bf{p})\subset H^{2+s}(\R^d;\R^d)\times\mathcal{H}^{3+s}(\R^d) \text{ for }\bf{p}\in(\R^+)^3,
    \end{equation}
    as well as a smooth mapping
    \begin{equation}\label{smooth mapping, easy case}
        \mathscr{U}_s\ni(\gam,\mu,\sig,\tcb{\tau_i}_{i=0}^\ell,\tcb{\varphi_i}_{i=0}^\ell)\mapsto(v,\eta)\in\bigcup_{\bf{p}\in\tp{\R^+}^3}\mathscr{V}_s(\bf{p})
    \end{equation}
    with the following property. For all $(\gam,\mu,\sig,\tcb{\tau_i}_{i=0}^\ell,\tcb{\varphi_i}_{i=0}^\ell)\in\mathscr{U}_s$ the pair $(v,\eta)$ associated  through~\eqref{smooth mapping, easy case} is the unique $(v,\eta)\in\mathscr{V}_s(\gam,\mu,\sig)$ such that system~\eqref{traveling wave formulation of the equation} is classically satisfied with wave speed $\gam\in\R^+$, viscosity $\mu\in\R^+$, surface tension $\sig\in\R^+$, and forcing $\mathcal{F}$ determined by $\eta$ and the data $(\tcb{\tau_i}_{i=0}^\ell,\tcb{\varphi_i}_{i=0}^\ell)$ via
    \begin{equation}\label{the easy case forcing term is written here}
        \mathcal{F}=\sum_{i=0}^\ell\eta^i\tp{\tau_i\grad\eta+\varphi_i}.
    \end{equation}
\end{customthm}

As it turns out, if we specialize to one dimension, there is a larger parameter space region in which we can avoid the complications of derivative loss. The point of our next main theorem is to record the refinements that can be made to Theorem~\ref{main theorem no. 1} under this specialization.

\begin{customthm}{2}[Proved in Theorems~\ref{thm on abstract 1d analysis} and~\ref{thm on one dimensional analysis, PDE formulation}]\label{main theorem no. 1.5}
    Let $d=1$ and $\N\ni s\ge 1$. Cover $(\R^+\setminus\tcb{0})\times\R^2\subseteq\mathcal{P}_1\cup\mathcal{P}_2\cup\mathcal{P}_3$, where the sets $\mathcal{P}_i$ are given in the second item of Definition~\ref{defn of reparameterization operators}, and fix $i\in\tcb{1,2,3}$. There exists an open set
    \begin{equation}
    \mathcal{P}_i\times\tcb{0}^{\ell+1}\times\tcb{0}^{\ell+1}\subset\mathscr{U}_s^i\subset\mathcal{P}_i\times(H^s(\R))^{\ell+1}\times(H^s(\R))^{\ell+1}
    \end{equation}
    and a sequence of radii $\tcb{\ep^i_{s,\bf{p}}}_{\bf{p}\in\mathcal{P}_i}\subset\R^+$ such that the following hold.
    \begin{enumerate}
        \item To each $(\gam,\mu,\eta,\tcb{\tau_j}_{j=0}^\ell,\tcb{\varphi_j}_{j=0}^\ell)\in\mathscr{U}^i_s$ there exists a unique $\eta$ satisfying $P^i_{\gam,\mu,\sig}\eta\in B_{H^s}(0,\ep^i_{s,(\gam,\mu,\sig)})$, where the operators $P^i_{\gam,\mu,\sig}$ are given in the first item of Definition~\ref{defn of reparameterization operators}, such that upon defining $v=\gam\eta/(1+\eta)$ we have that $(v,\eta)$ are a solution to~\eqref{traveling wave formulation of the equation} with data $\mathcal{F}$ determined from $\eta$, $\tcb{\tau_j}_{j=0}^\ell$, and $\tcb{\varphi_j}_{j=0}^\ell$ as in~\eqref{the easy case forcing term is written here}.
        \item The following mappings are well-defined and continuous.
        \begin{equation}
            \mathscr{U}^i_s\ni(\gam,\mu,\sig,\tcb{\tau_j}_{j=0}^\ell,\tcb{\varphi_j}_{j=1}^\ell)\mapsto\begin{cases}
                (v,\eta)\in (H^{1+s}(\R))^2,\\\
                \mu^2(v,\eta)\in(H^{2+s}(\R))^2,\\
                \sig^2(v,\eta)\in(H^{3+s}(\R))^2.
            \end{cases}
        \end{equation}
    \end{enumerate}
\end{customthm}

 The proofs of Theorems~\ref{main theorem no. 1} and~\ref{main theorem no. 1.5} proceed via applications of the standard implicit function theorem. As such, the theorems offer the following advantages. The number of derivatives required on the data to produce solutions is relatively low, at least compared to our next two main theorems. Additionally, the methodology is relatively simple and requires a low level of technical finesse. On the other hand, these results say nothing about the boundary of the $(\mu,\sig)$ parameter space nor the corresponding limits.  The reason for this defect is essentially because these limiting cases exhibit derivative loss, and thus the standard implicit function theorem does not provide the requisite tools for studying them.  By swapping to a Nash-Moser implicit function theorem, we are led to our next main result, which covers what we call the omnisonic case, since it applies equally well for all positive wave speeds.

\begin{customthm}{3}[Proved in Section~\ref{section on PDE construction}]\label{main theorem no. 2}
    Let $\varsigma=19+2\tfloor{d/2}\in\N$. There exists a nonincreasing sequence of relatively open subsets
    \begin{equation}\label{thm3 Us}
        \R^+\times[0,\infty)^2\times\tcb{0}^{\ell+1}\times\tcb{0}^{\ell+1}\subset\mathcal{U}_s\subset\R^+\times[0,\infty)^2\times(H^\varsigma(\R^d;\R^{d\times d}))^{\ell+1}\times(H^\varsigma(\R^d;\R^d))^{\ell+1}
    \end{equation}
    for $\N\ni s\ge\varsigma$, and a family of open subsets
    \begin{equation}\label{thm4 Vq}
        0\in\mathcal{V}(\mathfrak{q})\subset H^{\varsigma}(\R^d;\R^d)\times\mathcal{H}^{\varsigma}(\R^d)
        \text{ for }
        \mathfrak{q}\in\R^+\times[0,\infty)^2
    \end{equation}
    such that the following hold.
    \begin{enumerate}
        \item \emph{Existence and uniqueness:} For all $(\gam,\mu,\sig,\tcb{\tau_i}_{i=0}^\ell,\tcb{\varphi_i}_{i=0}^\ell)\in\mathcal{U}_\varsigma$ there exists a unique $(v,\eta)\in\mathcal{V}(\gam,\mu,\sig)$ such that system~\eqref{traveling wave formulation of the equation} is classically satisfied  with wave speed $\gam\in\R^+$, viscosity $\mu\in[0,\infty)$, surface tension $\sig\in[0,\infty)$, and data $\mathcal{F}$ determined from $\eta$, $\mu+\sig$, and $(\tcb{\tau_i}_{i=0}^\ell,\tcb{\varphi_i}_{i=0}^\ell)$ via the omnisonic case of~\eqref{conditions on the forcing}.
        \item \emph{Regularity promotion:} If $\N\ni s\ge\varsigma$ and 
        \begin{equation}
            (\gam,\mu,\sig,\tcb{\tau_i}_{i=0}^\ell,\tcb{\varphi_i}_{i=0}^\ell)\in\mathcal{U}_s\cap\tp{\R^3\times(H^s(\R^d;\R^{d\times d}))^{\ell+1}\times(H^s(\R^d;\R^d))^{\ell+1}}
        \end{equation}
        then the corresponding solution $(v,\eta)\in\mathcal{V}(\gam,\mu,\sig)$ produced by the previous item obeys the inclusions
        \begin{equation}\label{reg prom estimate}
            \tbr{\mu\grad}^2 v\in H^{s}(\R^d;\R^d)
            \text{ and }
            (1+(\mu+\sig)|\grad|)\tbr{\sig\grad}^2\eta\in\mathcal{H}^s(\R^d).
        \end{equation}
  
        \item \emph{Continuous dependence:} For any $\N\ni s\ge\varsigma$ the three maps
        \begin{multline}\label{the super maps}
            \mathcal{U}_s\cap\tp{\R^3\times(H^s(\R^d;\R^{d\times d}))^{\ell+1}\times(H^s(\R^d;\R^d))^{\ell+1}}\ni(\gam,\mu,\sig,\tcb{\tau_i}_{i=0}^\ell,\tcb{\varphi_i}_{i=0}^\ell)\\
             \mapsto
             \begin{cases}
             (v,\eta)\in H^s(\R^d;\R^d)\times\mathcal{H}^{s}(\R^d) \\
             (\mu\grad v,\mu^2\grad^2v)\in H^s(\R^d;\R^{d^2})\times H^s(\R^d;\R^{d^3}) \\
             (\mu+\sig)(\grad\eta,\sig\grad^2\eta,\sig^2\grad^3\eta)\in H^s(\R^d;\R^d)\times H^s(\R^d;\R^{d^2})\times H^s(\R^d;\R^{d^3})
             \end{cases}
        \end{multline}
        are all continuous.
    \end{enumerate}
\end{customthm}

Theorem~\ref{main theorem no. 2} is only optimal, in the sense of regularity counting in the limiting cases of~\eqref{the super maps}, in the supersonic regime defined via $\gam\ge 1$. In the subsonic case of $0<\gam<1$, we can invoke the Nash-Moser strategy with  improved estimates along the way and arrive at our fourth and final main result.

\begin{customthm}{4}[Proved in Section~\ref{section on PDE construction}]\label{main theorem no. 3}
    Let $\varpi=17+2\tfloor{d/2}\in\N$. There exists a nonincreasing sequence of relatively open subsets
    \begin{equation}
        (0,1)\times[0,\infty)^2\times\tcb{0}^{\ell+1}\times\tcb{0}^{\ell+1}\subset\mathcal{U}^{\m{sub}}_s\subset(0,1)\times[0,\infty)^2\times(H^\varpi(\R^d;\R^{d\times d}))^{\ell+1}\times(H^\varpi(\R^d;\R^d))^{\ell+1},
    \end{equation}
    for $\N\ni s\ge\varpi$, and a family of open subsets
    \begin{equation}
        0\in\mathcal{V}^{\m{sub}}(\mathfrak{q})\subset H^{\varpi}(\R^d;\R^d)\times\mathcal{H}^{1+\varpi}(\R^d)
        \text{ for }
        \mathfrak{q}\in(0,1)\times[0,\infty)^2
    \end{equation}
    such that the following hold.
    \begin{enumerate}
        \item \emph{Existence and uniqueness:} For all $(\gam,\mu,\sig,\tcb{\tau_i}_{i=0}^\ell,\tcb{\varphi_i}_{i=0}^\ell)\in\mathcal{U}^{\m{sub}}_\varpi$ there exists a unique $(v,\eta)\in\mathcal{V}^{\m{sub}}(\gam,\mu,\sig)$ such that system~\eqref{traveling wave formulation of the equation} is classically satisfied  with wave speed $\gam\in(0,1)$, viscosity $\mu\in[0,\infty)$, surface tension $\sig\in[0,\infty)$, and data $\mathcal{F}$ determined from $\eta$ and $(\tcb{\tau_i}_{i=0}^\ell,\tcb{\varphi_i}_{i=0}^\ell)$ via the subsonic case of~\eqref{conditions on the forcing}.
        \item \emph{Regularity promotion:} If $\N\ni s\ge\varpi$ and 
        \begin{equation}
            (\gam,\mu,\sig,\tcb{\tau_i}_{i=0}^\ell,\tcb{\varphi_i}_{i=0}^\ell)\in\mathcal{U}^{\m{sub}}_s\cap\tp{\R^3\times(H^s(\R^d;\R^{d\times d}))^{\ell+1}\times(H^s(\R^d;\R^d))^{\ell+1}}
        \end{equation}
        then the corresponding solution $(v,\eta)\in\mathcal{V}^{\m{sub}}(\gam,\mu,\sig)$ produced by the previous item obeys the inclusions
        \begin{equation}
            \tbr{\mu\grad}^2 v\in H^{s}(\R^d;\R^d),
            \text{ and }
            \tbr{\sig\grad}^2\eta\in\mathcal{H}^{1+s}(\R^d).
        \end{equation}
        \item \emph{Continuous dependence:} For any $\N\ni s\ge\varpi$ the three map maps
        \begin{multline}
            \mathcal{U}^{\m{sub}}_s\cap\tp{\R^3\times(H^s(\R^d;\R^{d\times d}))^{\ell+1}\times(H^s(\R^d;\R^d))^{\ell+1}}\ni(\gam,\mu,\sig,\tcb{\tau_i}_{i=0}^\ell,\tcb{\varphi_i}_{i=0}^\ell)\\\mapsto\begin{cases}
                (v,\eta)\in H^s(\R^d;\R^d)\times\mathcal{H}^{1+s}(\R^d)\\
                (\mu\grad v,\mu^2\grad^2v)\in H^s(\R^d;\R^{d^2})\times H^s(\R^d;\R^{d^3})\\
                (\sig\grad\eta,\sig^2\grad^2\eta)\in H^{1+s}(\R^d;\R^d)\times H^{1+s}(\R^d;\R^{d^2})
            \end{cases}
        \end{multline}
        are all continuous.
    \end{enumerate}
\end{customthm}

Theorems~\ref{main theorem no. 1}, \ref{main theorem no. 1.5}, \ref{main theorem no. 2}, and~\ref{main theorem no. 3} package several different properties of solutions to~\eqref{traveling wave formulation of the equation} in a rather technical form.  We now pause to elaborate on their content with a series of remarks. A high level summary of our results can be stated as follows. Small amplitude solitary traveling wave solutions to the shallow water equations  are a generic phenomenon for all nonnegative values viscosity and surface tension, and all positive wave speeds. These traveling wave solutions depend continuously on the data as well as the physical parameters (wave speed, viscosity, surface tension), even in the vanishing limits of the latter two parameters. Finally,  solutions to the traveling shallow water equations~\eqref{traveling wave formulation of the equation} have different  quantitative regularity depending on whether or not the wave speed is subsonic or supersonic.

Our main theorems also lead to some corollaries worth mentioning. If we fix the physical parameters $(\gam,\mu,\sig)\in\R^+\times[0,\infty)^2$, then we are assured of the existence of a non-empty open set of forcing data on which we have a well-posedness theory. See Corollary~\ref{coro on fixed physical parameters} for a precise formulation. On the other hand, our main theorems also show that for fixed forcing data we have a continuous two parameter family of solutions depending on $\mu$ and $\sig$; in particular, we can send $\mu\to0$, which is the inviscid limit, or $\sig\to0$, which is the vanishing surface tension limit. See Corollary~\ref{coro on limits} for a formal statement.  In thinking about this limit it is perhaps useful to recall that $\mu$ and $\sig$ are non-dimensional parameters related to gravity, the slip parameter, the  equilibrium height, the dimensional viscosity, and the dimensional surface tension via~\eqref{unitless quantities}.

We emphasize that our analysis of~\eqref{traveling wave formulation of the equation} relies on the wave speed $\gam$ being strictly positive. This is because it allows us to exploit the useful supercritical embedding and algebraic properties of the anisotropic Sobolev spaces $\mathcal{H}^s(\R^2)$ and thus close estimates in $L^2$-based Sobolev spaces. The fact that the free surface unknown is forced to belong to the anisotropic Sobolev spaces~\eqref{some definition of the anisobros} rather than a standard Sobolev space can be seen via an inspection of the linear-in-$\eta$ terms in~\eqref{traveling wave formulation of the equation}. The lowest order differential operators appearing in these linear terms are first order, namely $\gam\pd_1\eta$ in the first equation and $\grad\eta$ in the second equation. As it turns out, we shall view the first equation of~\eqref{traveling wave formulation of the equation} as an identity in $(\dot{H}^{-1}\cap H^{1+s})(\R^d)$ whereas the second equation is an identity in $H^s(\R^d;\R^d)$. Hence the low Fourier mode behavior of $\eta$ is determined via the inclusions $\grad\eta\in L^2(\R^d;\R^d)$ and $\gam\pd_1\eta\in\dot{H}^{-1}(\R^2)$. This is exactly what is captured for $|\xi|\le 1$ in the integral of definition~\eqref{some norm on the anisobros}.  While we remain optimistic that similar results  for the fully stationary ($\gam=0$) variant of~\eqref{traveling wave formulation of the equation} can be proved in the future, such results will not be possible within this paper's functional framework.  Our recent work on the solitary stationary wave problem for the three dimensional free boundary Navier-Stokes equations~\cite{stevensontice2023wellposedness} suggests that a framework built from $L^p(\R^d)$-based Sobolev spaces for $1<p<2$ might serve as an appropriate replacement for the  stationary shallow water problem when $d=2$.

Next, we aim to  overview both the main difficulties in proving Theorems~\ref{main theorem no. 1}, \ref{main theorem no. 1.5}, \ref{main theorem no. 2}, and~\ref{main theorem no. 3} and our strategies for overcoming them.  The system~\eqref{traveling wave formulation of the equation} is in general quasilinear, does not enjoy a variational formulation, and is posed in $\R^d$, which is non-compact. Consequently, Fredholm, variational, or compactness techniques are unavailable. On the other hand, we have the expectation of a robust theory for the family of linearizations of these equations, which suggests that the production of solutions ought to follow some perturbative argument based on linear theory. This leads us to a strategy modeled on the implicit and inverse function theorems. We begin our discussion of this strategy by examining the linearization of~\eqref{traveling wave formulation of the equation} around vanishing data and the trivial solution with a fixed choice of parameters $\gam\in\R^+$, $\mu,\sig\in[0,\infty)$:
\begin{equation}\label{linearization at the trivial solution introduction sanitized}
    \begin{cases}
        \grad\cdot v-\gam\pd_1\eta=h,\\
        -\gam\pd_1v+v-\mu^2\tp{\Delta v+3\grad\grad\cdot v}+(I-\sig^2\Delta)\grad\eta=f.
    \end{cases}
\end{equation}
In~\eqref{linearization at the trivial solution introduction sanitized}, the unknowns are $v:\R^d\to\R^d$ and $\eta:\R^d\to\R$, while the data are $f:\R^d\to\R^d$ and $h:\R^d\to\R$.

The equations~\eqref{linearization at the trivial solution introduction sanitized} constitute a time-independent linear system of differential equations, so it is natural to investigate its ellipticity, in the sense of Agmon-Douglis-Nirenberg~\cite{adn2}, as a function of the spatial dimension and the parameters  $\mu$ and $\sigma$.   When~\eqref{linearization at the trivial solution introduction sanitized} is ADN elliptic, the parent system~\eqref{traveling wave formulation of the equation} in the small data regime is entirely dominated by the linear part~\eqref{linearization at the trivial solution introduction sanitized}, provided that the nonlinear part of~\eqref{traveling wave formulation of the equation} has order not exceeding that of the linear part.  On the other hand, when~\eqref{linearization at the trivial solution introduction sanitized} fails to be ADN elliptic, the trivial linearization~\eqref{linearization at the trivial solution introduction sanitized} cannot tell the full story, and derivative loss is expected.

In order to  determine when~\eqref{linearization at the trivial solution introduction sanitized} is ADN elliptic,  we rewrite the problem as
\begin{equation}\label{trivial linearization operator}
\bf{P}
\begin{pmatrix}
   v \\ \eta 
\end{pmatrix}
=
\begin{pmatrix}
   f \\ h 
\end{pmatrix}
\text{ for }    \bf{P}=
\begin{pmatrix}
 (1-\gamma \partial_1)I -\mu^2 (\Delta I + 3 \nabla \otimes \nabla) & (1-\sigma^2 \Delta) \nabla \\
\nabla^{\m{t}} & -\gamma \partial_1
\end{pmatrix},
\end{equation}
and we compute the scalar differential operator
\begin{equation}\label{this is the determinant of the operator P}
    \det\bf{P}
\\=(1-\gamma\partial_1-\mu^2 \Delta)^{d-1}(\gamma^2\partial_1^2-\Delta+\sigma^2 \Delta^2-\gamma \partial_1 (1 - 4\mu^2 \Delta)).
\end{equation}
With these in hand, we can readily determine when the system \eqref{linearization at the trivial solution introduction sanitized} is ADN elliptic.  We refer to the table in Figure~\ref{ADN_table} for the full details, but in brief: if $d=1$ then the system is elliptic unless $\mu=\sigma =0$ and $\gam=1$ (the exact sonic speed), and if $d\ge 2$ then the system is elliptic only when both $\mu >0$ and $\sigma >0$.

\begin{figure}[!h]
    \centering
\begin{tikzpicture}
\node[scale=0.925] at (0,0){
$
\begin{array}{| c || c | c |  c | c |}
\hline
& \mu >0,\;\sigma >0 & \mu >0,\;\sigma =0 & \mu =0,\;\sigma >0 & \mu =0,\;\sigma=0\\
\hline\hline
%%%%%%%%%%%%%%%%%
R & 2d+2  & 2d+1  & d+3  & d+1 
 \\ \hline
%%%%%%%%%%%%%%%%%
r & 2d+2 & 2d+1 & d+3 & \begin{cases}
    d+1&\text{if }d\ge2\\
    2&\text{if }d=1,\;\gam\neq1\\
    1&\text{if }d=1,\;\gam=1
\end{cases}
 \\ \hline
%%%%%%%%%%%%%%%%%
(\det \bf{P})_{\star} &\sigma^2 \Delta^2 (-\mu^2 \Delta)^{d-1} & -4\gamma \partial_1 (-\mu^2 \Delta)^{d} & (-\gamma \partial_1)^{d-1} \sigma^2 \Delta^2 &\begin{cases}
    (-\gamma \partial_1)^{d-1} (\gamma^2 \partial_1^2 -\Delta)&\text{if }d\ge 2\\
    \tp{\gam^2-1}\pd_1^2&\text{if }d=1,\;\gam\neq1\\
    -\pd_1&\text{if }d=1,\;\gam=1
\end{cases} 
 \\ \hline
%%%%%%%%%%%%%%%%%
\textnormal{Ellipic, } d = 1 & \textnormal{Yes}  & \textnormal{Yes} & \textnormal{Yes} & \textnormal{Yes}\lra\gam\neq1 
 \\ \hline
%%%%%%%%%%%%%%%%%
\textnormal{Ellipic, } d\ge 2 & \textnormal{Yes} & \textnormal{No} & \textnormal{No} & \textnormal{No} 
 \\ \hline
%%%%%%%%%%%%%%%%%
\end{array}
$
};
\end{tikzpicture}
    \caption{This table summarizes the information needed to determine whether or not the system \eqref{linearization at the trivial solution introduction sanitized} is ADN elliptic.  For this we use the equivalent formulation of Volevi\v{c}~\cite{volevic_63} (see also chapter 9 of \cite{wloka_95}).  The number $R \in \N$ is the maximal order of $\prod_{i=d+1}^n \bf{P}_{i,\uppi(i) }$ over all permutations $\uppi$ of $\{1,\dotsc,d+1\}$.  The number $r \in \N$ is the order of the scalar operator $\det \bf{P}$, and $(\det\bf{P})_{\star}$ denotes the principal part of $\det\bf{P}$.  Volevi\v{c} showed that $\bf{P}$ is ADN elliptic if and only if $r=R$ and $\det\bf{P}$ is elliptic. }      
    \label{ADN_table}
\end{figure}

Our strategy for developing a well-posedness theory for~\eqref{traveling wave formulation of the equation} differs depending on whether or not~\eqref{linearization at the trivial solution introduction sanitized} is ADN elliptic. When the former condition holds, a standard inverse function theorem argument  - taking also into account the low mode behavior of $\bf{P}$ and the anisotropic Sobolev spaces~\eqref{some definition of the anisobros} - is sufficient to close. The full details of this execution are carried out in Section~\ref{section on analysis with both visc and surf}; this then leads us to Theorems~\ref{main theorem no. 1} and~\ref{main theorem no. 1.5}.

On the other hand, if the trivial linearization~\eqref{linearization at the trivial solution introduction sanitized} is not ADN elliptic, then the previous inverse function theorem strategy fails. The lack of ellipticity is connected to the emergence of derivative loss in our linear estimates; that is, we recover less on $v$ or $\eta$ than is spent by the forward differential operator. When presented with derivative loss at the boundary of a parameter space in which a standard inverse function theorem is applicable in the interior, it is natural to attempt upgrading and hence using a Nash-Moser inverse to reach the reach the remaining cases. This is what we do in Sections~\ref{section on nonlinear analysis}, \ref{section on linear analysis}, and~\ref{section on conclusion}.

The reason we expect a Nash-Moser strategy to work at all when the ellipticity of~\eqref{linearization at the trivial solution introduction sanitized} fails can be seen by closer inspection of the symbol associated with the determinant~\eqref{this is the determinant of the operator P}. Regardless of the choice of $\tp{\gam,\mu,\sig}\in\R^+\times[0,\infty)^2$, we always have the following lower bound for $\xi\in\R^d$:
\begin{equation}
    |\mathscr{F}[\det\bf{P}](\xi)|\ge\bp{\pi^2|\xi|^2\bbr{\f{2\pi\gam\xi_1}{\tbr{4\pi\mu\xi}^2}}^{-1}\tbr{2\pi\sig\xi}^2+2\pi\gam|\xi_1|\tbr{4\pi\mu\xi}^2}\tbr{\mu\xi}^{2(d-1)}\ge\pi|\xi|^2\tbr{2\pi\gam\xi_1}^{-1}+2\pi\gam|\xi_1|.
\end{equation}
This tells us that not only is the linearized differential operator $\bf{P}$ of~\eqref{trivial linearization operator} invertible for non-zero frequencies, but $\det\bf{P}$ is `bounded from below' by an elliptic first order pseudodifferential operator. As a result, inversion of the trivial linearization will result in the recovery of some estimates on the solution, albeit without the correct the number of derivatives needed to close with a standard inverse function theorem argument.  We emphasize that the reason that $\det\bf{P}$ does not degenerate completely is because we are considering the shallow water system with frictional damping, which is responsible for the zeroth order term for $v$ in~\eqref{linearization at the trivial solution introduction sanitized}, and with gravity, which gives us the full gradient of $\eta$ in these equations. In the Nash-Moser framework, one needs estimates on solution operators to some family of linearizations indexed by an open neighborhood of smooth backgrounds. In order to close estimates for these nearby linearizations, the mere existence of the damping and gravity terms is not quite sufficient. We also need that the terms in the system of higher order than our estimate recovery level obey certain cancellations during the energy argument;  see, for example, Lemma~\ref{lemma on integration by parts} or Proposition~\ref{prop on first base case estimate} for more details.

There are numerous versions of the Nash-Moser inverse function theorem in the literature; the one we use in this paper, which is recorded in Appendix~\ref{subsection on a NMIFT}, was developed in an earlier paper of the authors~\cite{stevenson2023wellposedness} for use on the free boundary compressible Navier-Stokes traveling wave problem. The advantage of this version in the context of the shallow water system is that is suffices to do analysis and estimates in finite scales of Banach spaces and that the inverse function granted by the theorem is continuous between spaces obeying an optimal relative derivative counting. A pleasant corollary of this latter fact is that with Nash-Moser we can do more than merely handle the cases on the boundary of the $(\mu,\sig)$-parameter space; in fact, we are able to understand very well the limiting behavior  as one approaches the boundary from the interior of the parameter space.

Section~\ref{section on nonlinear analysis} is dedicated to verifying the nonlinear hypotheses of the Nash-Moser theorem and to identifying a minimal essential structure for the linear analysis. Not only do hard inverse function theorems require one to check that the forward mapping is continuously differentiable, but it is typically the case that more regularity is sought and always the case that additional structured estimates on the map and its derivatives are required. In this section of the document we associate to the system~\eqref{traveling wave formulation of the equation} two nonlinear mapping and then verify the requisite smoothness and structured estimates. Afterward, we examine certain partial derivatives of these maps and identify a bare bones structure responsible for the derivative loss and prove that what remains can be handled perturbatively.

Section~\ref{section on linear analysis} is devoted to the verification of the linear hypotheses of the Nash-Moser theorem. We are tasked with showing that the linearizations of the nonlinear mappings identified in the previous section are left and right invertible and obey certain structured estimates for an open neighborhood of background points.  Our strategy for achieving this is to first prove precise a priori estimates for the principal part of the linearizations  by carefully exploiting cancellations of the higher-order terms; we then obtain existence for the principal part by combining with the method of continuity.  Finally, we synthesize these results with the remainder estimates of the previous section to obtain the full linear hypotheses.

Section~\ref{section on conclusion} combines the previous work on the nonlinear and linear hypotheses to two invocations of the Nash-Moser theorem. These are initially phrased abstractly, so we then unpack the details into the PDE-centric results of Theorems~\ref{main theorem no. 2} and~\ref{main theorem no. 3}. 

The remainder of the paper consists of appendices. In Appendix~\ref{appendix on the anisobros}, we review important properties of the anisotropic Sobolev spaces. Appendix~\ref{tame structure abs} is dedicated to the abstraction of the structures involved in the version of Nash-Moser employed here. Appendix~\ref{subsection on a NMIFT} is the precise statement of this aforementioned theorem. Appendix~\ref{toolbox for smooth-tameness} then gives us a concrete tool box for checking structured estimates on the Sobolev-type nonlinearities encountered here. Finally Appendix~\ref{appendix on the shallow water system} contains the derivation of the forced-traveling wave formulation of the shallow water system and a power-dissipation calculation.

% _+__+_ -_+__+_ -_+__+_ -_+__+_ -_+__+_ -_+__+_ -_+__+_ -_+__+_ -_+__+_ -_+__+_ -_+__+_ -_+__+_ -_+__+_ -
\subsection{Conventions of notation}
% _+__+_ -_+__+_ -_+__+_ -_+__+_ -_+__+_ -_+__+_ -_+__+_ -_+__+_ -_+__+_ -_+__+_ -_+__+_ -_+__+_ -_+__+_ -

The set $\tcb{0,1,2,\dots}$ is denoted by $\N$ whereas $\N^+=\N\setminus\tcb{0}$. We denote $\R^+=(0,\infty)$. The notation $\al\lesssim\be$ means that there exists $C\in\R^+$, depending only on the parameters that are clear from context, for which $\al\le C\be$. To highlight the dependence of $C$ on one or more particular parameters $a,\dots, b$, we will occasionally write $\al\lesssim_{a,\dots,b}\be$. We also express that two quantities $\al$, $\be$ are equivalent, written $\al\asymp\be$ if both $\al\lesssim\be$ and $\be\lesssim\al$. Throughout this document, the implicit constants always depend at least on the physical dimension $d\in\N^+$ and the nonlinearity order $\ell\in\N$ (see~\eqref{special form of the superposition nonlinearities} and~\eqref{superposition approximation}). The bracket notation 
\begin{equation}\label{bracket notation}
    \tbr{x}=\sqrt{1+|x_1|^2+\dots+|x_\ell|^2} \text{ for } x\in\C^\ell
\end{equation}
is frequently used. The notation $A\Subset B$ means that the closure $\Bar{A}$ is compact and a subset of the interior $\m{int}B$. If $\tcb{X_i}_{i=1}^\upnu$ are normed spaces and $X$ is their product, endowed with any choice of product norm $\tnorm{\cdot}_X$, then we shall write
\begin{equation}
    \tnorm{x_1,\dots,x_\upnu}_{X}=\tnorm{(x_1,\dots,x_\upnu)}_{X}  \text{ for } (x_1,\dots,x_\upnu)\in X.
\end{equation}
The Fourier transform and its inverse on the space of tempered distributions $\mathscr{S}^\ast(\R^d;)$, which are normalized to be unitary on $L^2(\R^d;\C)$, are denoted by $\mathscr{F}$, $\mathscr{F}^{-1}$, respectively. The $\R^d$-gradient is denoted by $\grad=(\pd_1,\dots,\pd_d)$ and the divergence of a vector field $X:\R^d\to\R^d$ is denoted by $\grad\cdot X=\sum_{j=1}^d\pd_j(X\cdot e_j)$. By $|\grad|$ and $\tbr{\grad}$, we mean the operator given by Fourier multiplication with symbols $\xi\mapsto2\pi|\xi|$ and $\xi\mapsto\tbr{2\pi\xi}$, respectively.

% _+__+_ -_+__+_ -_+__+_ -_+__+_ -_+__+_ -_+__+_ -_+__+_ -_+__+_ -_+__+_ -_+__+_ -_+__+_ -_+__+_ -_+__+_ -
\section{Analysis of the elliptic cases}\label{section on analysis with both visc and surf}
% _+__+_ -_+__+_ -_+__+_ -_+__+_ -_+__+_ -_+__+_ -_+__+_ -_+__+_ -_+__+_ -_+__+_ -_+__+_ -_+__+_ -_+__+_ -

In this section we record the linear and nonlinear analysis of system~\eqref{traveling wave formulation of the equation} in the regimes for which there is no derivative loss in the sense that the linear problem recovers all of the derivatives spent by the full nonlinear equations. For our equations, this is equivalent to saying that the trivial linearization~\eqref{linearization at the trivial solution introduction sanitized} is ADN elliptic. As it turns out, we do not have derivative loss under the following two (not mutually exclusive) assumptions. First, $\mu>0$ and $\sigma >0$. Second, the spatial dimension, $d$, is equal to $1$ and the positive wave speed, $\gam$, is not equal to $1$. We find that in either of these cases, the problem of local well-posedness is felled by a relatively simple application of the implicit function theorem. 

We glimpse how this initial strategy fails to handle the boundary cases of the above assumptions (e.g. the limits $\mu\to0$ or $\sig\to0$ and the corresponding limiting cases for $d\ge 2$ or $d=1$ and $\gam=1$), due to the emergent derivative loss's incompatibility with standard inverse function theorems. In what follows, we will instead attack~\eqref{traveling wave formulation of the equation} with a Nash-Moser inverse function theorem to continuously recover well-posedness in all of the missing cases.

The plan of this section is as follows. In Section~\ref{section on the trivialization} we study the linearization of~\eqref{traveling wave formulation of the equation} at zero when viscosity and surface tension are both nonzero. We find the right spaces between which the linear operator acts as an isomorphism. Next, in Section~\ref{section on basic nonlinear theory}, we combine these isomorphism results with some basic nonlinear mapping properties and the inverse function theorem. Finally, in Section~\ref{subsection on the one dimensional theory}, we have a special look at the case of one spatial dimension in~\eqref{traveling wave formulation of the equation}. We find that whenever $\gam\neq 1$, we can perform similar inverse function theorem analysis but this time gleaning more refined estimates.

% _+__+_ -_+__+_ -_+__+_ -_+__+_ -_+__+_ -_+__+_ -_+__+_ -_+__+_ -_+__+_ -_+__+_ -_+__+_ -_+__+_ -_+__+_ -
\subsection{Theory of the trivial linearization}\label{section on the trivialization}
% _+__+_ -_+__+_ -_+__+_ -_+__+_ -_+__+_ -_+__+_ -_+__+_ -_+__+_ -_+__+_ -_+__+_ -_+__+_ -_+__+_ -_+__+_ -

The linearization around zero of the traveling wave formulation for the shallow water system~\eqref{traveling wave formulation of the equation} is the following linear system.
\begin{equation}\label{linearization at the trivial solution}
    \begin{cases}
        \grad\cdot v-\gam\pd_1\eta=h,\\
        -\gam\pd_1v+v-\mu^2\grad\cdot\S v+(I-\sig^2\Delta)\grad\eta=f.
    \end{cases}
\end{equation}
Here the data are a given scalar field $h:\R^d\to\R$ and vector field $f:\R^d\to\R^d$, while the unknowns are the vector field $v:\R^d\to\R^d$ and the scalar $\eta:\R^d\to\R$.

Our first result derives a system of equations equivalent to~\eqref{linearization at the trivial solution}, but the unknowns are decoupled from one another. In what follows we shall use $\P$ to denote the orthogonal projection in $L^2$ onto the subspace of solenoidal vector fields. $\P$ corresponds to the Fourier multiplier $\xi\mapsto I-|\xi|^{-2}\xi\otimes\xi$. We shall also let $\mathcal{R}=\tp{\mathcal{R}_1,\dots,\mathcal{R}_d}$ denote the vector of Riesz transforms; $\mathcal{R}$ is given by the Fourier multiplier $\xi\mapsto \ii|\xi|^{-1}\xi$ so that $\mathcal{R}=|\grad|^{-1}\grad$ and $\P=I+\mathcal{R}\otimes\mathcal{R}$.

\begin{prop}[Decoupling reformulation]\label{prop on decoupling reformulation}
    The following are equivalent for $s\in\N$, $f\in H^s(\R^d;\R^d)$, $h\in (\dot{H}^{-1}\cap H^{1+s})(\R^d)$, $v\in H^{2+s}(\R^d;\R^d)$, and $\eta\in\mathcal{H}^{3+s}(\R^d)$.
    \begin{enumerate}
        \item System~\eqref{linearization at the trivial solution} is satisfied by $v$ and $\eta$ with the data $f$ and $h$.
        \item It holds that
        \begin{equation}\label{decoupled system of equations}
            \begin{cases}
                (I-\gam\pd_1-\mu^2\Delta)\P v=\P f,\\
                (I-\P)v=H-\gam\mathcal{R}\mathcal{R}_1\eta,\\
                -\gam(I-4\mu^2\Delta)\mathcal{R}\mathcal{R}_1\eta+\tp{I+\gam^2\mathcal{R}_1^2-\sig^2\Delta}\grad\eta=(I-\P)f-(I-\gam\pd_1-4\mu^2\Delta)H,
            \end{cases}
        \end{equation}
        where $H=\grad\Delta^{-1}h\in(I-\P)H^{2+s}(\R^d;\R^d)$. Moreover, we have norm equivalence $\norm{H}_{H^{2+s}} \asymp \norm{h}_{\dot{H}^{-1} \cap H^{1+s}}$.
    \end{enumerate}
\end{prop}
\begin{proof}
    The first equation in~\eqref{linearization at the trivial solution} and the second equation in~\eqref{decoupled system of equations} are manifestly equivalent.  The second equation of~\eqref{linearization at the trivial solution} is equivalent to the pair of equations obtained by applying the projectors $\P$ and $I-\P$; the $\P$ projection is the first equation of~\eqref{decoupled system of equations}, while the third equation of~\eqref{decoupled system of equations} is equivalent to the $I-\P$ projection via a substitution of the equivalent forms of the first equation in~\eqref{linearization at the trivial solution}.  The inclusion $H = \grad \Delta^{-1} h \in (I-\P)H^{s+2}$ and the corresponding norm equivalence follow by decomposing $h$ into its high and low frequency parts and using the $\dot{H}^{-1}$ and $H^{1+s}$ inclusions on each part, respectively.
\end{proof}

The benefit of the decoupling reformulation~\eqref{decoupled system of equations} is that the third equation allows us to solve for the $\eta$ unknown in terms of the data alone. The following proposition does exactly this.

\begin{prop}[Solving for the linear free surface]\label{prop on solving for the linear free surface}
    Given $s \in \N$ and $\varphi\in(I-\P)H^s(\R^d;\R^d)$, there exists a unique $\eta\in\mathcal{H}^{3+s}(\R^d)$ such that
    \begin{equation}\label{the easy eta equation}
        -\gam(I-4\mu^2\Delta)\mathcal{R}\mathcal{R}_1\eta+\tp{I+\gam^2\mathcal{R}_1^2-\sig^2\Delta}\grad\eta=\varphi.
    \end{equation}
\end{prop}
\begin{proof}

    We propose defining $\eta$ via the equation
    \begin{equation}\label{definition of the free surface}
        \eta=\ssb{\gam(I-4\mu^2\Delta)\mathcal{R}_1 +(I+\gam^2\mathcal{R}_1^2-\sig^2\Delta)\mathcal{R}\cdot\grad}^{-1}\mathcal{R}\cdot\varphi,
    \end{equation}
    but to justify this we must first establish the invertibility of the operator in brackets.  The symbol of said operator is  $\chi : \R^d \to \C$ given by
    \begin{equation}
        \chi(\xi)=\ii\gam\tp{1+16\pi^2\mu^2|\xi|^2}|\xi|^{-1}\xi_1-2\pi\tp{1-\gam^2|\xi|^{-2}\xi_1^2+4\pi^2\sig^2|\xi|^2}|\xi|,
    \end{equation}
    which obeys the  inequalities
    \begin{equation}
        \gam(1+16\pi^2\mu^2|\xi|^2)|\xi|^{-1}|\xi_1|\le|\chi(\xi)| \text{ and }
        2\pi(1+4\pi^2\sig^2 |\xi|^2)|\xi|\le|\chi(\xi)-2\pi\gam^2|\xi|^{-2}\xi_1^2|\xi||.
    \end{equation}
    Hence,
    \begin{equation}\label{the equation hoist}
        \gam(1+16\pi^2\mu^2|\xi|^2)|\xi|^{-1}|\xi_1|+2\pi(1+4\pi^2\sig^2|\xi|^2)|\xi|\le\bp{2+\f{2\pi \gam|\xi_1|}{1+16\pi^2\mu^2|\xi|^2}}|\chi(\xi)|
    \end{equation}
    and so, for a constant $C\in\R^+$ depending on $\gam$, $\mu$, and $\sigma$, we have the equivalence
    \begin{equation}
    C^{-1}\tbr{\xi}^2\tp{|\xi_1||\xi|^{-1}+|\xi|}\le|\chi(\xi)|\le C\tbr{\xi}^2\tp{|\xi_1||\xi|^{-1}+|\xi|}.
    \end{equation}
    From this estimate it's a simple matter to check that the operator $\chi(D) : \mathcal{H}^{s+3}(\R^d) \to H^s(\R^d)$ is an isomorphism, and so our proposed definition of $\eta$ in \eqref{definition of the free surface} is well-defined.  Moreover, we have that  $\eta\in\mathcal{H}^{3+s}(\R^d)$ with the estimate $\tnorm{\eta}_{\mathcal{H}^{3+s}}\lesssim\tnorm{\mathcal{R}\cdot\varphi}_{H^s}$.     To complete the existence proof, we note that $\varphi\in(I-\P)H^s$ results in $\mathcal{R}\mathcal{R}\cdot\varphi=-\varphi$, and hence~\eqref{definition of the free surface} and the identity $\mathcal{R}\mathcal{R}\cdot\grad=-\grad$ imply that $\eta$ solves~\eqref{the easy eta equation}. Uniqueness follows by noting that if \eqref{the easy eta equation} holds with $\varphi =0$, then we may apply $\mathcal{R}$ to see that $\chi(D) \eta =0$.
\end{proof}

We are now ready to study the linear isomorphism associated with~\eqref{linearization at the trivial solution}.

\begin{thm}[Linear well posedness in the case of positive surface tension and viscosity]\label{thm on the easy case of linear well-posedness}
    Given any choice of $s\in\N$ and $\gam,\mu,\sig\in\R^+$, the mapping
    \begin{equation}
        L:H^{2+s}(\R^d;\R^d) \times \mathcal{H}^{3+s}(\R^d) \to\tp{\dot{H}^{-1}\cap H^{1+s}}(\R^d) \times H^s(\R^d;\R^d)
    \end{equation}
    defined via
    \begin{equation}
        L(v,\eta)=(\grad\cdot v-\gam\pd_1\eta,-\gam\pd_1v+v-\mu^2\grad\cdot\S v+(I-\sig^2\Delta)\grad\eta)
    \end{equation}
    is a bounded linear isomorphism.
\end{thm}
\begin{proof}
    That $L$ is well-defined and bounded is clear from the characterizations of the anisotropic Sobolev spaces given in Proposition \ref{proposition on spatial characterization of anisobros}.  It remains to establish that $L$ is a bijection.
    
    If $L(v,\eta)=0$, then $\eta$ satisfies~\eqref{the easy eta equation} with right hand side identically zero. Hence, by the uniqueness assertion of Proposition~\ref{prop on solving for the linear free surface}, we find that $\eta=0$.  In turn, the first and second equations in~\eqref{decoupled system of equations} imply that $\P v=(I-\P)v=0$.  Hence, $(v,\eta)=0$, which demonstrates the injectivity of $L$.

    Next, we show that $L$ is a surjection. Let $h\in (\dot{H}^{-1}\cap H^{1+s})(\R^d)$ and $f\in H^s(\R^d;\R^d)$. We set $H=\grad\Delta^{-1}h\in(I-\P)H^{2+s}(\R^d;\R^d)$ as in Proposition~\ref{prop on decoupling reformulation} and then define
    \begin{equation}
        \varphi=(I-\P)f-(I-\gam\pd_1-4\mu^2\Delta)H\in(I-\P)H^s(\R^d;\R^d).
    \end{equation}
    Then we can construct $\eta\in\mathcal{H}^{3+s}(\R^d)$ solving~\eqref{the easy eta equation} via an application of Proposition~\ref{prop on solving for the linear free surface}. In turn, we construct $v\in H^{2+s}(\R^d;\R^d)$ via the decomposition $v=\P v+(I-\P)v$, where the projections are obtained by solving  the equations
    \begin{equation}
        \begin{cases}
            (I-\gam\pd_1-\mu^2\Delta)\P v=\P f,\\
            (I-\P)v=H-\gam\mathcal{R}\mathcal{R}_1\eta.
        \end{cases}
    \end{equation}
    The pair $(v,\eta)\in H^{2+s}(\R^d;\R^d) \times \mathcal{H}^{3+s}(\R^d)$ satisfy the second item of Proposition~\ref{prop on decoupling reformulation} with the data $f$ and $h$. By invoking this equivalence, we find that $L(v,\eta)=(h,f)$.
\end{proof}

% _+__+_ -_+__+_ -_+__+_ -_+__+_ -_+__+_ -_+__+_ -_+__+_ -_+__+_ -_+__+_ -_+__+_ -_+__+_ -_+__+_ -_+__+_ -
\subsection{A first pass nonlinear theory}\label{section on basic nonlinear theory}
% _+__+_ -_+__+_ -_+__+_ -_+__+_ -_+__+_ -_+__+_ -_+__+_ -_+__+_ -_+__+_ -_+__+_ -_+__+_ -_+__+_ -_+__+_ -

We now port the previous subsection's linear theory to a nonlinear result via an application of the implicit function theorem. At the end, we shall have Theorem~\ref{main theorem no. 1}. We begin by checking that the requisite mapping properties of the nonlinear differential operator are satisfied.

\begin{prop}[Smoothness of the forward mapping]\label{prop on smoothness of the forward mapping}
    Given any $\N\ni s>d/2$, the function
    \begin{multline}
        \Psi:[(\R^+)^3\times (H^s(\R^d;\R^{d\times d}))^{\ell+1}\times\tp{H^s(\R^d;\R^d)}^{\ell+1}]\times[H^{2+s}(\R^d;\R^d)\times\mathcal{H}^{3+s}(\R^d)] \\
        \to(\dot{H}^{-1}\cap H^{1+s})(\R^d)\times H^s(\R^d;\R^d)
    \end{multline}
    defined via
    \begin{multline}
        \Psi((\gam,\mu,\sig,\tcb{\tau_i}_{i=0}^\ell,\tcb{\varphi_i}_{i=0}^\ell),(v,\eta))\\=\bpm\grad\cdot((1+\eta)(v-\gam e_1))\\(1+\eta)(v-\gam e_1)\cdot\grad v+v-\mu^2\grad\cdot((1+\eta)\S v)+(1+\eta)(1-\sig^2\Delta)\grad\eta+\sum_{i=0}^\ell \eta^i\tp{\tau_i\grad\eta+\varphi_i}\epm
    \end{multline}
    is smooth.
\end{prop}
\begin{proof}
Corollary~\ref{coro on more algebra properties} and Proposition~\ref{corollary on tame estimates on simple multipliers} show that all of the products involved in the definition of $\Psi$ are well-defined.  $\Psi$ is then a polynomial mapping comprised of sums of bounded multilinear maps, and is thus smooth.
\end{proof}

The next result is our first version of positive viscosity and positive surface tension well-posedness. This theorem only works perturbatively from a fixed parameter tuple of wave speed, viscosity, and surface tension.

\begin{thm}[Well-posedness, I]\label{thm on well-posedness, I}
    Let $\N\ni s>d/2$. For each $\bf{p}=(\Bar{\gam},\Bar{\mu},\Bar{\sig})\in(\R^+)^3$ there exists $\ep_{s,\bf{p}}$, $\ep_{s,\bf{p}}'\in\R^+$ and a unique mapping
    \begin{equation}
        \upiota_{\bf{p}}:B((\bf{p},0),\ep_{s,\bf{p}})\subset(\R^+)^3\times (H^s(\R^d;\R^{d\times d}))^{\ell+1} \times(H^s(\R^d;\R^d))^{\ell+1} \to B(0,\ep_{s,\bf{p}}')\subset H^{2+s}(\R^d;\R^d)\times\mathcal{H}^{3+s}(\R^d)
    \end{equation}
    such that for all $(\gam,\mu,\sig,\tcb{\tau_i}_{i=0}^\ell, \tcb{\varphi_{i}}_{i=0}^\ell)\in B((\bf{p},0),\ep_{s,\bf{p}})$ if we set
    $(v,\eta)=\upiota_{\bf{p}}(\gam,\mu,\sig,\tcb{\tau_i}_{i=0}^\ell,\tcb{\varphi_i}_{i=0}^\ell)\in B(0,\ep_{s,\bf{p}}')$, then $\Psi(((\gam,\mu,\sig,\tcb{\tau_i}_{i=0}^\ell,\tcb{\varphi_i}_{i=0}^\ell)),(v,\eta))=0$. Moreover, the map $\upiota_\bf{p}$ is smooth.
\end{thm}
\begin{proof}
    Thanks to Theorem~\ref{thm on the easy case of linear well-posedness} and Proposition~\ref{prop on smoothness of the forward mapping} the hypotheses of the implicit function theorem are satisfied when applied to the map $\Psi$ at the point $(\bf{p},0)$. We are thus granted a smooth implicit function $\upiota_{\bf{p}}$ that obeys the stated properties.
\end{proof}

We augment Theorem~\ref{thm on well-posedness, I} with a simple gluing argument that allows us to consider simultaneously all tuples of strictly positive wave speed, viscosity, and surface tension.

\begin{thm}[Well-posedness, II]\label{thm on well-posedness II}
    Let $\N\ni s>d/2$. There exists an open set
    \begin{equation}
        (\R^+)^3\times\tcb{0}^{\ell+1} \times\tcb{0}^{\ell+1} \subset\mathscr{U}_s\subset(\R^+)^3 \times (H^s(\R^d;\R^{d\times d}))^{\ell+1} \times ( H^s(\R^d;\R^d))^{\ell+1}
    \end{equation}
    and a smooth mapping
    \begin{equation}
        \upiota:\mathscr{U}_s\to H^{2+s}(\R^d;\R^d)\times\mathcal{H}^{3+s}(\R^d)
    \end{equation}
    such that for all $\bf{U}=(\gam,\mu,\sig,\tcb{\tau_i}_{i=0}^\ell,\tcb{\varphi_i}_{i=0}^\ell)\in\mathscr{U}_s$ we have that $\upiota(\bf{U})\in B(0,\ep_{s,\bf{p}}')$ (with  $\bf{p}=(\gam,\mu,\sig)$ and $\ep_{s,\bf{p}}'$ as defined in Theorem~\ref{thm on well-posedness, I}) is the unique solution to $\Psi(\bf{U},\upiota(\bf{U}))=0$.
\end{thm}
\begin{proof}
    We set
    \begin{equation}
        \mathscr{U}_s=\bigcup_{\bf{p}\in(\R^+)^4}B((\bf{p},0),\ep_{s,\bf{p}}).
    \end{equation}
    and define $\upiota=\upiota_{\bf{p}}$ on the set $B((\bf{p},0),\ep_{s,\bf{p}})$. That this is well-defined and smooth follows from Theorem~\ref{thm on well-posedness, I}.
\end{proof}

% _+__+_ -_+__+_ -_+__+_ -_+__+_ -_+__+_ -_+__+_ -_+__+_ -_+__+_ -_+__+_ -_+__+_ -_+__+_ -_+__+_ -_+__+_ -
\subsection{One dimensional refinements}\label{subsection on the one dimensional theory}
% _+__+_ -_+__+_ -_+__+_ -_+__+_ -_+__+_ -_+__+_ -_+__+_ -_+__+_ -_+__+_ -_+__+_ -_+__+_ -_+__+_ -_+__+_ -

In this subsection we restrict our attention to system~\eqref{traveling wave formulation of the equation} in one spatial dimension. The reason for doing so is that, because the derivative loss is less severe in one dimension, we are able to obtain slightly more refined results with a standard inverse function theorem than what we shall later obtain with a Nash-Moser argument. In this case, we may reformulate the system into one single equation involving the free surface unknown alone. To see how, we note that the first equation in~\eqref{traveling wave formulation of the equation} tells us that $\gam\eta'=\tp{(1+\eta)v}'$. Since we are interested in the class of solitary wave solutions, we require that $\eta$ and $v$ both vanish at infinity. We also shall be concerned with solutions for which $1+\eta\ge0$. Thus, in this class, we can integrate this identity and obtain the relation
\begin{equation}
v=\gam\f{\eta}{1+\eta}.
\end{equation}
We can then substitute this into the second equation in~\eqref{traveling wave formulation of the equation} and obtain the following equation involving only the free surface:
\begin{equation}\label{reduced eta equation}
    -\gam^2\bp{\f{\eta}{1+\eta}}'+\gam\f{\eta}{1+\eta}-4\gam\mu^2\bp{(1+\eta)\bp{\f{\eta}{1+\eta}}'}'+(1+\eta)\tp{\eta-\sig^2\eta''}'+\mathcal{F}=0.
\end{equation}

For the rest of this section we shall consider the wave speed cases of~\eqref{reduced eta equation} which do not exhibit derivative loss in the limit of $\mu$ and $\sigma$ tend to zero. As it turns out, this means that we shall only need to avoid the precisely sonic case of $\gam=1$. Thus, we shall take $\mu,\sig\in\R$ and $\gam\in\R^+\setminus\tcb{1}$ for the remainder of this subsection. This does not mean that we are avoiding the precisely sonic case entirely, indeed in the subsequent sections of the paper where we perform more delicate analysis working towards an application of a Nash-Moser inverse function theorem, the cases absent in this section will be covered.

Our first result studies the mapping properties of the linearization of~\eqref{reduced eta equation} at the trivial solution.

 \begin{prop}[One dimensional linear analysis, estimates]\label{prop on one dimensional linear analysis}
    Given $s\in\N$, $R\in\R^+$, and $\gam\in\R^+\setminus\tcb{1}$, there exists a constant $C\in\R^+$ such that whenever $\mu,\sig\in[-R,R]$, $\psi\in H^s(\R)$, and $\eta\in H^{s}(\R)$, satisfy the equation
    \begin{equation}\label{one dimensional linearization}
        \tp{\gam\tp{1-4\mu^2\pd^2}+\tp{(1-\gam^2)-\sig^2\pd^2}\pd}\eta=\psi,
    \end{equation}
    we have the estimates
    \begin{equation}\label{estimates for the one dimensional linearization}
    \begin{cases}
        \tnorm{\eta}_{H^{1+s}}+\mu^2\tnorm{\eta}_{H^{2+s}}+\sig^2\tnorm{\eta}_{H^{3+s}}\le C\tnorm{f}_{H^s}&\text{if }\gam<1,\\
        \tnorm{\eta}_{H^{1+s}}+\mu^2\tnorm{\eta}_{H^{2+s}}\le C\tnorm{f}_{H^s}&\text{if }\gam>1,\;\sig=0,\\
        \tnorm{\eta}_{H^s}+\mu^2\tnorm{\eta}_{H^{2+s}}+(|\mu|+|\sig|)\sig^2\tnorm{\eta}_{H^{3+s}}\le C\tnorm{f}_{H^s}&\text{if }\gam>1,\;|\sig|+|\mu|>0.
    \end{cases}
    \end{equation}
\end{prop}
\begin{proof}
    We take the Fourier transform of~\eqref{one dimensional linearization} to find the multiplier formulation $m_{\gam,\mu,\sig}(\xi)\mathscr{F}[\eta](\xi)=\mathscr{F}[\psi](\xi)$ for a.e. $\xi\in\R$, where
    \begin{equation}\label{m_mult_def}
        m_{\gam,\mu,\sig}(\xi)=\gam\tp{1+16\pi^2\mu^2|\xi|^2}+2\pi\ii\xi\tp{(1-\gam^2)+4\pi^2\sig^2|\xi|^2}.
    \end{equation}
    By considerations of the real and imaginary parts of $m_{\gam,\mu,\sig}$, we observe the equivalence
    \begin{equation}\label{the multiplier equivalence}
        |m_{\gam,\mu,\sig}(\xi)|\asymp 1+\mu^2|\xi|^2+|\xi|\babs{\f{1-\gam^2}{4\pi^2}+\sig^2|\xi|^2}.
    \end{equation}
    If $\gam<1$, then we evidently have
    \begin{equation}
        |m_{\gam,\mu,\sig}(\xi)|\asymp \tbr{\mu\xi}^2+\tbr{\xi}\tbr{\sig\xi}^2,
    \end{equation}
    and this implies the first case in~\eqref{estimates for the one dimensional linearization}.
    
    On the other hand, if $\gam>1$ and $\sig=0$, then from~\eqref{the multiplier equivalence} we deduce that
    \begin{equation}
        |m_{\gam,\mu,\sig}(\xi)|\asymp\tbr{\mu\xi}^2+\tbr{\xi}
    \end{equation}
    and are thus led to the second case in~\eqref{estimates for the one dimensional linearization}.

    Finally, let us assume that $\gam>1$ and $|\mu|+|\sig|>0$. Then, by breaking to cases, we deduce that
    \begin{equation}\label{ask you}
        |m_{\gam,\mu,\sig}(\xi)|\gtrsim\tbr{\mu\xi}^2+\tbr{\xi}\begin{cases}
            1&\text{if }|\xi|\lesssim|\sig|^{-1},\\
            0&\text{if }|\xi|\asymp|\sig|^{-1},\\
            \tbr{\sig\xi}^2&\text{if }|\xi|\gtrsim|\sig|^{-1}.
        \end{cases}
    \end{equation}
    Then by using that $\tbr{\sig\xi}\lesssim 1$ for $|\xi|\lesssim|\sig|^{-1}$, we deduce from the above the lower bounds
    \begin{equation}\label{new0}
        |m_{\gam,\mu,\sig}(\xi)|\gtrsim\tbr{\mu\xi}^2\ge\tbr{\mu\xi}\gtrsim\tbr{\mu\xi}+\tbr{\sig\xi}\gtrsim\tp{\tbr{\mu\xi}+\tbr{\sig\xi}}\tbr{\sig\xi}^2
        \text{ for }|\xi|\lesssim|\sig|^{-1}
    \end{equation}
    and
    \begin{equation}\label{new1}
        |m_{\gam,\mu,\sig}(\xi)|\gtrsim\tbr{\xi}\tbr{\sig\xi}^2\gtrsim\tp{\tbr{\mu\xi}+\tbr{\sig\xi}}\tbr{\sig\xi}^2,\text{ for }|\xi|\gtrsim|\sig|^{-1}.
    \end{equation}
    Combining~\eqref{new0} and~\eqref{new1} with $|m_{\gam,\mu,\sig}(\xi)|\gtrsim\tbr{\mu\xi}^2$ then reveals that
    \begin{equation}
        |m_{\gam,\mu,\sig}(\xi)|\gtrsim\tbr{\mu\xi}^2+\tbr{\sig\xi}^2\tp{\tbr{\mu\xi}+\tbr{\sig\xi}},
    \end{equation}
    and hence the final case in~\eqref{estimates for the one dimensional linearization} is shown.
\end{proof}

Motivated by the previous result, we make the following definition.

\begin{defn}[Reparameterization operators and cases for the parameters]\label{defn of reparameterization operators}
We make the following definitions.
\begin{enumerate}
    \item Given $\gam,\mu,\sig\in\tp{\R^+\setminus\tcb{1}}\times\R^2$ and $i\in\tcb{1,2,3}$, we define the Fourier multiplier operator  $P^i_{\gam,\mu,\sig}=p^i_{\gam,\mu,\sig}(\pd/2\pi\ii)$ for
    \begin{equation}
        p^i_{\gam,\mu,\sig}(\xi)=\tbr{\mu\xi}^2+\begin{cases}
            \tbr{\xi}\tbr{\sig\xi}^2&\text{if }i=1\\
            \tbr{\xi}&\text{if }i=2\\
            \tp{\tbr{\sig\xi}+\tbr{\mu\xi}}\tbr{\sig\xi}^2&\text{if }i=3.
        \end{cases}
    \end{equation}
    \item We define the sets $\mathcal{P}_i\subset(\R^+\setminus\tcb{1})\times\R^2$ via
    \begin{equation}
        \mathcal{P}_i=\begin{cases}
            \tcb{(\gam,\mu,\sig)\;:\;\gam<1}&\text{if }i=1,\\
            \tcb{(\gam,\mu,\sig)\;:\;\gam>1,\;\sig=0}&\text{if }i=2,\\
            \tcb{(\gam,\mu,\sig)\;:\;\gam>1,\;|\mu|+|\sig|>0}&\text{if }i=3.
        \end{cases}
    \end{equation}
    Notice that $\mathcal{P}_1\cup\mathcal{P}_2\cup\mathcal{P}_3=(\R^+\setminus\tcb{1})\times\R^2$, but $\mathcal{P}_2\cap\mathcal{P}_3\neq\es$.
\end{enumerate}
    
\end{defn}

By combining the previous result and definition, we are led to the following corollary.

\begin{coro}[One dimensional linear analysis, isomorphisms]\label{coro on 1d isos}
    For any choice of $s\in\N$, $\gam\in\R^+\setminus\tcb{1}$, $R \in\R^+$, and $\mu,\sig\in[-R,R]$, the following hold.
    \begin{enumerate}
        \item For $i\in\tcb{1,2,3}$ and $(\gam,\mu,\sig)\in\mathcal{P}_i\cap\tp{\R\times[-R,R]^2}$ the maps
    \begin{equation}\label{the A map}
        A_{\gam,\mu,\sig}^i=\tp{\gam\tp{1-4\mu^2\pd^2}+\tp{(1-\gam^2)-\sig^2\pd^2}\pd}\circ (P_{\gam,\mu,\sig}^i)^{-1}:H^s(\R)\to H^s(\R)
    \end{equation}
    are bounded linear isomorphisms.
        \item The maps in~\eqref{the A map} satisfy the bounds
        \begin{equation}
            \tnorm{(A^i_{\gam,\mu,\sig})^{-1}}_{\mathcal{L}(H^s)}\le C
        \end{equation}
        and
        \begin{equation}
            \tnorm{A_{\gam,\mu,\sig}^i}_{\mathcal{L}(H^s)}\le C\begin{cases}
                1&\text{if }i=1,\\
                1&\text{if }i=2,\\
                \tp{|\mu|+|\sig|}^{-1}&\text{if }i=3,
            \end{cases}
        \end{equation}
        for a constant $C$ depending only on $\gam$, $R$, and $s$.
    \end{enumerate}
\end{coro}
\begin{proof}
    $A^i_{\gam,\mu,\sig}=a^i_{\gam,\mu,\sig}(\pd/2\pi\ii)$ is a translation-commuting linear operator with symbol 
    \begin{equation}
        a^i_{\gam,\mu,\sig}(\xi)=\f{m_{\gam,\mu,\sig}(\xi)}{p^i_{\gam,\mu,\sig}(\xi)},
    \end{equation}
    where the numerator is defined by \eqref{m_mult_def} and the denominator is given in Definition \ref{defn of reparameterization operators}.  Hence, the first and second items will follow immediately if we can establish the symbol estimates
    \begin{equation}
        \babs{\f{m_{\gam,\mu,\sig}(\xi)}{p^i_{\gam,\mu,\sig}(\xi)}}\lesssim\begin{cases}
            1&\text{if }i=1,\\
            1&\text{if }i=2,\\
            \tp{|\mu|+|\sig|}^{-1}&\text{if }i=3,
        \end{cases}
        \text{ and }
        \babs{\f{p^i_{\gam,\mu,\sig}(\xi)}{m_{\gam,\mu,\sig}(\xi)}}\lesssim 1.
    \end{equation}
    The former follows from a simple calculation, while the latter is enforced by Proposition~\ref{prop on one dimensional linear analysis} and Definition~\ref{defn of reparameterization operators}.
\end{proof}

To complete this subsection, we now  consider the one dimensional nonlinear analysis. The following property of the reparameterization operators will be implicitly used numerous times in the proof of Proposition~\ref{prop on the one dimensional nonlinear maps} below.

\begin{lem}[Facts about reparameterization operators]\label{lemma on facts about the rep ops}
    Let $s\in\N$ and $i\in\tcb{1,2,3}$. The following mappings are continuous.
    \begin{equation}
        \mathcal{P}_i\times H^s(\R)\ni(\gam,\mu,\sig,\eta)\mapsto\begin{cases}
            (P_{\gam,\mu,\sig}^i)^{-1}\eta\in H^{1+s}(\R),\\
            \mu^2(P_{\gam,\mu,\sig}^i)^{-1}\eta\in H^{2+s}(\R),\\
            \sig^2(P_{\gam,\mu,\sig}^i)^{-1}\eta\in H^{3+s}(\R),
        \end{cases}
    \end{equation}
    where we recall that the sets $\mathcal{P}_i$ and the operators $P^i_{\gam,\mu,\sig}$ are given in Definition~\ref{defn of reparameterization operators}.
\end{lem}
\begin{proof}
    The assertions of the Lemma are equivalent to proving that the following maps are continuous. Fix $R\in\R^+$ and define
    \begin{multline}
        \tp{(-R,R)^3\cap\mathcal{P}_i}\times H^s(\R)\ni(\gam,\mu,\sig,\eta)\mapsto B(\gam,\mu,\sig,\eta)\\=\tp{\tbr{\pd}(P^i_{\gam,\mu,\sig})^{-1}\eta,\mu^2\tbr{\pd}^2(P^i_{\gam,\mu,\sig})^{-1}\eta,\sig^2\tbr{\pd}^3(P^i_{\gam,\mu,\sig})^{-1}\eta}\in (H^s(\R))^3.
    \end{multline}
    This is slightly subtle, since $B$ is not actually differentiable or even locally Lipschitz. To verify continuity, we first note that for fixed $(\gam,\mu,\sig)\in(-R,R)^3\cap\mathcal{P}_i$ the linear mapping
    \begin{equation}
        H^s(\R)\ni\eta\mapsto B(\gam,\mu,\sig,\eta)\in (H^s(\R))^3
    \end{equation}
    is bounded with an operator norm depending only on $s$ and $R$. The effect of this is that it suffices to show that for each fixed $\eta\in H^s(\R)$ the mapping
    \begin{equation}\label{the curve}
        (-R,R)^3\cap\mathcal{P}_i\ni(\gam,\mu,\sig)\mapsto B(\gam,\mu,\sig,\eta)\in (H^s(\R))^3
    \end{equation}
    is continuous. To verify this claim, we shall consider the following difference decomposition
    \begin{equation}\label{B1B2B3}
        B(\gam_0+\gam,\mu_0+\mu,\sig_0+\sig,\eta)-B(\gam_0,\mu_0,\sig_0,\eta)=B_1+B_2+B_3,
    \end{equation}
    where
    \begin{equation}
        B_j=\begin{cases}
            B(\gam_0+\gam,\mu_0+\mu,\sig_0+\sig,\eta)-B(\gam_0+\gam,\mu_0+\mu,\sig_0+\sig,\mathcal{S}_\ep\eta)&\text{if }j=1,\\
            B(\gam_0+\gam,\mu_0+\mu,\sig_0+\sig,\mathcal{S}_\ep\eta)-B(\gam_0,\mu_0,\sig_0,\mathcal{S}_\ep\eta)&\text{if }j=2,\\
            B(\gam_0,\mu_0,\sig_0,\mathcal{S}_\ep\eta)-B(\gam_0,\mu_0,\sig_0,\eta)&\text{if }j=3,
        \end{cases}
    \end{equation}
    for $\ep\in\R^+$ a small mollification parameter, $\mathcal{S}_\ep$ a mollification operator, $(\gam_0,\mu_0,\sig_0)$ and $(\gam_0+\gam,\mu_0+\mu,\sig_0+\sig)$ belonging to $(-R,R)^3\cap\mathcal{P}_i$, and $\eta\in H^s(\R)$. For $B_1$ and $B_3$, we simply use linearity in the final argument of $B$:
    \begin{equation}\label{B1B2 bound}
    \tnorm{B_1}_{H^s}+\tnorm{B_3}_{H^s}\lesssim_{R,s}\tnorm{(I-\mathcal{S}_\ep)\eta}_{H^s}.
    \end{equation}
    On the other hand, for $B_2$, we instead leverage the smoothness of $\mathcal{S}_\ep$ and use that $B(\cdot,\mathcal{S}_\ep\eta)$ is locally Lipschitz if you pay for more derivatives on $\mathcal{S}_\ep\eta$. This results in the bound
    \begin{equation}\label{B3 bound}
        \tnorm{B}_2\lesssim_{\ep,R,s}\tnorm{\gam,\mu,\sig}_{\R^3}\tnorm{\eta}_{H^s}.
    \end{equation}
    We combine~\eqref{B1B2 bound} with~\eqref{B1B2B3}. By selecting $\ep\in\R^+$ sufficiently small depending on $\eta$, and then choosing $\tnorm{\gam,\mu,\sig}_{\R^3}$ sufficiently small depending on $\ep$ and $\eta$, we can make the difference~\eqref{B1B2B3} as measured in $H^s(\R)$ smaller than any desired positive value. So it is shown that the mapping of~\eqref{the curve} is continuous, and thus the proof is complete.
\end{proof}

The following result is meant to capture the nonlinear mapping properties of an operator derived from the one dimensional equation~\eqref{reduced eta equation}. This will subsequently be paired with the linear estimates and another inverse function theorem argument.

\begin{prop}[On the one dimensional nonlinear maps]\label{prop on the one dimensional nonlinear maps}
    The following hold for $\N\ni s\ge 1$. 
    \begin{enumerate}
        \item There exists $\uprho\in\R^+$ such that for all $\eta\in B_{H^1}(0,\uprho)$ we have $\tnorm{\eta}_{L^{\infty}}\le1/2$ and the map
        \begin{equation}
            B_{H^1}(0,\uprho)\cap H^s(\R)\ni\eta\mapsto\f{\eta}{1+\eta}\in H^s(\R)
        \end{equation}
        is smooth. Moreover, we have the estimate
        \begin{equation}
            \bnorm{\f{\eta}{1+\eta}}_{H^s}\le C\tnorm{\eta}_{H^s},
        \end{equation}
        for a constant $C\in\R^+$ depending only on $s$.
        \item With $\uprho\in\R^+$ as in the previous item and $i\in\tcb{1,2,3}$, the maps $\mathcal{N}^i :  \mathcal{P}_i\times B_{H^s}(0,\uprho) \to H^s(\R)$ given by 
        \begin{multline}
            \mathcal{N}^i(\gam,\mu,\sig,\eta)
            =-\gam^2\bp{\f{(P_{\gam,\mu,\sig}^i)^{-1}\eta}{1+(P_{\gam,\mu,\sig}^i)^{-1}\eta}}'+\gam\f{(P_{\gam,\mu,\sig}^i)^{-1}\eta}{1+(P_{\gam,\mu,\sig}^i)^{-1}\eta} 
            -4\gam\mu^2\bp{(1+(P_{\gam,\mu,\sig}^i)^{-1}\eta)\bp{\f{(P_{\gam,\mu,\sig}^i)^{-1}\eta}{1+(P_{\gam,\mu,\sig}^i)^{-1}\eta}}'}' 
            \\
            +(1+(P_{\gam,\mu,\sig}^i)^{-1}\eta)\tp{(P_{\gam,\mu,\sig}^i)^{-1}\eta-\sig^2(P_{\gam,\mu,\sig}^i)^{-1}\eta''}' 
        \end{multline}
        are well-defined, continuous, differentiable with respect to the final factor with continuous partial derivative.
        \item The maps $\mathcal{O}^i : \mathcal{P}_i\times B_{H^s}(0,\uprho)\times(H^s(\R))^{\ell+1}\times (H^s(\R))^{\ell+1} \to H^s(\R)$ given by 
        \begin{equation}
                \mathcal{O}^i(\gam,\mu,\sig,\eta,\tcb{\tau_j}_{j=0}^\ell,\tcb{\varphi_j}_{j=0}^\ell)=\sum_{j=0}^\ell(P_{\gam,\mu,\sig}^i)^{-1}\eta)^j\tp{\tau_j\grad((P_{\gam,\mu,\sig}^i)^{-1}\eta)+\varphi_j}
        \end{equation}
        are continuous, and continuously partially differentiable with respect to the final two factors.
    \end{enumerate}
\end{prop}
\begin{proof}
    The existence of $\uprho\in\R^+$ as in the first item follows from the embedding $H^1(\R)\emb C^0_0(\R)$ along with Theorem D.3 and Corollary D.9 from Stevenson and Tice~\cite{stevenson2023wellposedness}.
    
    We next prove the second item.  First, we fix a parameter $R \in \R^+$ and consider the auxiliary mapping
    \begin{multline}
        [-R,R]^3\times B_{H^{s+3}}(0,\uprho)\ni(\gam,\mu,\sig,\zeta)\mapsto \mathcal{M}(\gam,\mu,\sig,\zeta)=-\gam^2\bp{\f{\zeta}{1+\zeta}}'\\+\gam\f{\zeta}{1+\zeta}-4\gam\mu^2\bp{(1+\zeta)\bp{\f{\zeta}{1+\zeta}}'}'+(1+\zeta)\tp{\zeta-\sig^2\zeta''}'\in H^{s}(\R).
    \end{multline}
    Thanks to repeated applications of the estimate from the first item, Proposition~\ref{corollary on tame estimates on simple multipliers}, and simple interpolation inequalities we learn that for $s\ge 1$ $\mathcal{M}$ obeys the bounds 
    \begin{equation}\label{estimate on the thing}
        \tnorm{\mathcal{M}(\gam,\mu,\sig,\eta)}_{H^s}\le C_{s,\uprho,R}\tp{\tnorm{\zeta}_{H^{s+1}}+\mu^2\tnorm{\zeta}_{H^{2+s}}+\sig^2\tnorm{\zeta}_{H^{3+s}}}\tp{1+ \tnorm{\zeta}_{H^1}+\sig^2\tnorm{\zeta}_{H^3}}.
    \end{equation}
    Now we can pair estimate~\eqref{estimate on the thing} with the relation
    \begin{equation} \label{cat_scratch_fever}
        \mathcal{M}(\gam,\mu,\sig,(P_{\gam,\mu,\sig}^i)^{-1}\eta)=\mathcal{N}^i(\gam,\mu,\sigma,\eta)
        \text{ for }
        i\in\tcb{1,2,3},\;(\gam,\mu,\sig)\in\mathcal{P}_i
    \end{equation}
    and Lemma \ref{lemma on facts about the rep ops} to deduce the bound
    \begin{equation}
        \tnorm{\mathcal{N}^i(\gam,\mu,\sig,\eta)}_{H^s}\le C_{s,\uprho,R}\begin{cases}
            \tnorm{\eta}_{H^s}&\text{if }i\in\tcb{1,2},\\
            (|\mu|+|\sig|)^{-1}\tnorm{\eta}_{H^s}&\text{if }i=3.
        \end{cases}
    \end{equation}
    Since $R \in \R^+$ was arbitrary, this shows that $\mathcal{N}^i$ is well-defined.

    Next, we sketch the argument that the $\mathcal{N}^i$ are continuous. We begin by decomposing the  difference
    \begin{equation}\label{it makes no difference}
        \mathcal{N}^i(\gam_0,\mu_0,\sig_0,\eta_0)-\mathcal{N}^i(\gam_1,\mu_1,\sig_1,\eta_1)=\bf{I}_1^i+\bf{I}_2^i+\bf{I}_3^i,
    \end{equation}
    where
    \begin{equation}
        \bf{I}_j^i=
        \begin{cases}
            \mathcal{N}^i(\gam_0,\mu_0,\sig_0,\eta_0)-\mathcal{N}^i(\gam_0,\mu_0,\sig_0,\mathcal{S}_\ep\eta_0)&\text{if }j=1,\\
            \mathcal{N}^i(\gam_0,\mu_0,\sig_0,\mathcal{S}_\ep\eta_0)-\mathcal{N}^i(\gam_1,\mu_1,\sig_1,\mathcal{S}_\ep\eta_1)&\text{if }j=2,\\
            \mathcal{N}^i(\gam_1,\mu_1,\sig_1,\mathcal{S}_\ep\eta_1)-\mathcal{N}^i(\gam_1,\mu_1,\sig_1,\eta_1)&\text{if }j=3
        \end{cases}
    \end{equation}
    for $\ep\in\R^+$ a small mollification parameter and $\mathcal{S}_\ep$ a mollification operator. The terms $\bf{I}_1^i$ and $\bf{I}_3^i$ are handled by first considering the corresponding differences for the operator $\mathcal{M}$. By arguing as above, we may deduce that
    \begin{equation}
        \tnorm{\mathcal{M}(\gam,\mu,\sig,\zeta_0)-\mathcal{M}(\gam,\mu,\sig,\zeta_1)}_{H^s}\le C_{\zeta_0,\zeta_1}\tp{\tnorm{\zeta_0-\zeta_1}_{H^{1+s}}+\mu^2\tnorm{\zeta_0-\zeta_1}_{H^{2+s}}+\sig^2\tnorm{\zeta_0-\zeta_1}_{H^{3+s}}},
    \end{equation}
    for some constant $C_{\zeta_0,\zeta_1}$ that is a function of $R$, $\tnorm{\zeta_k}_{H^{1+s}}$, $\mu^2\tnorm{\zeta_k}_{H^{2+s}}$, and $\sig^2\tnorm{\zeta_k}_{H^{3+s}}$  for $k\in\tcb{0,1}$.  Recalling the identity \eqref{cat_scratch_fever} and Lemma \ref{lemma on facts about the rep ops}, we then learn that
    \begin{multline}
        \tnorm{\bf{I}_1^i}_{H^s}+\tnorm{\bf{I}_3^i}_{H^s}\le \tilde{C}\tp{\tnorm{(1-\mathcal{S}_\ep)\eta_0}_{H^s}+\tnorm{(1-\mathcal{S}_\ep)\eta_1}_{H^s}}\begin{cases}
            1&\text{if }i\in\tcb{1,2},\\
            (|\mu|+|\sig|)^{-1}&\text{if }i=3,
        \end{cases}\\
        \le \tilde{C}\tp{\tnorm{\eta_0-\eta_1}_{H^s}+\tnorm{(1-\mathcal{S}_\ep)\eta_0}_{H^s}}\begin{cases}
            1&\text{if }i\in\tcb{1,2},\\
            (|\mu|+|\sig|)^{-1}&\text{if }i=3,
        \end{cases}
    \end{multline}
    where $\tilde{C}$ is a function of $R$, $\uprho$, and $s$.
    
    On the other hand, for the $\bf{I}_2$ term, we  leverage the smoothness $\mathcal{S}_\ep\eta_0$ and $\mathcal{S}_\ep\eta_1$ in order to invoke a cheap bound that spends three derivatives ($\mathcal{M}$ is locally Lipschitz if you pay for three derivatives) at the expense of a bad constant depending on $\ep$:
    \begin{multline}
        \tnorm{\bf{I}_2^i}_{H^s}\lesssim_{R,s,\uprho} \norm{\gam_0-\gam_1,\mu_0-\mu_1,\sig_0-\sig_1}_{\R^3}+\tnorm{\mathcal{S}_\ep(\eta_0-\eta_1)}_{H^{3+s}} 
        \\
        \lesssim_{R,s,\uprho} \norm{\gam_0-\gam_1,\mu_0-\mu_1,\sig_0-\sig_1}_{\R^3}+C_\ep\tnorm{\eta_0-\eta_1}_{H^{s}}.
    \end{multline}
    The difference in~\eqref{it makes no difference} can then be made arbitrarily small by first choosing $\ep$ depending on $(\gam_0,\mu_0,\sig_0,\eta_0)$ and then selecting $(\gam_1,\mu_1,\sig_1,\eta_1)$ sufficiently near the former tuple.

    The previous argument can be readily adapted to show that the maps $\mathcal{N}^i$ are continuously partially differentiable with respect to the second factor. This finishes the proof of the second item.  The third item is proved in a similar fashion to the second item. We omit the details for brevity.
 \end{proof}

Once more, we will connect the linear and nonlinear theories via an application of the inverse function theorem. For our purposes here, we need a slightly specialized version of this result. We shall use the following version of the inverse function theorem, as formulated in Theorem A in Crandall and Rabinowitz~\cite{MR0288640} (for a verbose proof see Theorem 2.7.2 in Nirenberg~\cite{MR1850453}, but note that there is slight misstatement of the uniqueness assertion in the first item there that is correct in~\cite{MR0288640}). 

\begin{thm}[Implicit function theorem]\label{implicit function theorem}
    Let $X,Y,Z$ be Banach spaces over the same field, and let $U\subset X\times Y$ be open.  Suppose that $f:U \to Z$ is continuous and that $f$ has a Fr\'echet derivative with respect to the first factor, $D_1f : U\to\mathcal{L}(X;Z)$, that is also continuous.  Further suppose that $(x_0,y_0)\in U$ and $f(x_0,y_0)=0$.  If $D_1f(x_0,y_0)$ is an isomorphism of $X$ onto $Z$, then there exist balls $B(y_0,r_Y)\subset Y$ and $B(x_0,r_X)\subset X$ such that $B(x_0,r_X)\times B(y_0,r_Y)\subset U$ and a continuous function $u:B(y_0,r_Y)\to B(x_0,r_X)$ such that $u(y_0)=x_0$ and $f(u(y),y)=0$ for all $y\in B(y_0,r_Y)$. Moreover, the implicit function $u$ is continuous.
\end{thm}

We now state our results for~\eqref{reduced eta equation} in an abstract formulation. Afterward, we will unpack things to reveal a PDE result.

\begin{thm}[One dimensional analysis, abstract formulation]\label{thm on abstract 1d analysis}
    Let $\N\ni s\ge 1$. For $i\in\tcb{1,2,3}$ there exists open sets
    \begin{equation}
        \mathcal{P}_i\times\tcb{0}^{\ell+1}\times\tcb{0}^{\ell+1}\subset\mathscr{U}_s^i\subset\mathcal{P}_i\times(H^s(\R))^{\ell+1}\times(H^s(\R))^{\ell+1},
    \end{equation}
    radii $\tcb{\ep^i_{s,\bf{p}}}_{\bf{p}\in\mathcal{P}_i}\subset(0,\uprho]$, where $\uprho\in\R^+$ is from the first item of Proposition~\ref{prop on the one dimensional nonlinear maps}, and continuous mappings
    \begin{equation}
        \upiota^i:\mathscr{U}^i_s\to\bigcup_{\bf{p}\in\mathcal{P}_i}B_{H^s}(0,\ep^i_{s,\bf{p}})
    \end{equation}
    with the property that for all $\bf{U}=(\gam,\mu,\sig,\tcb{\tau_j}_{j=0}^\ell,\tcb{\varphi_j}_{j=0}^\ell)\in\mathscr{U}^i_s$ we have that $\eta=\upiota^i(\bf{U})\in B_{H^s}(0,\ep^i_{s,\bf{p}})$ (with $\bf{p}=(\gam,\mu,\sig)$ is the unique solution to
    \begin{equation}
        \mathscr{N}^i(\gam,\mu,\sig,\eta)+\mathcal{O}^i(\gam,\mu,\sig,\eta,\tcb{\tau_j}_{j=0}^\ell,\tcb{\varphi_j}_{j=0}^\ell)=0.
    \end{equation}
\end{thm}
\begin{proof}
    The strategy of the proof here is much the same as that of Theorems~\ref{thm on well-posedness, I} and~\ref{thm on well-posedness II}; as such, we only give a sketch of the argument. For the nonlinear mapping 
    \begin{multline}\label{then i think ill wait}
        \mathcal{P}_i\times B_{H^s}(0,\uprho)\times(H^s(\R))^{\ell+1}\times \tp{H^s(\R)}^{\ell+1}\ni(\gam,\mu,\sig,\eta,\tcb{\tau_j}_{j=0}^\ell,\tcb{\varphi_j}_{j=0}^\ell)\\\mapsto\mathscr{N}^i(\gam,\mu,\sig,\eta)+\mathcal{O}^i(\gam,\mu,\sig,\eta,\tcb{\tau_j}_{j=0}^\ell,\tcb{\varphi_j}_{j=0}^\ell)\in H^s(\R)
    \end{multline}
    we aim to satisfy the hypotheses of the implicit function theorem at a point $(\bf{p},0,0,0)$ for a fixed $\bf{p}\in\mathcal{P}_i$. That the mapping in~\eqref{then i think ill wait} is continuously partially differentiable with respect to the second factor is ensured by Proposition~\ref{prop on the one dimensional nonlinear maps}. The partial derivatives obey the requisite mapping properties in light of Corollary~\ref{coro on 1d isos}. In this way we obtain local continuous implicit functions $\upiota^i_{\bf{p}}$ that are defined in a small neighborhood of $(\bf{p},0,0,0)$. We then glue these maps together, via local uniqueness, and obtain an implicit function $\upiota^i$ defined on the union of these small neighborhoods.
\end{proof}

The next result is the PDE formulation of the previous abstract form.

\begin{thm}[One dimensional analysis, PDE formulation]\label{thm on one dimensional analysis, PDE formulation}
    Let $\N\ni s\ge 1$, $i\in\tcb{1,2,3}$, and $\mathscr{U}_s^i$ be the open set granted by Theorem~\ref{thm on abstract 1d analysis}. For $(\gam,\mu,\sig,\tcb{\tau_j}_{j=0}^\ell,\tcb{\varphi_j}_{j=1}^\ell)\in\mathscr{U}^i_s$ we set
    \begin{equation}\label{the free surface and the velocity}
        \eta=(P_{\gam,\mu,\sig}^i)^{-1}[\upiota^i(\gam,\mu,\sig,\tcb{\tau_j}_{j=0}^\ell,\tcb{\varphi_j}_{j=1}^\ell)],\quad v=\gam\f{\eta}{1+\eta},\quad\mathcal{F}=\mathcal{O}^i(\gam,\mu,\sig,P^i_{\gam,\mu,\sig}\eta,\tcb{\tau_j}_{j=0}^\ell,\tcb{\varphi_j}_{j=0}^\ell).
    \end{equation}
    The following hold.
    \begin{enumerate}
        \item System~\eqref{traveling wave formulation of the equation} is satisfied by $(v,\eta)$ with $\mathcal{F}$ as a forcing term.
        \item The following mappings are well-defined and continuous:
        \begin{equation}
            \mathscr{U}^i_s\ni(\gam,\mu,\sig,\tcb{\tau_j}_{j=0}^\ell,\tcb{\varphi_j}_{j=1}^\ell)\mapsto\begin{cases}
                (v,\eta)\in (H^{1+s}(\R))^2,\\\
                \mu^2(v,\eta)\in(H^{2+s}(\R))^2,\\
                \sig^2(v,\eta)\in(H^{3+s}(\R))^2.
            \end{cases}
        \end{equation}
    \end{enumerate}
\end{thm}
\begin{proof}
    We define $\eta$, $v$, and $\mathcal{F}$ as in~\eqref{the free surface and the velocity}. Thanks to the first item of Proposition~\ref{prop on the one dimensional nonlinear maps}, this definition of $v$ makes sense and is a smooth function of $\eta$ in matching Sobolev spaces. The first item is then immediate from the derivation of~\eqref{reduced eta equation} from~\eqref{traveling wave formulation of the equation}. The continuity of the mappings in the second item are similarly clear thanks to properties of the reparameterization operators $P^i_{\gam,\mu,\sig}$ from Definition~\ref{defn of reparameterization operators} and Lemma~\ref{lemma on facts about the rep ops}.
\end{proof}

% _+__+_ -_+__+_ -_+__+_ -_+__+_ -_+__+_ -_+__+_ -_+__+_ -_+__+_ -_+__+_ -_+__+_ -_+__+_ -_+__+_ -_+__+_ -
\section{Nonlinear analysis of the shallow water system}\label{section on nonlinear analysis}
% _+__+_ -_+__+_ -_+__+_ -_+__+_ -_+__+_ -_+__+_ -_+__+_ -_+__+_ -_+__+_ -_+__+_ -_+__+_ -_+__+_ -_+__+_ -

We now aim to understand the behavior of the system~\eqref{traveling wave formulation of the equation} in a regime larger than that covered by Section~\ref{section on analysis with both visc and surf}.  To do so, we must face the problem of derivative loss.  We are thus lead to replace the implicit function theorem strategy of Section~\ref{section on analysis with both visc and surf} with a Nash-Moser inverse function theorem. The version which we elect to use is the one carefully stated in Appendix~\ref{subsection on a NMIFT}.

The purpose of this section is to both handle the nonlinear hypotheses of the Nash-Moser theorem and to lay the groundwork for the verification of the linear hypotheses. In Section~\ref{section on banach scales for the TWP} we introduce the scales of Banach spaces and parameter dependent norms in play for our Nash-Moser and estimates framework. Next, in Section~\ref{section on smooth tameness of the nonlinear operator}, we associate to~\eqref{traveling wave formulation of the equation} some nonlinear differential operators and then verify that these satisfy certain smoothness and tameness assertions. Finally, in Section~\ref{section on derivative splitting}, we take a look at certain derivatives of the previously introduced nonlinear maps and isolate some minimal essential structure.

% _+__+_ -_+__+_ -_+__+_ -_+__+_ -_+__+_ -_+__+_ -_+__+_ -_+__+_ -_+__+_ -_+__+_ -_+__+_ -_+__+_ -_+__+_ -
\subsection{Banach scales for the traveling wave problem}\label{section on banach scales for the TWP}
% _+__+_ -_+__+_ -_+__+_ -_+__+_ -_+__+_ -_+__+_ -_+__+_ -_+__+_ -_+__+_ -_+__+_ -_+__+_ -_+__+_ -_+__+_ -

The goal of this subsection is to define the scales of Banach spaces that are relevant for our analysis of the traveling shallow water equations~\eqref{traveling wave formulation of the equation} and then to verify basic properties. 

Let $s\in\N$. First, we introduce spaces that will play a role in the domain of our nonlinear map.  We write
\begin{equation}\label{main piece of the domain}
    \X_s=H^s(\R^d;\R^d)\times\mathcal{H}^s(\R^d;\R^d)
    \text{ with norm }    
    \tnorm{v,\eta}_{\X_s}=\tp{\tnorm{v}_{H^s}^2+\tnorm{\eta}_{\mathcal{H}^s}^2}^{1/2}
\end{equation}
to be our omnisonic container space. We then write
\begin{equation}
    \Xs_s=H^s(\R^d;\R^d)\times\mathcal{H}^{1+s}(\R^d;\R^d)\text{ with norm }    
    \tnorm{v,\eta}_{\Xs_s}=\tp{\tnorm{v}_{H^s}^2+\tnorm{\eta}_{\mathcal{H}^{1+s}}^2}^{1/2}
\end{equation}
to denote our subsonic container space.

Next, we introduce spaces that will play a role in the codomain of the nonlinear map.
\begin{equation}\label{main piece of the codomain}
    \Y_s=\tp{\dot{H}^{-1}\cap H^{1+s}}(\R^d)\times H^s(\R^d;\R^d)
        \text{ with norm }    
    \tnorm{h,f}_{\Y_s}=\tp{\tsb{h}_{\dot{H}^{-1}}^2+\tnorm{h}_{H^{1+s}}^2+\tnorm{f}_{H^s}^2}^{1/2}.
\end{equation}
We then give the data spaces
\begin{equation}
    \W_s= (H^s(\R^d;\R^{d\times d}))^{\ell+1} \times (H^s(\R^d;\R^d))^{\ell+1}
        \text{ with norm }    
    \tnorm{\tcb{\tau_i}_{i=0}^\ell,\tcb{\varphi_i}_{i=0}^\ell}_{\W_s}=\bp{\sum_{i=0}^\ell\sp{\tnorm{\tau_i}_{H^s}^2+\tnorm{\varphi_i}^2_{H^s}}}^{1/2}.
\end{equation}
We synthesize the above into the domain and codomain of the nonlinear mapping that we shall associate with~\eqref{traveling wave formulation of the equation} in the next subsection. First we have the omnisonic domain and codomain:
\begin{equation}\label{the domain and codomain banach scales}
    \E_s=\R^3\times\W_s\times\X_s\times\X_s  
    \text{ and }
    \F_s=\R^3\times\Y_s\times\W_s\times\X_s.
\end{equation}
Next we have the subsonic domain and codomain:
\begin{equation}\label{the domain and codomain banach scales, subsonic case}
    \Es_s=\R^3\times\W_s\times\Xs_s\times\Xs_s  
    \text{ and }
    \Fs_s=\R^3\times\Y_s\times\W_s\times\Xs_s.
\end{equation}

From these spaces we now denote the following Banach scales:
\begin{equation}\label{the banach scales for our problem}
    \pmb{\X}=\tcb{\X_s}_{s\in\N},\quad\pmb{\Y}=\tcb{\Y_s}_{s\in\N},\quad \pmb{\W}=\tcb{\W_s}_{s\in\N},\quad\pmb{\E}=\tcb{\E_s}_{s\in\N},\quad\pmb{\F}=\tcb{\F_s}_{s\in\N},
\end{equation}
and
\begin{equation}
    \pmb{\X}^{\mathrm{sub}}=\tcb{\Xs_s}_{s\in\N},\quad\pmb{\E}^{\m{sub}}=\tcb{\Es_s}_{s\in\N},\quad\pmb{\F}^{\m{sub}}=\tcb{\Fs_s}_{s\in\N}.
\end{equation}

We have the following result which says that the scales~\eqref{the banach scales for our problem} our admissible for our Nash-Moser framework (of Appendix~\ref{subsection on a NMIFT}).
\begin{lem}[LP-smoothability of the Banach scales]\label{lem on LP smoothability of the Banach scales}
    Each of the Banach scales $\pmb{\X}$, $\pmb{\X}^{\m{sub}}$, $\pmb{\Y}$, $\pmb{\W}$, $\pmb{\E}$, $\pmb{\E}^{\m{sub}}$, $\pmb{\F}$, and $\pmb{\F}^{\m{sub}}$ is LP-smoothable in the sense of Definition~\ref{defn of smoothable and LP-smoothable Banach scales}.
\end{lem}
\begin{proof}
    As the product of LP-smoothable scales is again LP-smoothable (see Lemma 2.17 in Stevenson and Tice~\cite{stevenson2023wellposedness}), it is sufficient to check that each of the Banach scales $\R$, $\tcb{H^s(\R^d;\R^\ell)}_{s\in\N}$, $\tcb{\mathcal{H}^s(\R^d)}$, and $\tcb{(\dot{H}^{-1}\cap H^s)(\R^d)}_{s\in\N}$ is LP-smoothable. Each of these is LP-smoothable as a consequence of Examples~\ref{example of a fixed Banach space} and~\ref{example of Euclidean Soblev spaces}, along with Proposition~\ref{lem on lp smoothability of anisotropic Sobolev spaces}. 
\end{proof}

We shall also need norms (and spaces) that measure appropriately the behavior in the limit as $\mu\to0$ or $\sig\to0$; thus we make the following definitions.

\begin{defn}[Parameter dependent norms and spaces]\label{defn of parameter dependent norms}
    Given $s\in\N$, $\mu,\sig\in\R$ we make the following definitions of extended norms:
    \begin{equation}
        {_\mu}\tnorm{v}_{H^s}=\tp{\tnorm{v}_{H^s}^2+\mu^4\tnorm{v}^2_{H^{2+s}}}^{1/2},\quad {_\sig}\tnorm{\eta}_{\mathcal{H}^s}=\tp{\tnorm{\eta}^2_{\mathcal{H}^s}+\sig^4\tnorm{\eta}^2_{\mathcal{H}^{2+s}}}^{1/2},
    \end{equation}
    and
    \begin{equation}
        {_{\mu,\sig}}\tnorm{\eta}_{\mathcal{H}^s}=\tp{\tnorm{\eta}_{\mathcal{H}^s}^2+\sig^4\tnorm{\eta}_{\mathcal{H}^{2+s}}^2+\tp{\mu^2+\sig^2}\tp{\tnorm{\eta}_{\mathcal{H}^{1+s}}^2+\sig^4\tnorm{\eta}_{\mathcal{H}^{3+s}}^2}}^{1/2},
    \end{equation}
    where $v\in H^s(\R^d;\R^d)$ and $\eta\in\mathcal{H}^s(\R^d)$. We then make the following definitions of parameter dependent spaces.
    \begin{equation}
        {_\mu}H^s(\R^d;\R^d)=\tcb{v\in H^s(\R^d)\;:\;{_\mu}\tnorm{v}_{H^s}<\infty},\quad{_{\mu,\sig}}\mathcal{H}^s(\R^d)=\tcb{\eta\in\mathcal{H}^s(\R^d)\;:\;{_{\mu,\sig}}\tnorm{\eta}_{\mathcal{H}^s}<\infty},
    \end{equation}
    \begin{equation}
        {_{\sig}}\mathcal{H}^s(\R^d)=\tcb{\eta\in\mathcal{H}^s(\R^d)\;:\;{_{\sig}}\tnorm{\eta}_{\mathcal{H}^s}<\infty},
    \end{equation}
    and 
    \begin{equation}\label{sybil}
        {_{\mu,\sig}}\X_s=\tp{{_\mu}H^s(\R^d;\R^d)}\times\tp{{_{\mu,\sig}}\mathcal{H}^s(\R^d)},\quad{_{\mu,\sig}}\Xs_s=\tp{{_\mu}H^s(\R^d;\R^d)}\times\tp{{_\sig}\mathcal{H}^{1+s}(\R^d)}.
    \end{equation}
\end{defn}

\begin{rmk}
We have the following equivalences of Hilbert spaces:
\begin{equation}
    {_\mu}H^s(\R^d;\R^d)=\begin{cases}
        H^s(\R^d;\R^d)&\text{if }\mu=0,\\
        H^{2+s}(\R^d;\R^d)&\text{if }\mu\neq0,
    \end{cases}\quad{_{\mu,\sig}}\mathcal{H}^s(\R^d)=\begin{cases}
        \mathcal{H}^{s}(\R^d)&\text{if }\mu=\sig=0,\\
        \mathcal{H}^{1+s}(\R^d)&\text{if }\mu\neq0,\;\sig=0,\\
        \mathcal{H}^{3+s}(\R^d)&\text{if }\sig\neq0,
    \end{cases}
\end{equation}
and
\begin{equation}
    {_\sig}\mathcal{H}^{1+s}(\R^d)=\begin{cases}
        \mathcal{H}^{1+s}(\R^d)&\text{if }\sig=0,\\
        \mathcal{H}^{3+s}(\R^d)&\text{if }\sig\neq0.
    \end{cases}
\end{equation}
\end{rmk}

% _+__+_ -_+__+_ -_+__+_ -_+__+_ -_+__+_ -_+__+_ -_+__+_ -_+__+_ -_+__+_ -_+__+_ -_+__+_ -_+__+_ -_+__+_ -
\subsection{Smooth-tameness of the nonlinear operator}\label{section on smooth tameness of the nonlinear operator}
% _+__+_ -_+__+_ -_+__+_ -_+__+_ -_+__+_ -_+__+_ -_+__+_ -_+__+_ -_+__+_ -_+__+_ -_+__+_ -_+__+_ -_+__+_ -

In this part of the document we associate to system~\eqref{traveling wave formulation of the equation} nonlinear operators encoding, in particular, the forward mapping of the PDE. We then verify that these operators obey the requisite smoothness and tameness assumptions required by our Nash-Moser scheme. It turns out that this latter task is delightfully easy due to the simple structure of each constituent nonlinearity in the maps defined below.

We shall define two distinct nonlinear operators; the first is meant for a wave speed agnostic analysis that in the end misses optimal derivative counting in the `subsonic regime' while the second is adapted to the case of subsonic wave speed.

\begin{defn}[Nonlinear maps, omnisonic case]\label{nonlinear map defs}
We associate to the system~\eqref{traveling wave formulation of the equation} the following operators, which are defined for $\N\ni s\ge\max\tcb{3,1+\tfloor{d/2}}$. First, we have $\Uppsi:\R^3\times\X_{s+3}\to\Y_s$ defined by
\begin{equation}
    \Uppsi(\gam,\mu,\sig,v,\eta)=\bpm\grad\cdot((1+\eta)(v-\gam e_1))\\(1+\eta)(v-\gam e_1)\cdot\grad v+v-\mu^2\grad\cdot((1+\eta)\S v)+(1+\eta)(1-\sig^2\Delta)\grad\eta\epm.
\end{equation}
Second, we have $\Upphi:\R^2\times\W_s\times\X_{s+1}\to\Y_s$ defined by
\begin{equation}
     \Upphi(\mu,\sig,\tcb{\tau_{i}}_{i=0}^\ell,\tcb{\varphi_{i}}_{i=0}^\ell,v,\eta)=\sum_{i=0}^\ell\eta^i(0,\m{Tr}\tau_i\grad\eta+(\mu+\sig)\tau_i\grad\eta+\varphi_i).
\end{equation}
Next, we define $\Upupsilon:\R^3\times(\X_{s+3}\times\W_s)\to\Y_s\times\W_s$ via
\begin{equation}\label{upsilon is found here}
    \Upupsilon(\gam,\mu,\sig,v,\eta,\tcb{\tau_i}_{i=0}^\ell,\tcb{\varphi_i}_{i=0}^\ell)=\sp{\Uppsi(\gam,\mu,\sig,v,\eta)+\Upphi(\mu,\sig,\tcb{\tau_i}_{i=0}^\ell,\tcb{\varphi_i}_{i=0}^\ell,v,\eta),\tcb{\tau_i}_{i=0}^\ell,\tcb{\varphi_i}_{i=0}^\ell}.
\end{equation}
Finally, we set  $\Bar{\Uppsi}:\E_{s+3}\to\F_s$ to be the map
\begin{equation}\label{definition of uppsi bar}
    \Bar{\Uppsi}(\gam,\mu,\sig,\tcb{\tau_i}_{i=0}^\ell,\tcb{\varphi_i}_{i=0}^\ell,w,\xi,v,\eta)=\bpm\gam,\mu,\sig\\\Upupsilon(\gam,\mu,\sig,v,\eta,\tcb{\tau_i}_{i=0}^\ell,\tcb{\varphi_i}_{i=0}^\ell)\\\tbr{\mu\grad}^2v-w\\(1+(\mu+\sig)|\grad|)\tbr{\sig\grad}^2\eta-\xi\epm.
\end{equation}
\end{defn}

\begin{defn}[Nonlinear maps, subsonic case]\label{subsonic nonlinear maps definition}
Let $\N\ni s\ge\max\tcb{2,1+\tfloor{d/2}}$. We note that $\Uppsi:\R^3\times\Xs_{s+2}\to\Y_s$. In the subsonic case we also associate to system~\eqref{traveling wave formulation of the equation} the following additional operators. First we define $\Upphi^{\m{sub}}:\W_s\times\Xs_s\to\Y_s$ via 
\begin{equation}
    \Upphi^{\m{sub}}(\tcb{\tau_i}_{i=0}^\ell,\tcb{\varphi_i}_{i=0}^\ell,v,\eta)=\sum_{i=0}^\ell \eta^i (0,\tau_i\grad\eta+\varphi_i).
\end{equation}
Next we define $\Upupsilon^{\m{sub}}:\R^3\times(\Xs_{s+2}\times\W_s)\to\Y_s\times\W_s$ via
\begin{equation}
    \Upupsilon^{\m{sub}}(\gam,\mu,\sig,v,\eta,\tcb{\tau_i}_{i=0}^\ell,\tcb{\varphi_i}_{i=0}^\ell)=(\Uppsi(\gam,\mu,\sig,v,\eta)+\Upphi^{\mathrm{sub}}(\tcb{\tau_i}_{i=0}^\ell,\tcb{\varphi_i}_{i=0}^\ell,v,\eta),\tcb{\tau_i}_{i=0}^\ell,\tcb{\varphi_i}_{i=0}^\ell).
\end{equation}
Finally we set $\Bar{\Uppsi}^{\m{sub}}:\Es_{s+2}\to\Fs_s$ to be the map
\begin{equation}\label{uppsi bar sub}
    \Bar{\Uppsi}^{\m{sub}}(\gam,\mu,\sig,\tcb{\tau_i}_{i=0}^\ell,\tcb{\varphi_i}_{i=0}^\ell,w,\xi,v,\eta)=\bpm\gam,\mu,\sig\\\Upupsilon^{\m{sub}}(\gam,\mu,\sig,v,\eta,\tcb{\tau_i}_{i=0}^\ell,\tcb{\varphi_i}_{i=0}^\ell)\\\tbr{\mu\grad}^2v-w\\\tbr{\sig\grad}^2\eta-\xi\epm.
\end{equation}
\end{defn}

We shall use the Banach scales notation of~\eqref{atomic Banach scale notation} from Appendix~\ref{toolbox for smooth-tameness}. The remaining goals of this subsection are to verify the smooth-tameness of the maps $\Bar{\Uppsi}$ and $\Bar{\Uppsi}^{\m{sub}}$. 

\begin{thm}[Smooth-tameness]\label{thm on smooth-tameness}
     We have that
     \begin{enumerate}
         \item $\Bar{\Uppsi} \in {_{\m{s}}}T^\infty_{3,\max\tcb{3,1+\tfloor{d/2}}}(\pmb{\E};\pmb{\F})$,
         \item $\Bar{\Uppsi}^{\m{sub}}\in{_{\m{s}}}T^\infty_{2,\max\tcb{2,1+\tfloor{d/2}}}(\pmb{\E}^{\m{sub}};\pmb{\F}^{\m{sub}})$.
     \end{enumerate}
\end{thm}
\begin{proof}
    We only prove the first item, as the second item follows from a nearly identical argument. We begin by proving that $\Uppsi\in {_{\m{s}}}T^{\infty}_{3,\max\tcb{3,1+\tfloor{d/2}}}(\R^4\times\pmb{\X};\pmb{\Y})$. It suffices to show that each of the multilinear maps that comprise $\Uppsi$ is strongly $3-$tame with base $\max\tcb{3,1+\tfloor{d/2}}$, but due to Remark~\ref{remark on purely multlinear maps} we may reduce to proving $3-$tameness. In handling each of these maps we will repeatedly use our results on the preservation of tameness under compositions and products, namely Proposition~\ref{lem on composition of tame maps} and Corollary~\ref{lem on product of smooth tame maps}, without explicit reference.

    We first establish that the map
    \begin{equation}\label{prop on smooth tameness, I 1}
        \bf{H}(\R^d)\times\bm{\mathcal{H}}\ni(v,\eta)\mapsto\grad\cdot((1+\eta)(v-\gam e_1))=-\gam\pd_1\eta+\grad\cdot v+\grad\cdot(\eta v)\in\dot{H}^{-1}\cap\bf{H}(\R)
    \end{equation}
    is  $T^\infty_{1,1+\tfloor{d/2}}$ by examining each of the three terms in the sum.   The map $\mathcal{H}^s(\R^d)\ni\eta\mapsto\pd_1\eta\in\tp{\dot{H}^{-1}\cap H^{s-1}}(\R^d)$ is bounded and linear thanks to Proposition~\ref{proposition on spatial characterization of anisobros}, while the map $H^s(\R^d)\ni v\mapsto\grad\cdot v\in(\dot{H}^{-1}\cap H^{s-1})(\R^d)$ is clearly bounded and linear as well.  The bilinear  map $(v,\eta)\mapsto\grad\cdot(v\eta)$ is  $T^\infty_{1,1+\tfloor{d/2}}$ thanks to what we've just established and Lemma~\ref{lem on smooth tameness of products 3}. This completes the analysis of \eqref{prop on smooth tameness, I 1}.

   To see that the map
    \begin{equation}
        \bf{H}(\R^d)\times\bm{\mathcal{H}}\ni(v,\eta)\mapsto(1+\eta)(v-\gam e_1)\cdot\grad v\in\bf{H}(\R^d)
    \end{equation}
    is $T^\infty_{1,1+\tfloor{d/2}}$ we again expand and examine the resulting linear, bilinear, and trilinear terms. In each case, tameness follows from the boundedness of the linear map $\mathcal{H}^s(\R^d)\ni\eta\mapsto(\Uppi^1_{\m{L}}\eta,\Uppi^1_{\m{H}}\eta)\in W^{\infty,\infty}(\R^d)\times H^s(\R^d)$, which is Proposition~\ref{proposition on frequency splitting}, along with Lemma~\ref{lem on smooth tameness of products 2}.

    Similar arguments then verify that the map
    \begin{equation}
        \R\times\bf{H}(\R^d)\times\bf{H}\ni(\mu,v,\eta)\mapsto v- \mu^2\grad\cdot\tp{(1+\eta)\S v}\in\bf{H}(\R^d)
    \end{equation}
    is  $T^\infty_{2,\max\tcb{2,1+\tfloor{d/2}}}$  and that 
    \begin{equation}
        \R\times\bm{\mathcal{H}}\ni(\sig,\eta)\mapsto(1+\eta)(1-\sig^2\Delta)\grad\eta\in\bf{H}(\R^d)
    \end{equation}
    is $T^\infty_{3,\max\tcb{3,1+\tfloor{d/2}}}$.  Hence, $\Uppsi$ is strongly $3-$tame with base $\max\tcb{3,1+\tfloor{d/2}}$. 

    The next step of the proof is to verify that $\Upphi\in {_{\m{s}}}T^\infty_{1,1+\tfloor{d/2}}(\R^2\times\pmb{\W}\times\pmb{\X};\pmb{\Y})$; however, due to the simple structure of this operator, we see that this is a direct consequence of Lemmas~\ref{lemma on tameness of products 1} and~\ref{lem on smooth tameness of products 3}.

    As a consequence of the first and second steps of the proof, we thus learn that
    \begin{equation}\label{most banking institutions offer}
        \Upupsilon\in{_{\m{s}}}T^\infty_{3,\max\tcb{3,1+\tfloor{d/2}}}(\R^3\times\pmb{\W}\times\pmb{\X};\pmb{\Y}\times\pmb{\W}),
    \end{equation}
    where we recall that $\Upupsilon$ is the mapping defined in~\eqref{upsilon is found here}.

    Finally, we can complete the proof of the theorem statement. The first component of $\Bar{\Uppsi}$ is linear and thus trivially ${_{\m{s}}}T^\infty_{3,\max\tcb{3,1+\tfloor{d/2}}}$. The previous two steps of the proof handle the $\Upupsilon$ component of $\Bar{\Uppsi}$. Corollary~\ref{lem on product of smooth tame maps} handles the final two components.  
\end{proof}

% _+__+_ -_+__+_ -_+__+_ -_+__+_ -_+__+_ -_+__+_ -_+__+_ -_+__+_ -_+__+_ -_+__+_ -_+__+_ -_+__+_ -_+__+_ -
\subsection{Derivative splitting}\label{section on derivative splitting}
% _+__+_ -_+__+_ -_+__+_ -_+__+_ -_+__+_ -_+__+_ -_+__+_ -_+__+_ -_+__+_ -_+__+_ -_+__+_ -_+__+_ -_+__+_ -

Now that we understand the smoothness properties of the maps $\Bar{\Uppsi}$ and $\Bar{\Uppsi}^{\m{sub}}$, we next make a preliminary analysis of their derivatives with the goal of identifying a minimal essential structure. The sets up our linear analysis in the sequel to handle first pleasantly simplified equations.

The following lemma is a technical calculation that ensures we may make a useful change of unknowns in some linear problem. The subscript $\m{WD}$ is an acronym for `well-defined', referring to the fact that one can take the reciprocal of $1+\eta$.
\begin{lem}\label{the preliminary lemma for you and I}
    There exists $\rho_{\m{WD}}\in\R^+$ such that the following hold for $\N\ni s\ge 2+\tfloor{d/2}$.
    \begin{enumerate}
        \item For all $\eta\in B_{\mathcal{H}^{2+\tfloor{d/2}}}(0,\rho_{\m{WD}})$, we have that $\tnorm{\eta}_{L^\infty}\le 1/2$.
        \item For all $0<\ep\le\rho_{\m{WD}}$ and all $\eta\in B_{\mathcal{H}^{2+\tfloor{d/2}}}(0,\ep)\cap\mathcal{H}^s(\R^d)$, we have the estimate
        \begin{equation}\label{the preliminary lemma for you and I bound}
            \tnorm{1-\tp{1+\eta}^{-1}}_{\mathcal{H}^s}\lesssim\ep\tnorm{\eta}_{\mathcal{H}^s},
        \end{equation}
        where the implicit constant depends only on $s$ and $d$.
    \end{enumerate}
\end{lem}
\begin{proof}
    We define $\rho_{\m{WD}}$ to be sufficiently small so that for all $\eta\in B_{\mathcal{H}^{1+\tfloor{d/2}}}(0,\rho_{\m{WD}})$, we have that 
    \begin{equation}
        \max\tcb{\tnorm{\Uppi^1_{\m{L}}\eta}_{L^\infty},\tnorm{\Uppi^1_{\m{H}}\eta}_{L^\infty}}\le1/4,
    \end{equation}
    which is possible thanks to Proposition~\ref{proposition on frequency splitting} and the Sobolev embedding. This is the first item.

    To prove the second item, we first recall that $r_d$ is defined in~\eqref{the definition of rd} and Taylor expand 
    \begin{equation}
        1-\f{1}{1+\eta}=\sum_{\ell=1}^{r_d}(-\eta)^\ell+\int_0^1\f{(1-t)^{r_d-1}}{(r_d-1)!}\bp{\f{t\eta}{1+t\eta}-\f{\eta}{1+\eta}}\eta^{r_d}\;\m{d}t.
    \end{equation}
    We then estimate the $\mathcal{H}^s$ norm of the series on the right by using the second item of Corollary~\ref{coro on more algebra properties}.  In contrast, we estimate the integral term on the right in the $H^s$ norm, which is a valid strategy in light of the embedding $H^s\emb\mathcal{H}^s$ (see Proposition 5.3 in Leoni and Tice~\cite{leoni2019traveling}). To do so we use the third item of the aforementioned corollary, along with Corollary D.9 of Stevenson and Tice~\cite{stevenson2023wellposedness}.  Combining these yields \eqref{the preliminary lemma for you and I bound}.
\end{proof}

Armed with Lemma~\ref{the preliminary lemma for you and I}, we may now decompose a certain partial derivative of the entire nonlinear operator into a principal part and a perturbative remainder.

\begin{prop}[Derivative splitting]\label{prop on derivative splitting}
    Let $\N\ni s\ge\max\tcb{3,1+\tfloor{d/2}}$, $\ep\in(0,\rho_{\m{WD}}]$, $(\gam,\mu,\sig)\in\R^3$, $(v_0,\eta_0)\in B_{\X_{2+\tfloor{d/2}}}(0,\ep)\cap\X_{\infty}$, and $(v,\eta)\in\X_{s+3}$.  Let $\Uppsi:\R^3\times \X_{s+3}\to\Y_s$  or $\Uppsi:\R^3\times\Xs_{s+3}\to\Y_s$ be as in Definition~\ref{nonlinear map defs} or~\ref{subsonic nonlinear maps definition}, and write  $D_2\Uppsi$
    for its derivative with respect to the second factor (i.e. $\X_{s+3}$ or $\Xs_{s+2}$).  Then the following hold.
    \begin{enumerate}
        \item We have the splitting
    \begin{equation}\label{the second factor decomposition identity}
        D_2\Uppsi(\gam,\mu,\sig,v_0,\eta_0)[v,\eta]=P^{\gam,\mu,\sig}_{v_0,\eta_0}[v,\eta]+(0,R^{\gam,\mu,\sig}_{v_0,\eta_0}[v,\eta]),
    \end{equation}
    where $P_{v_0,\eta_0}^{\gam,\mu,\sig}[v,\eta]=(h,f)$ is given by
    \begin{equation}\label{first version of the principal part equations}
        \begin{cases}
            \grad\cdot((v_0-\gam e_1)\eta)+\grad\cdot((1+\eta_0)v)=h,\\
            (1+\eta_0)(v_0-\gam e_1)\cdot\grad v+v-\mu^2\grad\cdot((1+\eta_0)\S v)+(1+\eta_0)(1-\sig^2\Delta)\grad\eta=f,
        \end{cases}
    \end{equation}
    and the remainder piece $R^{\gam,\mu,\sig}_{v_0,\eta_0}[v,\eta]\in H^s(\R^d;\R^d)$ is 
    \begin{equation}\label{remainder formula}
        R^{\gam,\mu,\sig}_{v_0,\eta_0}[v,\eta]=\tp{(1+\eta_0)v+\eta(v_0-\gam e_1)}\cdot\grad v_0-\mu^2\grad\cdot(\eta\S v_0)+\eta(1-\sig^2\Delta)\grad\eta_0.
    \end{equation}
        \item We have the further decomposition:
    \begin{equation}\label{second splitting}
        P^{\gam,\mu,\sig}_{v_0,\eta_0}[v/(1+\eta_0),\eta]=Q^{\gam,\mu,\sig}_{v_0,\eta_0}[v,\eta]+(0,S^{\gam,\mu,\sig}_{v_0,\eta_0}[v]),
    \end{equation}
    where $Q^{\gam,\mu,\sig}_{v_0,\eta_0}[v,\eta]=(h,f)$ is given by
    \begin{equation}\label{second principal part equations}
        \begin{cases}
        \grad\cdot v+\grad\cdot((v_0-\gam e_1)\eta)=h,\\
        (v_0-\gam e_1)\cdot\grad v+v-\mu^2\grad\cdot\S v+(1+\eta_0)(1-\sig^2\Delta)\grad\eta=f,
        \end{cases}
    \end{equation}
    and the remainder $S^{\gam,\mu,\sig}_{v_0,\eta_0}[v]\in H^s(\R^d;\R^d)$ is
    \begin{multline}\label{remainder formula. 2}
        S^{\gam,\mu,\sig}_{v_0,\eta_0}[v]=-\tp{(1+\eta_0)^{-1}(v_0-\gam e_1)\cdot\grad\eta_0}v+\tp{(1+\eta_0)^{-1}-1}v\\+\mu^2\grad\cdot\tp{(1+\eta_0)^{-1}\tp{v\otimes\grad\eta_0+\grad\eta_0\otimes v+2v\cdot\grad\eta_0 I}}.
    \end{multline}
    \end{enumerate}
\end{prop}
\begin{proof}
    These follow from elementary calculations.
\end{proof}

 The following computation verifies that the remainder pieces obtained in Proposition~\ref{prop on derivative splitting} are indeed perturbative in the framework of tame estimates.

\begin{prop}[Remainder estimates]\label{prop on remainder estimates}

    Let $I\Subset\R^+$, $J,K\Subset\R$ be bounded, open, and non-empty intervals. Suppose that $s\in\N$, $\ep\in(0,\rho_{\m{WD}}]$, $(v_0,\eta_0)\in B_{\X_{3+\tfloor{d/2}}}(0,\ep)\cap\X_\infty$, and $(\gam,\mu,\sig)\in I\times J\times K=\mathcal{Q}$. The following hold.
    \begin{enumerate}
        \item The remainder $R^{\gam,\mu,\sig}_{v_0,\eta_0}$, defined via~\eqref{remainder formula}, extends to a bounded linear maps 
        \begin{equation}
            R^{\gam,\mu,\sig}_{v_0,\eta_0}:\;{_{\mu,\sig}}\X_s\to H^s(\R^d;\R^d),\quad R^{\gam,\mu,\sig}_{v_0,\eta_0}:{_{\mu,\sig}}\Xs_s\to H^s(\R^d;\R^d)
        \end{equation}
        that obeys the estimates
        \begin{equation}\label{remainder estimate}
            \tnorm{R^{\gam,\mu,\sig}_{v_0,\eta_0}[v,\eta]}_{H^s}\lesssim\ep\tnorm{v,\eta}_{{_{\mu,\sig}}\X_s}+\begin{cases}
                0&\text{if }s\le 1+\tfloor{d/2},\\
                \tnorm{v_0,\eta_0}_{\X_{s+3}}\tnorm{v,\eta}_{\X_{2+\tfloor{d/2}}}&\text{if }s>1+\tfloor{d/2},
            \end{cases}
        \end{equation}
        and
        \begin{equation}\label{is this estimate actually useful}
            \tnorm{R^{\gam,\mu,\sig}_{v_0,\eta_0}[v,\eta]}_{H^s}\lesssim\ep\tnorm{v,\eta}_{{_{\mu,\sig}}\Xs_s}+\begin{cases}
                0&\text{if }s\le 1+\tfloor{d/2},\\
                \tnorm{v_0,\eta_0}_{\Xs_{s+2}}\tnorm{v,\eta}_{\X_{2+\tfloor{d/2}}}&\text{if }s>1+\tfloor{d/2},
            \end{cases}
        \end{equation}
        \item The remainder $S^{\gam,\mu,\sig}_{v_0,\eta_0}$, defined via~\eqref{remainder formula. 2}, extends to a bounded linear maps
        \begin{equation}
            S^{\gam,\mu,\sig}_{v_0,\eta_0}:\;{_\mu}H^s(\R^d;\R^d)\to H^s(\R^d;\R^d),
        \end{equation}
        that obeys the estimates
        \begin{equation}\label{remainder estimate, 2}
            \tnorm{S^{\gam,\mu,\sig}_{v_0,\eta_0}[v]}_{H^s}\lesssim\ep\;{_\mu}\tnorm{v}_{H^s}+\begin{cases}
            0&\text{if }s\le 1+\tfloor{d/2},\\
                \tnorm{v_0,\eta_0}_{\X_{2+s}}\tnorm{v}_{H^{1+\tfloor{d/2}}}&\text{if }s>1+\tfloor{d/2}.
                \end{cases}
        \end{equation}
    \end{enumerate}
    The implicit constants in~\eqref{remainder estimate} and~\eqref{remainder estimate, 2} depend only on $s$, $\mathcal{Q}$, and $d$.
\end{prop}
\begin{proof}
    The estimates~\eqref{remainder estimate} and~\eqref{is this estimate actually useful} are granted by Corollary~\ref{coro on more algebra properties} and the second item of Proposition~\ref{corollary on tame estimates on simple multipliers}.  The only potentially delicate term is  $\mu^2\grad\cdot(\eta\S v_0)$ in the case of~\eqref{remainder estimate}, but here we have a derivative of $\eta$ appearing with a factor of $\mu^2$, and so the norm $_{\mu,\sig}\tnorm{\cdot}_{\mathcal{H}^s}$ controls $|\mu|\norm{\cdot}_{\mathcal{H}^{s+1}}$, which is more than enough.

    The estimate~\eqref{remainder estimate, 2} is equally simple; it follows readily from   Lemma~\ref{the preliminary lemma for you and I}, Corollary~\ref{coro on more algebra properties}, and the second item of Proposition~\ref{corollary on tame estimates on simple multipliers}.
\end{proof}

% _+__+_ -_+__+_ -_+__+_ -_+__+_ -_+__+_ -_+__+_ -_+__+_ -_+__+_ -_+__+_ -_+__+_ -_+__+_ -_+__+_ -_+__+_ -
\section{Linear analysis of the shallow water system}\label{section on linear analysis}
% _+__+_ -_+__+_ -_+__+_ -_+__+_ -_+__+_ -_+__+_ -_+__+_ -_+__+_ -_+__+_ -_+__+_ -_+__+_ -_+__+_ -_+__+_ -

This goal of this section is to verify the linear hypotheses of the Nash-Moser inverse function theorem applied to the mappings of Definitions~\ref{nonlinear map defs} and~\ref{subsonic nonlinear maps definition}. Our strategy is to study first the estimates and existence for a unified principal part of the linear systems. This is done in Sections~\ref{subsection on estimates} and~\ref{subsection on existence}. Then, in Section~\ref{subsection on synthesis}, we carefully combine the analysis of the principal part equations with the derivative splitting calculations of Section~\ref{section on derivative splitting} to reach our main theorems, which are Theorems~\ref{thm on synthesis, II} and~\ref{thm on syn 2 sub}, on the linear analysis of the derivatives of the omnisonic and subsonic nonlinear operators

% _+__+_ -_+__+_ -_+__+_ -_+__+_ -_+__+_ -_+__+_ -_+__+_ -_+__+_ -_+__+_ -_+__+_ -_+__+_ -_+__+_ -_+__+_ -
\subsection{Estimates}\label{subsection on estimates}
% _+__+_ -_+__+_ -_+__+_ -_+__+_ -_+__+_ -_+__+_ -_+__+_ -_+__+_ -_+__+_ -_+__+_ -_+__+_ -_+__+_ -_+__+_ -

This section develops a priori estimates for the following system of equations:
\begin{equation}\label{linear equations to be studied as part of the linear analysis}
    \begin{cases}
        \grad\cdot v+\grad\cdot(X\eta)=h,\\
        v+X\cdot\grad v-\mu^2\grad\cdot\S v+\tp{\vartheta-\sig^2\theta\Delta}\grad\eta=f.
    \end{cases}
\end{equation}
Here the given data are $f:\R^d\to\R^d$, $h:\R^d\to\R$, while the unknowns are $w:\R^d\to\R^d$, $\eta:\R^d\to\R$. The parameters in the problem are $X:\R^d\to\R^d$, $\vartheta,\theta:\R^d\to\R$, and $\gam,\mu,\sig\in\R$. The precise form of $X$, $\vartheta$, $\theta$, and $\gam,\mu,\sig$ is given by the following definition.

\begin{defn}[Parameters]\label{parameters definition}
In this subsection the tuneable parameters in system~\eqref{linear equations to be studied as part of the linear analysis} are to be taken of following form.
\begin{enumerate}
    \item We fix non-empty, open, and bounded intervals $I\Subset\R^+$, $J,K\Subset\R$ and enforce that
    \begin{equation}\label{the speed, viscosity, and surface tension}
        (\gam,\mu,\sig)\in I\times J\times K= \mathcal{Q} \subset\R^3
    \end{equation}
    Note that in Theorem~\ref{a priori estimates subsonic principal part} we impose a further subsonic restriction on $I$, that is $I\Subset(0,1)$.
    \item Let $\rho_{\m{WD}} \in \R^+$ be as in Lemma \ref{the preliminary lemma for you and I}. We let $\ep\in(0,\rho_{\m{WD}}]$, $r=\max\tcb{6,5+\tfloor{d/2}}$, and 
    \begin{equation}\label{conditions on the data for the linear analysis}
    (v_0,\eta_0,\be_0)\in \tp{B_{H^r}(0,\ep)\times B_{\mathcal{H}^r}(0,\ep)\times B_{H^r}(0,\ep)}\cap\tp{H^\infty(\R^d;\R^d)\times \mathcal{H}^\infty(\R^d)\times H^\infty(\R^d)}.
    \end{equation}

    \item   We shall take
    \begin{equation}
        X=v_0-\gam e_1,\quad\vartheta=1+\eta_0+\be_0,\quad\theta=1+\eta_0.
        \end{equation}
\end{enumerate}
\end{defn}

The following product estimate exploits integration by parts to crucially save on the number of derivatives appearing in the right hand side of~\eqref{lemma on integration by parts 0}.
\begin{lem}[A product estimate]\label{lemma on integration by parts}
We have the estimate
    \begin{equation}\label{lemma on integration by parts 0}
        \babs{\int_{\R^d}(\vartheta-\sig^2\theta\Delta)\eta\;\grad\cdot(X\eta)}\lesssim \ep\tp{\tnorm{\eta}^2_{\mathcal{H}^0}+\sig^2\tnorm{\eta}^2_{\mathcal{H}^1}}
    \end{equation}
    for all $\eta\in\mathcal{H}^2(\R^d)$ such that $0\not\in\m{supp}\mathscr{F}[\eta]$. The implicit constant depends only on $\mathcal{Q}$ and $d$ from Definition~\ref{parameters definition}.
\end{lem}
\begin{proof}
    Using Proposition \ref{proposition on frequency splitting}, we decompose $\eta=\eta_{\m{L}}+\eta_{\m{H}}$, for $\eta_{\m{L}}=\Uppi^1_{\m{L}}\eta \in C^\infty_0(\R^d)$ and  $\eta_{\m{H}}=\Uppi^1_{\m{H}}\eta \in H^r(\R^d)$, and expand
    \begin{equation}
        \int_{\R^d}(\vartheta-\sig^2\theta\Delta)\eta\;\grad\cdot(X\eta)=\sum_{(m,n)\in\tcb{\m{L},\m{H}}^2}\int_{\R^d}(\vartheta-\sig^2\theta\Delta)\eta_m\;\grad\cdot(X\eta_n).
    \end{equation}
    We will deal with each of these four terms separately.

    To handle the three of the four terms involving at least one copy of $\eta_{\m{L}}$, we will employ the cancellation
    \begin{multline}\label{lemma on integration by parts 1}
        \int_{\R^d}\tp{\vartheta-\sig^2\theta\Delta}\eta_m\;\grad\cdot(X\eta_n)=\int_{\R^d} (\vartheta-\sig^2\theta\Delta)\eta_m\;\grad\cdot(X\eta_n)+\gam(1-\sig^2\Delta)\eta_m\pd_1\eta_n\\
        =\int_{\R^d}\tp{\vartheta-1-\sig^2(\theta-1)\Delta}\eta_m\grad\cdot(X\eta_n)+(1-\sig^2\Delta)\eta_m\grad\cdot((X+\gam e_1)\eta_n),
    \end{multline}
    which actually holds for all $(m,n)\in\tcb{\m{L},\m{H}}^2$.  The $\m{L}-\m{L}$ term is estimated by the right side of \eqref{lemma on integration by parts 0} by writing
    \begin{multline}
        \int_{\R^d}\tp{\vartheta-\sig^2\theta\Delta}\eta_{\m{L}}\;\grad\cdot(X\eta_{\m{L}})=\int_{\R^d}\min\tcb{1,|\grad|}\tp{\vartheta-1-\sig^2(\theta-1)\Delta}\eta_{_{\m{L}}}\max\tcb{1,|\grad|^{-1}}\grad\cdot(X\eta_{\m{L}})\\+
        \int_{\R^d}\min\tcb{1,|\grad|}(1-\sig^2\Delta)\eta_{\m{L}}\max\tcb{1,|\grad|^{-1}}\grad\cdot((X+\gam e_1)\eta_{\m{L}}),
    \end{multline}
    and then using Cauchy-Schwarz and Corollary~\ref{coro on more algebra properties} on the right hand side's integrals. In the above we are using $\min\tcb{1,|\grad|}$ and $\max\tcb{1,|\grad|^{-1}}$ to denote the Fourier multiplication operators with symbols $\xi\mapsto\min\tcb{1,2\pi|\xi|}$ and $\xi\mapsto\max\tcb{1,(2\pi|\xi|)^{-1}}$, respectively.
    
    A similar strategy can be employed on the mixed $\m{L}-\m{H}$ and $\m{H}-\m{L}$ terms, but now we rearrange the right side of \eqref{lemma on integration by parts 1} to 
    ensure that all of the derivatives fall on $\eta_{\m{L}}$:
    \begin{multline}
        \int_{\R^d}\tp{\vartheta-\sig^2\theta\Delta}\eta_{\m{H}}\;\grad\cdot(X\eta_{\m{L}})=\int_{\R^d}\eta_{\m{H}}\tp{(\vartheta-1)\grad\cdot(X\eta_{\m{L}})-\sig^2\Delta\tp{(\theta-1)\grad\cdot(X\eta_{\m{L}})}}\\
        +\int_{\R^d}\eta_{\m{H}}(1-\sig^2\Delta)\grad\cdot((X+\gam e_1)\eta_{\m{L}})
    \end{multline}
    and
    \begin{multline}
        \int_{\R^d}(\vartheta-\sig^2\theta\Delta)\eta_{\m{L}} \grad\cdot(X\eta_{\m{H}})=-\int_{\R^d} X\cdot\grad\tp{\tp{\vartheta-1-\sig^2(\theta-1)\Delta}}\eta_{_{\m{L}}}\;\eta_{\m{H}}\\-
        \int_{\R^d} (X+\gam e_1)\cdot\grad\tp{(1-\sig^2\Delta)\eta_{\m{L}}}\;\eta_{\m{H}}.
    \end{multline}
    With these in hand, we may argue as above to bound them in terms of the right side of \eqref{lemma on integration by parts 0}.    
    
    It remains only to consider the $\m{H}-\m{H}$ term. Thanks to several integrations by parts, we have that
    \begin{multline}
        -\f12\int_{\R^d}(\grad\cdot X)\tp{\vartheta\eta_{\m{H}}^2+\sig^2\theta|\grad\eta_{\m{H}}|^2}=\int_{\R^d}\tp{\vartheta-\sig^2\theta\Delta}\eta_{\m{H}}\;\grad\cdot(X\eta_{\m{H}})+\f12X\cdot\tp{\grad\vartheta\eta_{\m{H}}^2+\sig^2\grad\theta|\grad\eta_{\m{H}}|^2}\\
        -\int_{\R^d}\vartheta(\grad\cdot X)\eta_{\m{H}}^2+\sig^2 [\grad\theta\cdot\grad\eta_{\m{H}}(X\cdot\grad)\eta_{\m{H}}+\theta(\grad\eta_{\m{H}}\otimes\grad \eta_{\m{H}}):\grad X+\grad\tp{(\grad\cdot X)\theta\eta_{\m{H}}}\cdot\grad\eta_{\m{H}}].
    \end{multline}
    We isolate the first term on the right hand side of this equality and then estimate each of the remaining terms using  Corollary~\ref{coro on more algebra properties} and Proposition~\ref{corollary on tame estimates on simple multipliers}.
\end{proof}

We now utilize the previous lemma in the derivation of the following estimate. The acronym $\m{AP}$ stands for `a priori' while the $0$ refers to the fact that $f$ and $h$ belong to $\Y_0$.
\begin{prop}[A priori estimates, I]\label{prop on first base case estimate}
Suppose that $f\in L^2(\R^d;\R^d)$, $h\in \tp{\dot{H}^{-1}\cap H^1}(\R^d)$, $v\in H^2(\R^d;\R^d)$, and $\eta\in\mathcal{H}^{3}(\R^d)$ with $0\not\in\m{supp}\mathscr{F}[\eta]$ solve system~\eqref{linear equations to be studied as part of the linear analysis}. There exists $\R^+\ni\rho_{\m{AP0}}\le\rho_{\m{WD}} $ with  the property that for all $\R^+\ni\ep\le\rho_{\m{AP0}}$ and $(v_0,\eta_0,\be_0)$ satisfying~\eqref{conditions on the data for the linear analysis} we have the estimates
    \begin{equation}\label{prop on first base case estimate 0}
        \tnorm{v}_{L^2}+|\mu|\tnorm{v}_{H^1}\lesssim 
        \tnorm{f}_{L^2}+\ep^{1/2}{_{\sig}}\tnorm{\eta}_{\mathcal{H}^0}+\babs{\int_{\R^d} (\vartheta-\sig^2\theta\Delta)\eta\; h}^{1/2}
    \end{equation}
and 
    \begin{equation}\label{the second version a priori estimates are here for your eyes and mine beholdeth, begone, and felled}
        |\mu|\tnorm{\grad v}_{L^2}+\mu^2\tnorm{\grad v}_{H^1}\lesssim\tnorm{f}_{L^2}+\ep^{1/2}{_{\sig}}\tnorm{\eta}_{\mathcal{H}^0}+|\mu|\ep^{1/2}{_{\sig}}\tnorm{\eta}_{\mathcal{H}^1}  +|\mu|\babs{\int_{\R^d}(\vartheta-\sig^2\theta\Delta)\grad\eta\cdot\grad h}^{1/2}.
    \end{equation}
   The implicit constants in~\eqref{prop on first base case estimate 0} and~\eqref{the second version a priori estimates are here for your eyes and mine beholdeth, begone, and felled}, depend only on $\mathcal{Q}$ and $d$ from Definition~\ref{parameters definition}.
\end{prop}
\begin{proof}
We begin with the proof of \eqref{prop on first base case estimate 0}.  We take the dot-product of the second equation in~\eqref{linear equations to be studied as part of the linear analysis} with $v$ and integrate by parts:
\begin{equation}\label{prop on first base case estimate 1}
\int_{\R^d} \abs{v}^2 \bp{1- \f{\grad\cdot X}{2}}  + \mu^2 \int_{\R^d} \frac{1}{2} \abs{\mathbb{D} v}^2 +  2\abs{\nabla \cdot v}^2    = \int_{\R^d} f\cdot v - \tp{\vartheta-\sig^2\theta\Delta}\grad\eta\cdot v.
\end{equation}
Taking $\rho_{\m{AP0}}$ sufficiently small, employing the Korn inequality in $\R^d$, and recalling that we allow implicit constants to depend on $\mathcal{Q}$, we have that 
\begin{equation}
    \tnorm{v}_{L^2}^2+\mu^2\tnorm{v}^2_{H^1} \lesssim \tnorm{v}_{L^2}^2+\mu^2\tnorm{\grad v}^2_{L^2}  \lesssim \int_{\R^d} \abs{v}^2 \bp{1- \f{\grad\cdot X}{2}}  + \mu^2 \int_{\R^d} \frac{1}{2} \abs{\mathbb{D} v}^2 +  2\abs{\nabla \cdot v}^2 .
\end{equation}
It remains only to handle the terms on the right side of \eqref{prop on first base case estimate 1}.  
The $f$ term is trivially handled with Cauchy-Schwarz and an absorbing argument.   For the $\eta$-term, we introduce a commutator and integrate by parts:
\begin{equation}
    \int_{\R^d}\tp{\vartheta-\sig^2\theta\Delta}\grad\eta\cdot v=-\int_{\R^d}\tp{\vartheta-\sig^2\theta\Delta}\eta\;\grad\cdot v+\int_{\R^d}[(\vartheta-\sig^2\theta\Delta),\grad]\eta\cdot v.
\end{equation}
The first equation in~\eqref{linear equations to be studied as part of the linear analysis} and Lemma~\ref{lemma on integration by parts} yield the bound
\begin{equation}
    \babs{\int_{\R^d}\tp{\vartheta-\sig^2\theta\Delta}\eta\;\grad\cdot v}\le \ep\tp{\tnorm{\eta}_{\mathcal{H}^0}^2+\sig^2\tnorm{\eta}^2_{\mathcal{H}^1}}+\babs{\int_{\R^d} (\vartheta-\sig^2\theta\Delta)\eta\; h}.
\end{equation}
On the other hand, thanks to Corollary~\ref{coro on more algebra properties}, we have that
\begin{equation}
    \babs{\int_{\R^d}[(\vartheta-\sig^2\theta\Delta),\grad]\eta\cdot v}=\babs{\int_{\R^d}\grad\vartheta\eta\cdot v-\sig^2\grad\theta\Delta\eta\cdot v}\lesssim\ep\tp{\tnorm{\eta}_{\mathcal{H}^0}+\sig^2\tnorm{\eta}_{\mathcal{H}^2}}\tnorm{v}_{L^2}.
\end{equation}
Synthesizing these with \eqref{prop on first base case estimate 1} then yields \eqref{prop on first base case estimate 0}.

Next, we turn to the proof of \eqref{the second version a priori estimates are here for your eyes and mine beholdeth, begone, and felled}.  We will prove this initially under the stronger assumption that $v\in H^3(\R^d;\R^d)$ and $\eta\in\mathcal{H}^4(\R^d)$, and then we will conclude the proof by showing how to reduce the general case to this special case.

We dot the second equation in~\eqref{linear equations to be studied as part of the linear analysis} with $-\mu^2\Delta v$ and integrate to see that 
\begin{equation}\label{prop on second base case estimate 1}
-\int_{\R^d}\tp{v-\mu^2\grad\cdot\S v}\cdot \mu^2\Delta v =   \mu^2 \int_{\R^d} -f \cdot \Delta v + (X\cdot\grad) v\cdot\Delta v + \tp{\vartheta-\sig^2\theta\Delta}\grad\eta\cdot \Delta v.
\end{equation}
We will deal with the left and right sides separately.  Integrating by parts and using Korn's inequality grants the coercive estimate for the left side:
\begin{equation}
    \mu^2\tnorm{\grad v}_{L^2}^2 +\mu^4\tnorm{\grad v}_{H^1}^2 \lesssim -\int_{\R^d}\tp{v-\mu^2\grad\cdot\S v}\cdot \mu^2\Delta v.
\end{equation}

Now we estimate the right side of \eqref{prop on second base case estimate 1}.  The $f$ term is easily handled with Cauchy-Schwarz and an absorbing argument since $\tnorm{\Delta v}_{L^2} \lesssim \tnorm{\grad v}_{H^1}$.  The $X\cdot\grad v$ term controlled by integrating by parts:
\begin{equation}
    \babs{\mu^2\int_{\R^d} (X\cdot\grad) v\cdot\Delta v}=\babs{-\mu^2\int_{\R^d}\f{\grad\cdot X}{2}|\grad v|^2+\mu^2\sum_{j=1}^d \int_{\R^d}(\pd_j X\cdot\grad)v\cdot\pd_j v}\lesssim\ep\mu^2\tnorm{\grad v}^2_{L^2}.
\end{equation}

The $\eta$-term in \eqref{prop on second base case estimate 1} is more involved.  We have that
\begin{multline}\label{prop on second base case estimate 2}
    -\mu^2\int_{\R^d}\tp{\vartheta-\sig^2\theta\Delta}\grad\eta\cdot \Delta v
    =\mu^2 \sum_{j=1}^d\int_{\R^d}\tp{\pd_j\vartheta-\sig^2\pd_j\theta\Delta}\grad\eta\cdot\pd_j v\\
    -\mu^2\sum_{j=1}^d \int_{\R^d}\grad\vartheta\pd_j\eta\cdot\pd_jv-\sig^2\grad\theta\Delta\pd_j\eta\cdot\pd_j v-\mu^2 \sum_{j=1}^d\int_{\R^d}\tp{\vartheta-\sig^2\theta\Delta}\pd_j\eta\;\grad\cdot\pd_j v.
\end{multline}
The first two terms on the right here are straightforward to estimate:
\begin{equation}
    \mu^2 \sum_{j=1}^d \babs{\int_{\R^d}\tp{\pd_j\vartheta-\sig^2\pd_j\theta\Delta}\grad\eta\cdot\pd_j v +  \int_{\R^d}\tp{\grad\vartheta\pd_j\eta-\sig^2\grad\theta\Delta\eta}\cdot\pd_jv}\lesssim\ep|\mu|\tp{\tnorm{\eta}_{\mathcal{H}^1}+\sig^4\tnorm{\eta}_{\mathcal{H}^3}}\cdot|\mu|\tnorm{\grad v}_{L^2}.
\end{equation}
For the third term on the right side of \eqref{prop on second base case estimate 2}, we use the identity $\grad\cdot\pd_j v+\grad\cdot(X\pd_j\eta)=\pd_jh-\grad\cdot(\pd_jX\eta)$ and appeal to Lemma~\ref{lemma on integration by parts} to bound
\begin{equation}
    \mu^2\babs{\sum_{j=1}^d \int_{\R^d}\tp{\vartheta-\sig^2\theta\Delta}\pd_j\eta\;\grad\cdot\pd_j v}\lesssim\mu^2\babs{\int_{\R^d}(\vartheta-\sig^2\theta\Delta)\grad\eta\cdot\grad h}+\ep\mu^2\tp{\tnorm{\eta}_{\mathcal{H}^1}^2+\sig^4\tnorm{\eta}^2_{\mathcal{H}^3}}.
\end{equation}

Upon synthesizing these estimates for the various terms in \eqref{prop on second base case estimate 1}, we arrive at the bound
\begin{multline}
    \mu^2\tnorm{\grad v}_{L^2}^2+\mu^4\tnorm{\grad v}_{H^1}^2\lesssim\tnorm{f}_{L^2}^2+\ep\mu^2\tnorm{\grad v}_{L^2}^2+\ep|\mu|\tp{\tnorm{\eta}_{\mathcal{H}^1}+\sig^2\tnorm{\eta}_{\mathcal{H}^3}}\cdot|\mu|\tnorm{\grad v}_{L^2}\\
    +\mu^2\babs{\int_{\R^d}\tp{\vartheta-\sig^2\theta\Delta}\grad\eta\cdot\grad h}+\ep\mu^2\tp{\tnorm{\eta}^2_{\mathcal{H}^1}+\sig^4\tnorm{\eta}_{\mathcal{H}^{3}}^2}.
\end{multline}
Using an absorbing argument and taking $\rho_{\m{AP0}}$ small enough, we deduce from this that \eqref{the second version a priori estimates are here for your eyes and mine beholdeth, begone, and felled} holds  in this special case of additional smoothness assumptions.

To conclude, we now explain how to reduce the general case of \eqref{the second version a priori estimates are here for your eyes and mine beholdeth, begone, and felled} to the special case.  Suppose, then, that $v\in H^2(\R^d;\R^d)$ and $\eta\in\mathcal{H}^3(\R^d)$ with $0\not\in\m{supp}\mathscr{F}[\eta]$. Given any $\del\in\R^+$ there exists $v_\del\in H^3(\R^d;\R^d)$ and $\eta_\del\in\mathcal{H}^4(\R^d)$ with $\m{dist}(0,\m{supp}\mathscr{F}[\eta_\del])=\m{dist}(0,\m{supp}\mathscr{F}[\eta])$ and $\tnorm{v-v_\del}_{H^2}<\del$, $\tnorm{\eta-\eta_\del}_{\mathcal{H}^3}<\del$. The equations satisfied by $v_\del$ and $\eta_\del$ are~\eqref{linear equations to be studied as part of the linear analysis} with right hand side data
\begin{equation}
    h_\del=h-\grad\cdot(v-v_\del)-\grad\cdot(X(\eta-\eta_\del))
\end{equation}
and
\begin{equation}
    f_\del=f-\tp{1+X\cdot\grad-\mu^2\grad\cdot\S}(v-v_\del)-\tp{\vartheta-\sig^2\theta\Delta}\grad(\eta-\eta_\del).
\end{equation}
We invoke the a priori estimate from the special case to deduce that 
    \begin{multline}\label{the delta is a fan of the book that was indeed written for it to read delta the}
          |\mu|\tnorm{\grad v_\delta}_{L^2}^2+\mu^2\tnorm{\grad v_\delta}_{H^1}
        \lesssim\tnorm{f_\del}_{L^2}+\ep^{1/2}\tp{\tnorm{\eta_\del}_{\mathcal{H}^0}+\sig^2\tnorm{\eta_\del}_{\mathcal{H}^2}}\\
        +|\mu|\ep^{1/2}\tp{\tnorm{\eta_\del}_{\mathcal{H}^1}+\sig^2\tnorm{\eta_\del}_{\mathcal{H}^3}}
         +|\mu|\babs{\int_{\R^d}(\vartheta-\sig^2\theta\Delta)\grad\eta_\del\cdot\grad h_\del}^{1/2}.
    \end{multline}
    We now send $\del\to 0$ and observe that $f_\del\to f$ in $L^2(\R^d;\R^d)$, $h_\del\to h$ in $(\dot{H}^{-1}\cap H^1)(\R^d)$, $v_\del\to v$ in $H^2(\R^d;\R^d)$, and $\eta_\del\to\eta$ in $\mathcal{H}^3(\R^d)$ (actually in $H^3(\R^d)$). By passing to the limit in~\eqref{the delta is a fan of the book that was indeed written for it to read delta the}, we thus acquire~\eqref{the second version a priori estimates are here for your eyes and mine beholdeth, begone, and felled} in the general case, which completes the proof.
\end{proof}

Our next result obtains closed (when $h=0$) a priori estimates at a base case level of derivatives.

\begin{thm}[A priori estimates, II]\label{thm on base case a priori estimate}
    Suppose that $f\in H^1(\R^d;\R^d)$, $h\in\tp{\dot{H}^{-1}\cap H^2}(\R^d)$, $v\in H^3(\R^d;\R^d)$, and $\eta\in\mathcal{H}^4(\R^d)$ with $0\not\in\m{supp}\mathscr{F}[\eta]$ solve system~\eqref{linear equations to be studied as part of the linear analysis}.  Let $\rho_{\m{AP0}}\in \R^+$ be as in Proposition \ref{prop on first base case estimate}.  Then there exists $\R^+\ni\rho_{\m{AP1}} \le \rho_{\m{AP0}}$ with the property that for all $\R^+\ni\ep\le\rho_{\m{AP1}}$ and $(v_0,\eta_0,\be_0)$ satisfying~\eqref{conditions on the data for the linear analysis}, we have the estimate
    \begin{multline}\label{the base case a priori estimate is here, today and its gone}
        {_{\mu}}\tnorm{v}_{H^1}+{_{\mu,\sig}}\tnorm{\eta}_{\mathcal{H}^1} \lesssim \tnorm{f}_{H^1}+\babs{\int_{\R^d}\tp{\vartheta-\sig^2\theta\Delta}\eta\;h}^{1/2}\\+\sum_{j=1}^d\babs{\int_{\R^d}\tp{\vartheta-\sig^2\theta\Delta}\pd_j\eta\;\pd_j h}^{1/2}+|\mu|\sum_{j=1}^d\babs{\int_{\R^d}\tp{\vartheta-\sig^2\Delta}\grad\pd_j\eta\cdot\grad\pd_jh}^{1/2}+|\sig|\tnorm{\grad h}_{L^2}.
    \end{multline}
    The implicit constant in~\eqref{the base case a priori estimate is here, today and its gone} depends only on $\mathcal{Q}$ and $d$ from Definition~\ref{parameters definition}.
\end{thm}
\begin{proof}
For $j \in \{1,\dotsc,d\}$ we apply $\partial_j$ to~\eqref{linear equations to be studied as part of the linear analysis} to obtain the system
\begin{equation}
    \begin{cases}
            \grad\cdot \pd_jv+\grad\cdot(X\pd_j\eta)=\pd_jh-\grad\cdot(\pd_jX\eta),\\
            \pd_jv+X\cdot\grad \pd_jv-\mu^2\grad\cdot\S\pd_jv+(\vartheta-\sig^2\theta\Delta)\grad\pd_j\eta=\pd_jf-\pd_jX\cdot\grad v-\tp{\pd_j\vartheta-\sig^2\pd_j\theta\Delta}\grad\eta.
        \end{cases}
\end{equation}
To each of these we apply the a priori estimates of Proposition~\ref{prop on first base case estimate} and then sum over $j$; we then sum this with the estimate for the non-differentiated problem provided by the same proposition.  This results in the bound
\begin{multline}\label{thm on base case a priori estimate 1}
    \tnorm{v}_{H^1}+\mu^2\tnorm{v}_{H^3}\lesssim\tnorm{f}_{H^1}+\ep^{1/2}\tp{\tnorm{v}_{H^1}+\tnorm{\eta}_{\mathcal{H}^1}+\sig^2\tnorm{\eta}_{\mathcal{H}^3}+|\mu|\tp{\tnorm{\eta}_{\mathcal{H}^2}+\sig^2\tnorm{\eta}_{\mathcal{H}^4}}}\\
    + \sum_{j=1}^d\babs{\int_{\R^d}\tp{\vartheta-\sig^2\theta\Delta}\pd_j\eta\;\grad\cdot(\pd_jX\eta)}^{1/2}+|\mu|\sum_{j=1}^d\babs{\int_{\R^d}\tp{\vartheta-\sig^2\Delta}\grad\pd_j\eta\cdot\grad\grad\cdot(\pd_jX\eta)}^{1/2}+R_{\eta,h},
\end{multline}
where
\begin{equation}
    R_{\eta,h}=\babs{\int_{\R^d}\tp{\vartheta-\sig^2\theta\Delta}\eta\;h}^{1/2}+\sum_{j=1}^d\babs{\int_{\R^d}\tp{\vartheta-\sig^2\theta\Delta}\pd_j\eta\;\pd_j h}^{1/2}+|\mu|\sum_{j=1}^d\babs{\int_{\R^d}\tp{\vartheta-\sig^2\Delta}\grad\pd_j\eta\cdot\grad\pd_jh}^{1/2}.
\end{equation}
It is straightforward to estimate
\begin{multline}
    \sum_{j=1}^d\babs{\int_{\R^d}\tp{\vartheta-\sig^2\theta\Delta}\pd_j\eta\;\grad\cdot(\pd_jX\eta)}^{1/2}+|\mu|\sum_{j=1}^d\babs{\int_{\R^d}\tp{\vartheta-\sig^2\Delta}\grad\pd_j\eta\cdot\grad\grad\cdot(\pd_jX\eta)}^{1/2} \\
    \lesssim \ep^{1/2}\tp{\tnorm{\eta}_{\mathcal{H}^1}+\sig^2\tnorm{\eta}_{\mathcal{H}^3}}+\ep^{1/2}|\mu|\tp{\tnorm{\eta}_{\mathcal{H}^2}+\sig^2\tnorm{\eta}_{\mathcal{H}^4}}.
\end{multline}
Thus, if $\rho_{\m{AP1}}$ is small enough, this and \eqref{thm on base case a priori estimate 1} yield the bound
\begin{equation}\label{thm on base case a priori estimate 2}
    \tnorm{v}_{H^1}+\mu^2\tnorm{v}_{H^3}\lesssim 
    \tnorm{f}_{H^1}+\ep^{1/2}\tp{\tnorm{\eta}_{\mathcal{H}^1}+\sig^2\tnorm{\eta}_{\mathcal{H}^3}+|\mu|\tp{\tnorm{\eta}_{\mathcal{H}^2}+\sig^2\tnorm{\eta}_{\mathcal{H}^4}}}+R_{\eta,h}.
\end{equation}

We now turn our attention to higher regularity estimates for $\eta$.  We know that $\grad \eta$ satisfies the equation 
\begin{equation}\label{thm on base case a priori estimate 3}
    \tp{\vartheta-\sig^2\theta\Delta}\grad\eta=f-v-X\cdot\grad v+\mu^2\grad\cdot\S v,
\end{equation}
and, decreasing $\rho_{\m{AP1}}$ if necessary, we may assume that the operator $\vartheta-\sig^2\theta\Delta$ is uniformly elliptic.  Sobolev interpolation allows us to bound 
\begin{equation}
    \tnorm{X\cdot\grad v}_{L^2}+|\mu|\tnorm{X\cdot\grad v}_{H^1}
    \lesssim
    \tnorm{v}_{H^1} + |\mu| \tnorm{v}_{H^2} \lesssim \tnorm{v}_{H^1} + |\mu| \tnorm{v}_{H^1}^{1/2} \tnorm{v}_{H^3}^{1/2}
    \lesssim
    \tnorm{v}_{H^1}+\mu^2\norm{v}_{H^3},
\end{equation}
and hence we may employ standard elliptic regularity on \eqref{thm on base case a priori estimate 3} to obtain the bound
\begin{equation}
    \tnorm{\grad\eta}_{L^2}+\sig^2\tnorm{\grad\eta}_{H^2}+|\mu|\tp{\tnorm{\grad\eta}_{H^1}+\sig^2\tnorm{\grad\eta}_{H^3}}\lesssim\tnorm{f}_{H^1}+\tnorm{v}_{H^1}+\mu^2\tnorm{v}_{H^3}.
\end{equation}
On the other hand, from the identity $\grad\cdot v+\grad\cdot(X\eta)=0$, we find that
\begin{equation}
    |\gam|\tnorm{\mathcal{R}_1\eta}_{L^2}\lesssim\tnorm{v}_{L^2}+\ep\tnorm{\eta}_{\mathcal{H}^0}.
\end{equation}
Combining these two estimates with the norm equivalence in the third item of Proposition \ref{proposition on spatial characterization of anisobros} and the estimate \eqref{thm on base case a priori estimate 2} then shows that 
\begin{multline}\label{thm on base case a priori estimate 4}
    \tnorm{\eta}_{\mathcal{H}^1}+\sig^2\tnorm{\eta}_{\mathcal{H}^3}+|\mu|\tp{\tnorm{\eta}_{\mathcal{H}^2}+\sig^2\tnorm{\eta}_{\mathcal{H}^4}}\lesssim \ep\tnorm{\eta}_{\mathcal{H}^0} + \tnorm{f}_{H^1}+\tnorm{v}_{H^1}+\mu^2\tnorm{v}_{H^3} \\
  \lesssim 
    \tnorm{f}_{H^1}
    +  \ep\tnorm{\eta}_{\mathcal{H}^0}  
    +\ep^{1/2}\tp{\tnorm{\eta}_{\mathcal{H}^1}+\sig^2\tnorm{\eta}_{\mathcal{H}^3}+|\mu|\tp{\tnorm{\eta}_{\mathcal{H}^2}+\sig^2\tnorm{\eta}_{\mathcal{H}^4}}}+R_{\eta,h}.
\end{multline}

Summing \eqref{thm on base case a priori estimate 2} and \eqref{thm on base case a priori estimate 4} and taking $\rho_{\m{AP1}}$ sufficiently small, we may absorb the $\eta$   terms on the right with those on the left, resulting in the estimate
\begin{equation}\label{somewhere, beyond the sea}
    \tnorm{v}_{H^1}+\mu^2\tnorm{v}_{H^3} +
    \tnorm{\eta}_{\mathcal{H}^1}+\sig^2\tnorm{\eta}_{\mathcal{H}^3}+|\mu|\tp{\tnorm{\eta}_{\mathcal{H}^2}+\sig^2\tnorm{\eta}_{\mathcal{H}^4}}
        \lesssim
    \tnorm{f}_{H^1} + R_{\eta,h}.
\end{equation}
Thus, to finish the proof of~\eqref{the base case a priori estimate is here, today and its gone} we only need to prove the estimate
\begin{equation}\label{thm on base case a priori estimate 5}
    |\sig| \tp{\tnorm{\eta}_{\mathcal{H}^2}+\sig^2\tnorm{\eta}_{\mathcal{H}^4}} \lesssim     \tnorm{f}_{H^1} + R_{\eta,h} +|\sig|\tnorm{\grad h}_{L^2}.
\end{equation}
 To this end, we take the divergence of the second equation in~\eqref{linear equations to be studied as part of the linear analysis} and substitute in the first equation to derive the identity
\begin{multline}\label{the eta equation strikes back}
(\vartheta - \sig^2\theta\Delta)\Delta \eta =(X\cdot\grad) \grad\cdot(X\eta)+\grad\cdot f-\grad X^{\m{t}}:\grad v-(1-4\mu^2\Delta)\grad\cdot v-X\cdot\grad h \\
- (\grad \vartheta \cdot \grad \eta - \sigma^2 \grad \theta \cdot \grad \Delta \eta).
\end{multline}
Then we may again use elliptic regularity to bound 
\begin{multline}
    |\sig|\tp{\tnorm{\grad^2\eta}_{L^2}+\sig^2\tnorm{\grad^4\eta}_{L^2}}\lesssim|\sig|\tnorm{(X\cdot\grad)\grad\cdot(X\eta)+\grad\cdot f-\grad X^{\m{t}}:\grad v-(1-4\mu^2\Delta)\grad\cdot v-X\cdot\grad h}_{L^2} \\
    + |\sig| \tnorm{\grad \vartheta \cdot \grad \eta - \sigma^2 \grad \theta \cdot \grad \Delta \eta}_{L^2},
\end{multline}
and from here we can use Sobolev interpolation again (but now with powers of $\sigma$) together with \eqref{somewhere, beyond the sea} to deduce \eqref{thm on base case a priori estimate 5}.
\end{proof}

We can now use the previous result with a simple differentiation argument to derive a priori estimates for higher derivatives of solutions to~\eqref{linear equations to be studied as part of the linear analysis}.

\begin{thm}[Principal part a priori estimates, omnisonic case]\label{fourth theorem on a priori estimates}
Suppose that $\N\ni s\ge 2+\tfloor{d/2}$, $f\in H^s(\R^d;\R^d)$, $h\in\tp{\dot{H}^{-1}\cap H^{1+s}}(\R^d)$, $v\in H^{2+s}(\R^d;\R^d)$, and $\eta\in\mathcal{H}^{3+s}(\R^d)$ solve system~\eqref{linear equations to be studied as part of the linear analysis}. There exists $\R^+\ni \rho_{\m{AP}s}\le\rho_{\m{WD}} $, depending only on $\mathcal{Q}$, $\rho_{\m{WD}}$, $r$ (from Definition~\ref{parameters definition}) and $s$, with the property that for all $\R^+\ni\ep\le\rho_{\m{AP}s}$ and $(v_0,\eta_0,\be_0)$ satisfying~\eqref{conditions on the data for the linear analysis}, we have the estimate
\begin{equation}\label{final a priori estimate question mark}
    \tnorm{v,\eta}_{{_{\mu,\sig}}\X_s}\lesssim\tnorm{h,f}_{\Y_s}+\tnorm{v_0,\eta_0,\be_0}_{\X_{1+s}\times H^{1+s}}\tnorm{h,f}_{\Y_{2+\tfloor{d/2}}}
\end{equation}
The implicit constant in~\eqref{final a priori estimate question mark} depends only on $s$, $\mathcal{Q}$, and $d$ from Definition~\ref{parameters definition}.
\end{thm}
\begin{proof}
    We divide the proof into two stages.  In the first of these we consider only the special case in which $h=0$ and $0\not\in\m{supp}\mathscr{F}[\eta]$.  An application of Theorem~\ref{thm on base case a priori estimate} yields the estimate
    \begin{equation}\label{e1}
        \tnorm{v,\eta}_{{_{\mu,\sig}}\X_1}\lesssim\tnorm{f}_{H^1}.
    \end{equation}
    In order to obtain higher-order estimates we will apply this same a priori estimate to the derivatives of~\eqref{linear equations to be studied as part of the linear analysis}; it will suffice to consider only pure derivatives of order $s-1$.  To this end, we let $q=s-1$, $j\in\tcb{1,\dots,d}$, and consider the equations satisfied by $\pd_j^qv$ and $\pd_j^q\eta$:
    \begin{equation}\label{the higher order differentiated equation}
        \begin{cases}
            \grad\cdot\pd_j^qv+\grad\cdot(X\pd_j^q\eta)=\grad\cdot([X,\pd_j^q]\eta),\\
            \pd_j^q v+X\cdot\pd_j^qv-\mu^2\grad\cdot\S\pd_j^qv+(\vartheta-\sig^2\theta\Delta)\grad\pd_j^q\eta=\pd_j^q f+[X,\pd_j^q]\cdot\grad v+\tp{[\vartheta,\pd_j^q]-\sig^2[\theta,\pd_j^q]\Delta}\grad\eta.
        \end{cases}
    \end{equation}
    Using Theorem~\ref{thm on base case a priori estimate} again, we acquire the estimate
    \begin{multline}\label{stage three asparagus}
        {_\mu}\tnorm{\pd_j^q v}_{H^1}+{_{\mu,\sig}}\tnorm{\pd_j^q\eta}_{\mathcal{H}^1}\lesssim\tnorm{\pd_j^q f+[X,\pd_j^q]\cdot\grad v+\tp{[\vartheta,\pd_j^q]-\sig^2[\theta,\pd_j^q]\Delta}\grad\eta}_{H^1}\\
        +\sum_{k=1}^d\babs{\int_{\R^d}\tp{\vartheta-\sig^2\theta\Delta}\pd_k\pd_j^{q}\eta\pd_k\grad\cdot([X,\pd_j^q]\eta)}^{1/2}+|\mu|\sum_{k=1}^d\babs{\int_{\R^d}(\vartheta-\sig^2\theta\Delta)\grad\pd_k\pd_j^q\eta\cdot\grad\pd_k\grad\cdot([X,\pd_j^q]\eta)}^{1/2}\\
        +|\sig|\tnorm{\grad\grad\cdot([X,\pd_j^q]\eta)}_{L^2}.
    \end{multline}
    The rest of our work in this stage consists of obtaining structured estimates for the terms appearing on the right here.
    
    For the first term on the right hand side of~\eqref{stage three asparagus} we utilize the tame estimates on commutators from Proposition~\ref{prop tame commutator estimates} to  bound
    \begin{multline}\label{e2}
        \tnorm{\pd_j^q f+[X,\pd_j^q]\cdot\grad v+\tp{[\vartheta,\pd_j^q]-\sig^2[\theta,\pd_j^q]\Delta}\grad\eta}_{H^1}\lesssim\tnorm{f}_{H^s}+\ep\tp{\tnorm{v}_{H^s}+{_\sig}\tnorm{\eta}_{\mathcal{H}^s}}\\
        +{\tnorm{v_0,\be_0,\eta_0}_{H^{s}\times H^{s}\times\mathcal{H}^{s}}}\tp{\tnorm{v}_{H^{2+\tfloor{d/2}}}+{_\sig}\tnorm{\eta}_{\mathcal{H}^{2+\tfloor{d/2}}}}.
    \end{multline}
    For the remaining terms on the right hand side of~\eqref{stage three asparagus}, we first claim that the following preliminary estimate holds:
    \begin{equation}\label{e3}
        \tnorm{[X,\pd_j^q]\eta}_{H^2}\lesssim\ep\tnorm{\eta}_{\mathcal{H}^s}+\tnorm{v_0}_{H^{s+1}}\tnorm{\eta}_{\mathcal{H}^{1+\tfloor{d/2}}}.
    \end{equation} To see this, we note that $[X,\pd_j^q]=[v_0,\pd_j^q]$ and employ the Fourier localization operators from \eqref{notation for the Fourier projection operators} to bound
    \begin{equation}\label{e4}
        \tnorm{[X,\pd_j^q]\eta}_{H^2}\le\tnorm{[v_0,\pd_j^q]\Uppi^1_{\m{H}}\eta}_{H^2}+\tnorm{\pd_j^q(v_0\Uppi^1_{\m{L}}\eta)}_{H^2}+\tnorm{v_0\pd_j^q\Uppi^1_{\m{L}}\eta}_{H^2}.
    \end{equation}
    Propositions~\ref{proposition on frequency splitting}, \ref{corollary on tame estimates on simple multipliers}, and~\ref{prop tame commutator estimates}, together with Corollary~\ref{coro on more algebra properties} then allow us to estimate 
    \begin{equation}
        \tnorm{[v_0,\pd_j^q]\Uppi^1_{\m{H}}\eta}_{H^2}+\tnorm{\pd_j^q(v_0\Uppi^1_{\m{L}}\eta)}_{H^2}+\tnorm{v_0\pd_j^q\Uppi^1_{\m{L}}\eta}_{H^2} 
        \lesssim
        \ep\tnorm{\eta}_{\mathcal{H}^s}+\tnorm{v_0}_{H^{s+1}}\tnorm{\eta}_{\mathcal{H}^{1+\tfloor{d/2}}},
    \end{equation}
    and so we obtain \eqref{e3} from this and \eqref{e4}, completing the proof of the claim.   With the claim in hand, we deduce that
    \begin{multline}\label{e5}
        \sum_{k=1}^d\babs{\int_{\R^d}\tp{\vartheta-\sig^2\theta\Delta}\pd_k\pd_j^{q}\eta\pd_k\grad\cdot([X,\pd_j^q]\eta)}^{1/2}+|\mu|\sum_{k=1}^d\babs{\int_{\R^d}(\vartheta-\sig^2\theta\Delta)\grad\pd_k\pd_j^q\eta\cdot\grad\pd_k\grad\cdot([X,\pd_j^q]\eta)}^{1/2}
        \\+|\sig|\tnorm{\grad\grad\cdot([X,\pd_j^q]\eta)}_{L^2}\lesssim\;{_{\mu,\sig}}\tnorm{\eta}_{\mathcal{H}^s}^{1/2}\tnorm{[X,\pd_j^q]\eta}_{H^2}^{1/2}+|\sig|\tnorm{[X,\pd_j^q]\eta}_{H^2}.
    \end{multline}
    We may thus combine~\eqref{e1}, \eqref{stage three asparagus}, \eqref{e2}, \eqref{e3}, and \eqref{e5}  to arrive at the bound
    \begin{multline}\label{e6}
        {_\mu}\tnorm{v}_{H^s}+{_{\mu,\sig}}\tnorm{\eta}_{\mathcal{H}^s}
        \asymp
        {_\mu}\tnorm{v}_{H^1}+{_{\mu,\sig}}\tnorm{\eta}_{\mathcal{H}^1}
        + 
        \sum_{j=1}^d ({_\mu}\tnorm{\pd_j^{s-1} v}_{H^1}+{_{\mu,\sig}}\tnorm{\pd_j^{s-1}\eta}_{\mathcal{H}^1})
        \lesssim
        \tnorm{f}_{H^s}+\ep\tp{\tnorm{v}_{H^s}+{_\sig}\tnorm{\eta}_{\mathcal{H}^s}}\\+{\tnorm{v_0,\eta_0,\be_0}_{\X_{s+1}\times H^{s+1}}}\tp{\tnorm{v}_{H^{2+\tfloor{d/2}}}+{_\sig}\tnorm{\eta}_{\mathcal{H}^{2+\tfloor{d/2}}}}
        +{_{\mu,\sig}}\tnorm{\eta}_{\mathcal{H}^s}^{1/2}\tp{\ep\tnorm{\eta}_{\mathcal{H}^s}+\tnorm{v_0}_{H^{s+1}}\tnorm{\eta}_{\mathcal{H}^{1+\tfloor{d/2}}}}^{1/2}.
    \end{multline}
    Taking $\rho_{\m{APs}}$ sufficiently small to absorb the $\ep$ terms, \eqref{e6} implies that  
    \begin{equation}\label{e7}
        {_\mu}\tnorm{v}_{H^s}+{_{\mu,\sig}}\tnorm{\eta}_{\mathcal{H}^s}\lesssim\tnorm{f}_{H^s}+{\tnorm{v_0,\be_0,\eta_0}_{H^{s+1}\times H^{s+1}\times\mathcal{H}^{s+1}}}\tp{\tnorm{v}_{H^{2+\tfloor{d/2}}}+{_\sig}\tnorm{\eta}_{\mathcal{H}^{2+\tfloor{d/2}}}}.
    \end{equation}
    This is nearly \eqref{final a priori estimate question mark}; it remains to remove the $v$ and $\eta$ terms on the right.  For this we consider \eqref{e7} in the case $s=s_0=2+\tfloor{d/2}$ and exploit~\eqref{conditions on the data for the linear analysis} to make another absorption argument in order to obtain the closed estimate
    \begin{equation}\label{e8}
        {_\mu}\tnorm{v}_{H^{2+\tfloor{d/2}}}+{_{\mu,\sig}}\tnorm{\eta}_{\mathcal{H}^{2+\tfloor{d/2}}}\lesssim\tnorm{f}_{H^{2+\tfloor{d/2}}}.
    \end{equation}
    Thus, by combining~\eqref{e7} and~\eqref{e8}, we find that for any $\N\ni s\ge2+\tfloor{d/2}$ we have the bound
    \begin{equation}\label{e9}
        {_\mu}\tnorm{v}_{H^s}+{_{\mu,\sig}}\tnorm{\eta}_{\mathcal{H}^s}\lesssim\tnorm{f}_{H^s}+{\tnorm{v_0,\eta_0,\be_0}_{\X_{s+1}\times H^{s+1}}}\tnorm{f}_{H^{2+\tfloor{d/2}}},
    \end{equation}
    which is \eqref{final a priori estimate question mark}.  This completes the first stage of the proof.

    In the second stage, the task is to add in the dependence of $h$ and to remove the extra Fourier support hypotheses on $\eta$.   Suppose first that the system~\eqref{linear equations to be studied as part of the linear analysis} is satisfied by $v$, $\eta$, and $f$ as in the previous stage, but now with $h\in(\dot{H}^{-1}\cap H^{1+s})(\R^d)$.  As in Proposition \ref{prop on decoupling reformulation}, we define $H\in H^{2+s}(\R^d)$ via $H=\grad\Delta^{-1}h$ and note that  $\norm{H}_{H^{2+s}} \asymp \norm{h}_{\dot{H}^{-1} \cap H^{1+s}}$.  We then set  $w=v-H$ and observe that $w$ and $\eta$ satisfy the system
    \begin{equation}\label{the special system is special for very special reasons}
        \begin{cases}
            \grad\cdot w+\grad\cdot(X\eta)=0,\\
        w+X\cdot\grad w-\mu^2\grad\cdot\S w+\tp{\vartheta-\sig^2\theta\Delta}\grad\eta=f-(I+X\cdot\grad-\mu^2\grad\cdot\S)H,
        \end{cases}
    \end{equation}
    which means we are in a position to apply estimate~\eqref{e9} from the previous stage. Doing so and handling the $X\cdot\grad H$ term with Proposition~\ref{corollary on tame estimates on simple multipliers}, we obtain the bound
    \begin{equation}\label{e10}
        {_\mu}\tnorm{w}_{H^s}+{_{\mu,\sig}}\tnorm{\eta}_{\mathcal{H}^s}\lesssim\tnorm{h,f}_{\dot{H}^{-1}\cap H^{1+s}\times H^s}+{\tnorm{v_0,\eta_0,\be_0}_{\X_{s+1}\times H^{s+1}}}\tnorm{h,f}_{\dot{H}^{-1}\cap H^{3+\tfloor{d/2}}\times H^{2+\tfloor{d/2}}}.
    \end{equation}
    Since ${_\mu}\tnorm{v}_{H^s}\lesssim {_\mu}\tnorm{w}_{H^s}+\tnorm{h}_{\dot{H}^{-1}\cap H^{1+s}}$, the estimate~\eqref{e10} implies~\eqref{final a priori estimate question mark}.

    The final task in the proof is to remove the auxiliary Fourier support assumption on $\eta$. This is straightforward, since we can frequency split $\eta=\Uppi^\kappa_{\m{H}}\eta+\Uppi^\kappa_{\m{L}}\eta$ with $\kappa\in(0,1)$ and consider the equations satisfied by $v$ and $\Uppi^\kappa_{\m{H}}\eta$:
    \begin{equation}\label{freq_split_system}
    \begin{cases}
        \grad\cdot v+\grad\cdot(X\Uppi^\kappa_{\m{H}}\eta)=h-\grad\cdot(X\Uppi^\kappa_{\m{L}}\eta),\\
        v+X\cdot\grad v-\mu^2\grad\cdot\S v+\tp{\vartheta-\sig^2\theta\Delta}\grad\Uppi^\kappa_{\m{H}}\eta=f-(\vartheta-\sig^2\theta\Delta)\grad\Uppi^\kappa_{\m{L}}\eta.
    \end{cases}
    \end{equation}
    To this system we invoke the a priori estimates of the previous stage and send $\kappa\to0$.  Since the right hand side of \eqref{freq_split_system} converges in $\tp{\dot{H}^{-1}\cap H^{1+s}}(\R^d)\times H^s(\R^d)$ to $(h,f)$ as $\kappa\to0$, we obtain the estimate \eqref{final a priori estimate question mark} as a result of this limiting procedure.
\end{proof}

For the final result of this subsection, we provide specialized a priori estimates for~\eqref{linear equations to be studied as part of the linear analysis} in the subsonic regime. The acronym $\m{SAP}$ stands for `subsonic a priori'.

\begin{thm}[Principal part a priori estimates, subsonic case]\label{a priori estimates subsonic principal part}
    Assume that $\mathcal{Q}=I\times J\times K$ is subsonic in the sense that $I\Subset(0,1)$ and suppose that $\N\ni s\ge 2+\tfloor{d/2}$. Further suppose that $f\in H^s(\R^d;\R^d)$, $h\in\tp{\dot{H}^{-1}\cap H^{1+s}}(\R^d)$, $v\in H^{2+s}(\R^d;\R^d)$, and $\eta\in\mathcal{H}^{3+s}(\R^d)$ solve system~\eqref{linear equations to be studied as part of the linear analysis}. There exists $\R^+\ni \rho_{\m{SAP}s}\le\rho_{\m{WD}} $, depending only on $\mathcal{Q}$, $\rho_{\m{WD}}$, $r$ (from Definition~\ref{parameters definition}) and $s$, with the property that for all $\R^+\ni\ep\le\rho_{\m{SAP}s}$ and $(v_0,\eta_0,\be_0)$ satisfying~\eqref{conditions on the data for the linear analysis}, we have the estimate
\begin{equation}\label{final a priori estimate question mark subsonic}
    \tnorm{v,\eta}_{{_{\mu,\sig}}\Xs_s}\lesssim\tnorm{h,f}_{\Y_s}+\tnorm{v_0,\eta_0,\be_0}_{\X_{1+s}\times H^{1+s}}\tnorm{h,f}_{\Y_{2+\tfloor{d/2}}}.
\end{equation}
Here the implicit constant in~\eqref{final a priori estimate question mark} depends only on $s$, $\mathcal{Q}$, and $d$ from Definition~\ref{parameters definition}.
\end{thm}
\begin{proof}
    We shall seek $\rho_{\m{SAP}s}\le\rho_{\m{AP}s}$, where the latter parameter is the one granted by Theorem~\ref{linear equations to be studied as part of the linear analysis}, so that estimate~\eqref{final a priori estimate question mark} holds for our solution.  To prove \eqref{final a priori estimate question mark subsonic} we only need to promote the regularity of $\eta$, as the required estimates on $v$ are already met; our main tool in achieving this promotion is a more careful study of equation~\eqref{the eta equation strikes back}. We can rewrite said equation as
    \begin{multline}\label{sharon}
        \tp{\Delta-\gam^2\pd_1^2}\eta-\sig^2\Delta^2\eta=-\eta_0\tp{1-\sig^2\Delta}\Delta\eta-\be_0\Delta\eta+\tp{X\otimes X-\gam^2 e_1\otimes e_1}:\grad^2\eta+X\cdot\grad\tp{\eta\grad\cdot X}\\
        +\grad\cdot f-\grad X^{\m{t}}:\grad v-(1-3\mu^2\Delta)\grad\cdot v-X\cdot\grad h 
- (\grad \vartheta \cdot \grad \eta - \sigma^2 \grad \theta \cdot \grad \Delta \eta)
    \end{multline}
    in order to make evident the utility of inclusion $\gam\in I\Subset(0,1)$ in the negative definiteness of the symbol of the operator on the left away from zero.  Note that this condition also implies the ellipticity of the operator on the left when $\sigma =0$. 
    
    We shall take the norm of both sides of equation~\eqref{sharon} in the space $H^{s-1}(\R^d)$.  Due to ellipticity, we obtain the bound
    \begin{multline}\label{ll1}
        \tnorm{\grad^2\eta}_{{_\sig}H^{s-1}}\lesssim\tnorm{-\eta_0\tp{1-\sig^2\Delta}\Delta\eta-\be_0\Delta\eta+\tp{X\otimes X-\gam^2 e_1\otimes e_1}:\grad^2\eta}_{H^{s-1}}\\
        + \tnorm{X\cdot\grad\tp{\eta\grad\cdot X}
        +\grad\cdot f-\grad X^{\m{t}}:\grad v-(1-3\mu^2\Delta)\grad\cdot v-X\cdot\grad h 
- (\grad \vartheta \cdot \grad \eta - \sigma^2 \grad \theta \cdot \grad \Delta \eta)}_{H^{s-1}}\\=\bf{I}+\bf{II},
    \end{multline}
    and we will estimate the two terms on the right separately. Thanks to Corollary~\ref{coro on more algebra properties} and Proposition~\ref{corollary on tame estimates on simple multipliers}, we have the bounds
    \begin{equation}\label{ll2}
        \bf{I}\lesssim\ep\tnorm{\grad^2\eta}_{{_\sig}H^{s-1}}+\tnorm{v_0,\eta_0,\be_0}_{\X_{s-1}\times H^{s-1}}\tnorm{\grad^2\eta}_{{_\sig}H^{1+\tfloor{d/2}}}
    \end{equation}
    and
    \begin{equation}\label{ll3}
        \bf{II}\lesssim\tnorm{v,\eta,h,f}_{{_{\mu,\sig}}\X_s\times\Y_s}+\tnorm{v_0,\eta_0,\be_0}_{\X_{s+1}\times H^{s+1}}\tnorm{v,\eta,h,f}_{{_{\mu,\sig}}\X_{2+\tfloor{d/2}}\times\Y_{2+\tfloor{d/2}}}.
    \end{equation}
    We combine~\eqref{ll1}, \eqref{ll2}, and~\eqref{ll3} with~\eqref{final a priori estimate question mark} to acquire the estimate
    \begin{multline}\label{i hear thunder}
        \tnorm{\grad^2\eta}_{{_\sig}H^{s-1}}\lesssim\ep\tnorm{\grad^2\eta}_{{_\sig}H^{s-1}}+\tnorm{v_0,\eta_0,\be_0}_{\X_{s-1}\times H^{s-1}}\tnorm{\grad^2\eta}_{{_\sig}H^{1+\tfloor{d/2}}}\\+\tnorm{h,f}_{\Y_s}+\tnorm{v_0,\eta_0,\be_0}_{\X_{1+s}\times H^{1+s}}\tnorm{h,f}_{\Y_{2+\tfloor{d/2}}}.
    \end{multline}
    By taking $s=2+\tfloor{d/2}$ and $\ep\le\rho_{\m{SAP}s}$ sufficiently small in~\eqref{i hear thunder}, we may use an absorbing argument to learn that
    \begin{equation}\label{i see lightning}
        \tnorm{\grad^2\eta}_{{_\sig}H^{1+\tfloor{d/2}}}\lesssim\tnorm{h,f}_{\Y_{2+\tfloor{d/2}}}.
    \end{equation}
    Inputting~\eqref{i see lightning} into~\eqref{i hear thunder} for $s\ge 2+\tfloor{d/2}$ and taking $\ep$ smaller, if necessary, we arrive at the bound
    \begin{equation}\label{i am groot}
        \tnorm{\grad^2\eta}_{{_\sig}H^{s-1}}\lesssim\tnorm{h,f}_{\Y_s}+\tnorm{v_0,\eta_0,\be_0}_{\X_{1+s}\times H^{1+s}}\tnorm{h,f}_{\Y_{2+\tfloor{d/2}}}.
    \end{equation}
    Combining~\eqref{i am groot} with~\eqref{final a priori estimate question mark} then yields~\eqref{final a priori estimate question mark subsonic}.
\end{proof}

% _+__+_ -_+__+_ -_+__+_ -_+__+_ -_+__+_ -_+__+_ -_+__+_ -_+__+_ -_+__+_ -_+__+_ -_+__+_ -_+__+_ -_+__+_ -
\subsection{Existence}\label{subsection on existence}
% _+__+_ -_+__+_ -_+__+_ -_+__+_ -_+__+_ -_+__+_ -_+__+_ -_+__+_ -_+__+_ -_+__+_ -_+__+_ -_+__+_ -_+__+_ -

The first goal of this subsection is to establish the existence of solutions to system~\eqref{linear equations to be studied as part of the linear analysis}. The case of $\mu\cdot\sig\neq0$ is a direct application of the method of continuity along with Theorems~\ref{thm on the easy case of linear well-posedness} and~\ref{fourth theorem on a priori estimates}.

\begin{prop}[Preliminary existence result]\label{first proposition on existence}
    Suppose that $\N\ni s\ge 2+\tfloor{d/2}$,  $\R^+\ni\ep\le \rho_{\m{APs}}$, and  $(v_0,\eta_0,\be_0)$ satisfy~\eqref{conditions on the data for the linear analysis}. If $(\mu,\sig)\in (J\setminus\tcb{0})\times(K\setminus\tcb{0})$, and $\gam\in I\Subset\R$, then for every $(h,f)\in(\dot{H}^{-1}\cap H^{1+s})(\R^d)\times H^s(\R^d;\R^d)$ there exists a unique $(v,\eta)\in H^{2+s}(\R^d;\R^d)\times\mathcal{H}^{3+s}(\R^d)$ satisfying system~\eqref{linear equations to be studied as part of the linear analysis}.
\end{prop}
\begin{proof}
    We consider the affine homotopy of bounded linear operators 
    \begin{equation}\label{d1.1}
        L_t:H^{2+s}(\R^d;\R^d)\times\mathcal{H}^{3+s}(\R^d)\to\tp{\dot{H}^{-1}\cap H^{1+s}}(\R^d)\times H^s(\R^d;\R^d) \text{ for } t\in[0,1]
    \end{equation}
    defined via
    \begin{equation}\label{d1.2}
        L_t(v,\eta)=\bpm\grad\cdot w-\gam\pd_1\eta+t\grad\cdot(v_0\eta)\\v-\gam\pd_1v-\mu^2\grad\cdot\S v+(1-\sig^2\Delta)\grad\eta+t\tp{v_0\grad\cdot v+(\eta_0+\be_0-\sig^2\eta_0\Delta)\grad\eta}\epm.
    \end{equation}
    Theorem~\ref{fourth theorem on a priori estimates} provides a constant $C>0$ such that 
    \begin{equation}
        {_\mu}\tnorm{v}_{H^s}+{_{\mu,\sig}}\tnorm{\eta}_{\mathcal{H}^s} \le C \tnorm{L_t(v,\eta)}_{\dot{H}^{-1}\cap H^{1+s}\times H^s} \text{ for all } t\in [0,1].
    \end{equation}
    On the other hand, Theorem~\ref{fourth theorem on a priori estimates} shows that $L_0$ is a linear isomorphism.  Hence,  the method of continuity (see, for instance, Theorem 5.2 in Gilbarg and Trudinger~\cite{MR1814364}) guarantees that the map $L_1$ is also an isomorphism, and $L_1$ corresponds to the operator in system~\eqref{linear equations to be studied as part of the linear analysis}.
\end{proof}

Our main existence theorems are the following two results.

\begin{thm}[Omnisonic and subsonic existence result]\label{main theorem of all of math in the math document where is it}
    Suppose that $\N\ni s\ge \max\tcb{3,2+\tfloor{d/2}}$. The following hold.
    \begin{enumerate}
        \item If $\R^+\ni\ep\le\rho_{\m{APs}}$, where the latter parameter is from Theorem~\ref{fourth theorem on a priori estimates}, $(v_0,\eta_0,\be_0)$ satisfy~\eqref{conditions on the data for the linear analysis}, $(\mu,\sig)\in J\times K$, and $\gam\in I\Subset\R^+$, then for every $(h,f)\in\Y_s$ there exists a unique $(v,\eta)\in\X_s$ satisfying system~\eqref{linear equations to be studied as part of the linear analysis}. Moreover, the solution obeys estimate~\eqref{final a priori estimate question mark}.
        \item If $\R^+\ni\ep\le\rho_{\m{SAPs}}$, where the latter parameter is from Theorem~\ref{a priori estimates subsonic principal part},   $(v_0,\eta_0,\be_0)$ satisfy~\eqref{conditions on the data for the linear analysis}, $(\mu,\sig)\in J\times K$, and $\gam\in I\Subset(0,1)$, then for every $(h,f)\in\Y_s$ there exists a unique $(v,\eta)\in\Xs_s$ satisfying system~\eqref{linear equations to be studied as part of the linear analysis}. Moreover, the solution obeys estimate~\eqref{final a priori estimate question mark subsonic}.
    \end{enumerate}
\end{thm}
\begin{proof} We only prove the first item, as the second item follows from a nearly identical argument. We begin by justifying the uniqueness assertion. By shifting to the system~\eqref{the special system is special for very special reasons} as in the proof of Theorem~\ref{fourth theorem on a priori estimates}, we see that it is sufficient to prove uniqueness in the case that $h=0$.  Uniqueness in this case is a consequence of estimate~\eqref{the base case a priori estimate is here, today and its gone} from Theorem~\ref{thm on base case a priori estimate}.

    We now prove existence. Let $\tcb{(\mu_n,\sig_n)}_{n\in\N}\subset(J\setminus\tcb{0})\times(K\setminus\tcb{0})$ be any sequence satisfying $(\mu_n,\sig_n)\to(\mu,\sig)$ as $n\to\infty$. Thanks to Proposition~\ref{first proposition on existence}, we obtain a corresponding sequence $\tcb{(v_n,\eta_n)}_{n\in\N}\subset H^{2+s}(\R^d;\R^d)\times\mathcal{H}^{3+s}(\R^d)$ such that
    \begin{equation}\label{standing solo in the sun}
        \begin{cases}
            \grad\cdot v_n+\grad\cdot(X\eta_n)=h,\\
        v_n+X\cdot\grad v_n-\mu_n^2\grad\cdot\S v_n+\tp{\vartheta-\sig_n^2\theta\Delta}\grad\eta_n=f.
        \end{cases}
    \end{equation}
    Theorem~\ref{fourth theorem on a priori estimates}, applied along this sequence, yields the uniform bounds
    \begin{equation}\label{final a priori estimate question mark, but now appearing with yep you guessed an index n}
    \tnorm{v_n,\eta_n}_{{_{\mu_n,\sig_n}}\X_s}\lesssim\tnorm{h,f}_{\Y_s}+\tnorm{v_0,\eta_0,\be_0}_{\X_{1+s}\times H^{1+s}}\tnorm{h,f}_{\Y_{2+\tfloor{d/2}}}.
    \end{equation}
    Thus, by the reflexivity of the container space, we may extract a subsequence that converges weakly to  $(v,\eta)\in H^s(\R^d;\R^d)\times\mathcal{H}^s(\R^d)$ in the weak $H^s(\R^d;\R^d)\times\mathcal{H}^s(\R^d)$ topology.  Taking weak limits along the subsequence in~\eqref{standing solo in the sun}, we deduce that  $(v,\eta)$ is the desired solution.  Sequential weak lower semicontinuity and~\eqref{final a priori estimate question mark, but now appearing with yep you guessed an index n} then provide the bound~\eqref{final a priori estimate question mark}. 
\end{proof}

% _+__+_ -_+__+_ -_+__+_ -_+__+_ -_+__+_ -_+__+_ -_+__+_ -_+__+_ -_+__+_ -_+__+_ -_+__+_ -_+__+_ -_+__+_ -
\subsection{Synthesis}\label{subsection on synthesis}
% _+__+_ -_+__+_ -_+__+_ -_+__+_ -_+__+_ -_+__+_ -_+__+_ -_+__+_ -_+__+_ -_+__+_ -_+__+_ -_+__+_ -_+__+_ -

The goal of this subsection is to combine the principal part linear analysis of Sections~\ref{subsection on estimates} and~\ref{subsection on existence} with the splitting analysis of Section~\ref{section on derivative splitting} to complete the satisfaction of the linear hypotheses of the Nash-Moser theorem.

The main emphasis of this subsection is to prove the results in full detail in the omnisonic regime. We then, at the end, consider the completely analogous results in the subsonic regime. The proofs for these are then greatly abbreviated, as they are very similar to the omnisonic case except for a few minor distinctions.

We begin with the following definition of maps and function spaces, which are crucially utilized in the subsequent proofs. Throughout the rest of this subsection we will employ the  notation set in Proposition~\ref{prop on derivative splitting}.

\begin{defn}[Operators and adapted domains, omnisonic case]\label{definition of operators and adapted domains I}
    Given $\gam$, $\mu$, $\sig$, $v_0$, $\eta_0$, and $\be_0$ as in Definition~\ref{parameters definition}, we define the following.
    \begin{enumerate}
        \item For $s \in \N$ the map $L^{\gam,\mu,\sig}_{v_0,\eta_0,\be_0}:H^{2+s}(\R^d;\R^d)\times\mathcal{H}^{3+s}(\R^d)\to\tp{\dot{H}^{-1}\cap H^{1+s}}(\R^d)\times H^s(\R^d;\R^d)$  is given by
        \begin{equation}
            L^{\gam,\mu,\sig}_{v_0,\eta_0,\be_0}(v,\eta)=P^{\gam,\mu,\sig}_{v_0,\eta_0}[v,\eta]+\tp{0,\be_0\grad\eta}.
        \end{equation}
        \item For $\N\ni s\ge 3$ we define the space ${_{v_0,\eta_0,\be_0}^{\gam,\mu,\sig}}\X_s=\tcb{(v,\eta)\in{_{\mu,\sig}}\X_s\;:L_{v_0,\eta_0,\be_0}^{\gam,\mu,\sig}(v,\eta)\in\Y_s}$ and equip it with the graph norm
        \begin{equation}\label{gn2}
            \tnorm{v,\eta}_{{_{v_0,\eta_0,\be_0}^{\gam,\mu,\sig}}\X_s}=\tnorm{v,\eta}_{{_{\mu,\sig}\X_s}}+\tnorm{L^{\gam,\mu,\sig}_{v_0,\eta_0,\be_0}(v,\eta)}_{\Y_s}.
        \end{equation}
    \end{enumerate}
\end{defn}

We now come to the first part of the synthesis. The subscript $\m{S}$ stands for synthesis and the $s$ refers to the data belonging to $\Y_s$.
\begin{prop}[Synthesis I, omnisonic case]\label{prop on the first part of the synthesis}
    Let $(\gam, \mu, \sig) \in \mathcal{Q}$ with $\mathcal{Q}$ as in Definition~\ref{parameters definition}.  For each $\N\ni s\ge\max\tcb{3,2+\tfloor{d/2}}$ there exists $\R^+\ni \tilde{\rho}_{\m{S}s}\le\rho_{\m{AP}s}$, depending only on $s$ and $\mathcal{Q}$, such that if $\R^+\ni\ep\in(0,\tilde{\rho}_{\m{S}s}]$, $(v_0,\eta_0,\be_0)$ satisfy \eqref{conditions on the data for the linear analysis}, $\upvarphi_0\in B_{H^r}(0,\ep)\cap H^\infty(\R^d;\R^d)$, and $\uptau_0\in B_{H^r}(0,\ep)\cap H^\infty(\R^d;\R^{d\times d})$, then the following hold.
    \begin{enumerate}
        \item The following linear map is well-defined and bounded:
        \begin{equation}\label{dies illa}
            {^{\gam,\mu,\sig}_{v_0,\eta_0,\be_0}}\X_s\ni(v,\eta)\mapsto L^{\gam,\mu,\sig}_{v_0,\eta_0,\be_0}(v,\eta)+(0,\upvarphi_0\eta+(\mu+\sig)\uptau_0\grad\eta)+(0,R^{\gam,\mu,\sig}_{v_0,\eta_0}[v,\eta])\in\Y_s.
        \end{equation}
        Moreover, this map is an isomorphism.
        \item  For all $(h,f)\in\Y_s$ the unique $(v,\eta)\in{{_{v_0,\eta_0,\be_0}^{\gam,\mu,\sig}}\X_s}$ satisfying the system
    \begin{equation}\label{equations of the first synthesis}
    L^{\gam,\mu,\sig}_{v_0,\eta_0,\be_0}(v,\eta)+(0,\upvarphi_0\eta+(\mu+\sig)\uptau_0\grad\eta)+(0,R^{\gam,\mu,\sig}_{v_0,\eta_0}[v,\eta])=(h,f)
    \end{equation}
    obeys the tame estimate
    \begin{equation}\label{estimate of the first synthesis}
        \tnorm{v,\eta}_{{_{\mu,\sig}}\X_s}\lesssim\tnorm{h,f}_{\Y_s}+\tnorm{v_0,\eta_0,\be_0,\uptau_0,\upvarphi_0}_{\X_{3+s}\times H^{3+s}\times H^{3+s}\times H^{3+s}}\tnorm{h,f}_{\Y_{\max\tcb{3,2+\tfloor{d/2}}}},
    \end{equation}
    where the implicit constant depends only on $s$, $d$, and $\mathcal{Q}$. 
    \end{enumerate}
\end{prop}
\begin{proof}

We begin by verifying the map~\eqref{dies illa} is well-defined and bounded. Due to the definition of the norm on the space ${^{\gam,\mu,\sig}_{v_0,\eta_0,\be_0}}\X_s$, which is provided in Definition~\ref{definition of operators and adapted domains I}, and the embedding ${^{\gam,\mu,\sig}_{v_0,\eta_0,\be_0}}\X_s\emb{_{\mu,\sig}}\X_s$, it is sufficient to verify that the map
\begin{equation}
    {_{\mu,\sig}}\X_s\ni(v,\eta)\mapsto\upvarphi_0\eta+(\mu+\sig)\uptau_0\grad\eta+R^{\gam,\mu,\sig}_{v_0,\eta_0}[v,\eta]\in H^s(\R^d;\R^d)
\end{equation}
is well-defined and bounded.  This is a consequence of the first item of Proposition~\ref{prop on remainder estimates}, Corollary~\ref{coro on more algebra properties}, and the definition of the norm on ${_{\mu,\sig}}\X_s$ from~\eqref{sybil}.

It remains to show that the map~\eqref{dies illa} is an isomorphism that obeys estimate~\eqref{estimate of the first synthesis}.   We will prove this in two steps. For the first step we shall use the following pair of definitions.
\begin{itemize}
    \item For $s \in \N$ the map $M_{v_0,\eta_0,\be_0}^{\gam,\mu,\sig}:H^{2+s}(\R^d;\R^d)\times\mathcal{H}^{3+s}(\R^d)\to\tp{\dot{H}^{-1}\cap H^{1+s}}(\R^d)\times H^s(\R^d;\R^d)$  is given via
        \begin{equation}
            M_{v_0,\eta_0,\be_0}^{\gam,\mu,\sig}(v,\eta)=Q^{\gam,\mu,\sig}_{v_0,\eta_0}[v,\eta]+\tp{0,\be_0\grad\eta}.
        \end{equation}
        \item For $\N\ni s\ge 3$ we define the space ${^{\gam,\mu,\sig}_{v_0,\eta_0,\be_0}}\tilde{\X}_s=\tcb{(v,\eta)\in{_{\mu,\sig}}\X_s\;:\;M^{\gam,\mu,\sig}_{v_0,\eta_0,\be_0}(v,\eta)\in\Y_s}$ and equip it with the graph norm
        \begin{equation}\label{gn1}
            \tnorm{v,\eta}_{{^{\gam,\mu,\sig}_{v_0,\eta_0,\be_0}}\tilde{\X}_s}=\tnorm{v,\eta}_{{_{\mu,\sig}\X_s}}+\tnorm{M^{\gam,\mu,\sig}_{v_0,\eta_0,\be_0}(v,\eta)}_{\Y_s}.
        \end{equation}
\end{itemize}
The claim of the first step is then as follows. Let $b=\max\tcb{3,2+\tfloor{d/2}}$. For each $\N\ni s\ge b$ there exists $\R^+\ni\rho_{\m{XS}s}\le\rho_{\m{AP}s}$ with the property that for all $\R^+\ni\ep\in(0,\rho_{\m{XS}s}]$ and $(h,f)\in\Y_s$ there exists a unique $(v,\eta)\in{{^{\gam,\mu,\sig}_{v_0,\eta_0,\be_0}}\tilde{\X}_s}$ satisfying
\begin{equation}\label{synthesis I, step I equations}
    M^{\gam,\mu,\sig}_{v_0,\eta_0,\be_0}(v,\eta)+(0,\upvarphi_0\eta+(\mu+\sig)\uptau_0\grad\eta)+(0,S^{\gam,\mu,\sig}_{v_0,\eta_0}[v])=(h,f)
\end{equation}
as well as the estimate
\begin{equation}\label{synthesis I, step I estimate}
    \tnorm{v,\eta}_{{_{\mu,\sig}}\X_s}\lesssim\tnorm{h,f}_{\Y_s}+\tnorm{v_0,\eta_0,\be_0,\uptau_0,\upvarphi_0}_{\X_{2+s}\times H^{2+s}\times H^{2+s}\times H^{2+s}}\tnorm{h,f}_{\Y_{b}}
\end{equation}
for an implicit constant depending only on $s$, $d$, and $\mathcal{Q}$. 

We begin the proof of the claim by considering the homotopy of operators $M_t:\;{^{\gam,\mu,\sig}_{v_0,\eta_0,\be_0}}\tilde{\X}_s\to\Y_s$ defined for $t\in[0,1]$ via
\begin{equation}\label{synthesis I, step I equations, I}
    M_t(v,\eta)=M^{\gam,\mu,\sig}_{v_0,\eta_0,\be_0}(v,\eta)+t(0,\upvarphi_0\eta+(\mu+\sig)\uptau_0\grad\eta)+t(0,S^{\gam,\mu,\sig}_{v_0,\eta_0}[v]).
\end{equation}
Thanks to the embedding ${^{\gam,\mu,\sig}_{v_0,\eta_0,\be_0}}\tilde{\X}_s\emb\;{_{\mu,\sig}}\X_s$ and the remainder estimate of the second item of Proposition~\ref{prop on remainder estimates}, we readily deduce that each $M_t$ is a well-defined and bounded linear map.  Our goal is to use the method of continuity on the family $\{M_t\}_{t \in [0,1]}$, but to do so we must first establish the requisite estimates.

Provided that $\rho_{\m{XS}s}\le\rho_{\m{AP}s}$, the estimates of Theorem~\ref{main theorem of all of math in the math document where is it} grant us the bound
\begin{multline}\label{intermediate estimate adsfjkgna}
    \tnorm{v,\eta}_{{_{\mu,\sig}}\X_s}\lesssim\tnorm{M_t(v,\eta)}_{\Y_s}+\tnorm{\upvarphi_0\eta+(\mu+\sig)\uptau_0\grad\eta+S^{\gam,\mu,\sig}_{v_0,\eta_0}[v]}_{H^s}\\+\tnorm{v_0,\eta_0,\be_0}_{\X_{1+s}\times H^{1+s}}\tp{\tnorm{M_t(v,\eta)}_{\Y_{2+\tfloor{d/2}}}+\tnorm{\upvarphi_0\eta+(\mu+\sig)\uptau_0\grad\eta+S^{\gam,\mu,\sig}_{v_0,\eta_0}[v]}_{H^{2+\tfloor{d/2}}}}.
\end{multline}
We next take $s=b$ in~\eqref{intermediate estimate adsfjkgna} and estimate
\begin{equation}
    \tnorm{\upvarphi_0\eta+(\mu+\sig)\uptau_0\grad\eta}_{H^b}\lesssim\ep\tnorm{\eta}_{{_{\mu,\sig}}\mathcal{H}^b}
    \text{ and }
    \tnorm{S^{\gam,\mu,\sig}_{v_0,\eta_0}[v]}_{H^b}
    \lesssim \ep\cdot{_\mu}\tnorm{v}_{H^b},
\end{equation}
by employing Corollary~\ref{coro on more algebra properties} for the former and the second item of Proposition~\ref{prop on remainder estimates} for the latter.  Combining these with \eqref{intermediate estimate adsfjkgna} shows that 
\begin{equation}
\tnorm{v,\eta}_{{_{\mu,\sig}}\X_b}\lesssim\tnorm{M_t(v,\eta)}_{\Y_b}+\ep\tnorm{v,\eta}_{{_{\mu,\sig}}\X_b},    
\end{equation}
and thus  if $\ep\le\rho_{\m{XS}b}$ is taken sufficiently small we may use an absorbing argument to obtain the estimate
\begin{equation}\label{base estimate 1}
    \tnorm{v,\eta}_{{_{\mu,\sig}}\X_b}\lesssim\tnorm{M_t(v,\eta)}_{\Y_b}.
\end{equation}
Armed with~\eqref{base estimate 1}, we return to~\eqref{intermediate estimate adsfjkgna} for $s\ge b$, but this time we estimate
\begin{multline}\label{vivaldi has a mohawk}
    \tnorm{\upvarphi_0\eta+(\mu+\sig)\uptau_0\grad\eta}_{H^s}\lesssim\ep\tnorm{\eta}_{{_{\mu,\sig}}\mathcal{H}^s}+\tnorm{\upvarphi_0,\uptau_0}_{H^s\times H^s}\tnorm{\eta}_{{_{\mu,\sig}}\mathcal{H}^b}
    \\\text{and }
    \tnorm{S^{\gam,\mu,\sig}_{v_0,\eta_0}[v]}_{H^s}\lesssim\ep\cdot{_{\mu}}\tnorm{v}_{H^s}+\tnorm{v_0,\eta_0}_{\X_{2+s}}\tnorm{v}_{H^b}
\end{multline}
in order to deduce from~\eqref{intermediate estimate adsfjkgna}, \eqref{base estimate 1}, and~\eqref{vivaldi has a mohawk} that
\begin{equation}\label{rho sub arse}
    \tnorm{v,\eta}_{{_{\mu,\sig}}\X_s}\lesssim\tnorm{M_t(v,\eta)}_{\Y_s}+\tnorm{v_0,\eta_0,\be_0,\uptau_0,\upvarphi_0}_{\X_{2+s}\times H^{2+s}\times H^{2+s}\times H^{2+s}}\tnorm{M_t(v,\eta)}_{\Y_b}+\ep\tnorm{v,\eta}_{{_{\mu,\sig}}\X_s}.
\end{equation}
Further decreasing $\ep\le\rho_{\m{XS}s}\le\rho_{\m{XS}b}$ if necessary, we can again use an absorbing argument to see that
\begin{equation}\label{trinity}
    \tnorm{v,\eta}_{{_{\mu,\sig}}\X_s}\lesssim\tnorm{M_t(v,\eta)}_{\Y_s}+\tnorm{v_0,\eta_0,\be_0,\uptau_0\upvarphi_0}_{\X_{2+s}\times H^{2+s}\times H^{2+s}\times H^{2+s}}\tnorm{M_t(v,\eta)}_{\Y_b},
\end{equation}
for an implicit constant depending only on $s$ and $\mathcal{Q}$.

Estimate~\eqref{trinity} is not quite what we need for the method of continuity, so we continue to work.  To control the full norm on ${_{v_0,\eta_0,\be_0}^{\gam,\mu,\sig}}\tilde{\X}_s$ it remains to estimate $\tnorm{M_0(v,\eta)}_{\Y_s}$.  To this end, we first use~\eqref{synthesis I, step I equations, I} along with estimate~\eqref{vivaldi has a mohawk}:
\begin{multline}\label{dies irae}
    \tnorm{M_0(v,\eta)}_{\Y_s}\le\tnorm{M_t(v,\eta)}_{\Y_s}+\tnorm{(M_0-M_t)(v,\eta)}_{\Y_s}\lesssim\tnorm{M_t(v,\eta)}_{\Y_s}\\+\ep\tnorm{v,\eta}_{{_{\mu,\sig}}\X_s}+\tnorm{v_0,\eta_0,\upvarphi_0,\uptau_0}_{\X_{2+s}\times H^s\times H^s}\tnorm{v,\eta}_{\X_b}.
\end{multline}
Next, we use~\eqref{base estimate 1} and~\eqref{trinity} to handle the $(v,\eta)$ terms on the right hand side of~\eqref{dies irae}, which results in the bound
\begin{equation}\label{Lacrimosa}
    \tnorm{M_0(v,\eta)}_{\Y_s}\lesssim\tnorm{M_t(v,\eta)}_{\Y_s}+\tnorm{v_0,\eta_0,\be_0,\uptau_0\upvarphi_0}_{\X_{2+s}\times H^{2+s}\times H^{2+s}\times H^{2+s}}\tnorm{M_t(v,\eta)}_{\Y_b}.
\end{equation}
Therefore, by combining~\eqref{trinity} and~\eqref{Lacrimosa} we obtain a  constant $C\in\R^+$, depending on $s$, the tuple $(v_0,\eta_0,\be_0,\uptau_0,\upvarphi_0)$, $d$, and $\mathcal{Q}$ but independent of $(v,\eta)$, such that
\begin{equation}\label{uniform closed range estimates}
    \tnorm{v,\eta}_{{_{v_0,\eta_0,\be_0}^{\gam,\mu,\sig}}\tilde{\X}_s}\le C\tnorm{M_t(v,\eta)}_{\Y_s} \text{ for all } t \in [0,1].  
\end{equation}
This, finally, is the estimate needed for the method of continuity.  Theorem~\ref{main theorem of all of math in the math document where is it} shows that the operator $M_0$ is an isomorphism, and so the bounds \eqref{uniform closed range estimates} and the the method of continuity establish that the operator $
M_1$ is also an isomorphism.  Thus, the existence and uniqueness of solutions to~\eqref{synthesis I, step I equations} is established.  The claimed estimate~\eqref{synthesis I, step I estimate} then follows from~\eqref{trinity}, and  this concludes the first step of the proof.

In the second step we aim to build on the result of the first step to complete the proof, and so we will actually seek $\tilde{\rho}_{\m{S}s}\le\rho_{\m{XS}s}$ in order to justify use of the first step.   Once more we will proceed via the method of continuity by considering the operator homotopy $L_t:{_{v_0,\eta_0,\be_0}^{\gam,\mu,\sig}}\X_s\to\Y_s$ defined for $t\in[0,1]$ via
\begin{equation}
L_t(v,\eta)=M_1((1+\eta_0)v,\eta)+t(0,R^{\gam,\mu,\sig}_{v_0,\eta_0}[v,\eta]).
\end{equation}
Thanks to the embedding ${^{\gam,\mu,\sig}_{v_0,\eta_0,\be_0}}\X_s\emb\;{_{\mu,\sig}}\X_s$, the remainder estimate in the first item of Proposition~\ref{prop on remainder estimates}, and the identity
\begin{equation}
    M_1((1+\eta_0)v,\eta)=P^{\gam,\mu,\sig}_{v_0,\eta_0}[v,\eta]+(0,\be_0\grad\eta+\upvarphi_0\eta)=L^{\gam,\mu,\sig}_{v_0,\eta_0,\be_0}(v,\eta)+(0,\upvarphi_0\eta+(\mu+\sig)\uptau_0\grad\eta),
\end{equation}
which is a consequence of the second item of Proposition~\ref{prop on derivative splitting}, we have that the $L_t$ are well-defined and bounded linear maps.  As above, we now turn to the derivation of the estimates needed for the method of continuity.

The condition $\ep\le\tilde{\rho}_{\m{S}s}\le\rho_{\m{XS}s}$ allows us to invoke estimate~\eqref{trinity} to acquire the following bound for $s\ge b=\max\tcb{3,2+\tfloor{d/2}}$:
\begin{multline}\label{sec aux aukland n'zehland}
    \tnorm{(1+\eta_0)v,\eta}_{{_{\mu,\sig}}\X_s}\lesssim\tnorm{L_t(v,\eta)}_{\Y_s}+\tnorm{R^{\gam,\mu,\sig}_{v_0,\eta_0}[v,\eta]}_{H^s}\\+\tnorm{v_0,\eta_0,\be_0,\uptau_0,\upvarphi_0}_{\X_{2+s}\times H^{2+s}\times H^{2+s}\times H^{2+s}}\tp{\tnorm{L_t(v,\eta)}_{\Y_b}+\tnorm{R^{\gam,\mu,\sig}_{v_0,\eta_0}[v,\eta]}_{H^b}}.
\end{multline}
By taking $s=b$ in~\eqref{sec aux aukland n'zehland}, using $\tnorm{\eta_0v}_{{_\mu}H^b}\lesssim\ep\tnorm{v}_{{_\mu}H^b}$, and invoking the remainder estimate~\eqref{remainder estimate}, we find that $\tnorm{v,\eta}_{{_{\mu,\sig}}\X_b}\lesssim\tnorm{L_t(v,\eta)}_{\Y_b}+\ep\tnorm{v,\eta}_{{_{\mu,\sig}}\X_b}$; and hence, if $\ep\le\rho_{\m{S}b}$ is sufficiently small we have the bound
\begin{equation}\label{another base case estimate}
    \tnorm{v,\eta}_{{_{\mu,\sig}}\X_b}\lesssim\tnorm{L_t(v,\eta)}_{\Y_b}.
\end{equation}
We now return to~\eqref{sec aux aukland n'zehland} with the base estimate~\eqref{another base case estimate} in hand;  again invoking the remainder estimate~\eqref{remainder estimate} and using the bound $\tnorm{\eta_0v}_{{_\mu}H^s}\lesssim\ep\tnorm{v}_{H^s}+\tnorm{\eta_0}_{\mathcal{H}^{2+s}}\tnorm{v}_{H^b}$, together with~\eqref{another base case estimate}, we find that if $\ep\le\tilde{\rho}_{\m{S}s}\le\rho_{\m{S}b}$ is sufficiently small then
\begin{equation}\label{gadget}
    \tnorm{v,\eta}_{{_{\mu,\sig}}\X_s}\lesssim\tnorm{L_t(v,\eta)}_{\Y_s}+\tnorm{v_0,\eta_0,\be_0,\uptau_0,\upvarphi_0}_{\X_{3+s}\times H^{3+s}\times H^{3+s}\times H^{3+s}}\tnorm{L_t(v,\eta)}_{\Y_b}.
\end{equation}

As in the proof of the first step, estimate~\eqref{gadget} will lead us to the estimates we need for the operator homotopy $\{L_t\}_{t \in [0,1]}$ with a bit more work.  The remaining piece of the norm on ${_{v_0,\eta_0,\be_0}^{\gam,\mu,\sig}}\X_s$ is $\tnorm{L_0(v,\eta)-(0,\upvarphi_0\eta+(\mu+\sig)\uptau_0\grad\eta)}_{\Y_s}$. By the same strategy used in~\eqref{dies irae}, we derive the estimate
\begin{multline}\label{stop and smell the roses}
    \tnorm{L_0(v,\eta)-(0,\upvarphi_0\eta+(\mu+\sig)\uptau_0\grad\eta)}_{\Y_s}\lesssim\tnorm{L_t(v,\eta)}_{\Y_s}+\ep\tnorm{v,\eta}_{_{\mu,\sig}\X_s}\\+\tnorm{v_0,\eta_0,\be_0,\uptau_0,\upvarphi_0}_{\X_{3+s}\times H^{3+s}\times H^{3+s}\times H^{3+s}}\tnorm{v,\eta}_{\X_b}.
\end{multline}
Combining~\eqref{gadget} and~\eqref{stop and smell the roses} then provides  a constant $C'\in\R^+$, depending only on $s$, $d$, $\mathcal{Q}$, and $(v_0,\eta_0,\be_0,\upvarphi_0)$, such that 
\begin{equation}\label{charlie chaplin}
    \tnorm{v,\eta}_{{_{v_0,\eta_0,\be_0}^{\gam,\mu,\sig}}\X_s}\le C'\tnorm{L_t(v,\eta)}_{\Y_s} \text{ for all } t\in [0,1].
\end{equation}
This and the method of continuity show that in order to show that $L_1$ is an isomorphism it suffices to prove that the map $L_0$ is an isomorphism.

Estimate~\eqref{charlie chaplin} proves the injectivity of $L_0$.  For surjectivity we let $(h,f)\in\Y_s$. The first step of the proof show that the map $M_1$ is an isomorphism, so there exists $(w,\eta)\in{_{v_0,\eta_0,\be_0}^{\gam,\mu,\sig}}\tilde{\X}_s$ such that $M_1(w,\eta)=(h,f)$. We set $v=w/(1+\eta_0)$ and note that Lemma~\ref{the preliminary lemma for you and I} and Corollary~\ref{coro on more algebra properties} guarantee the inclusion $v\in{_\mu}H^s(\R^d;\R^d)$. Then
\begin{equation}
    L^{\gam,\mu,\sig}_{v_0,\eta_0,\be_0}(v,\eta)=(h,f)-(0,\upvarphi_0\eta+(\mu+\sig)\uptau_0\grad\eta)\in\Y_s, 
\end{equation}
and hence $(v,\eta)\in{_{v_0,\eta_0,\be_0}^{\gam,\mu,\sig}}\X_s$ and $L_0(v,\eta)=(h,f)$. This shows that $L_0$ is surjective, and so $L_1$ is an isomorphism.  This establishes that the map \eqref{dies illa} is an isomorphism; the estimate \eqref{estimate of the first synthesis} follows from ~\eqref{gadget} with $t=1$.
\end{proof}

\begin{coro}[Synthesis II, omnisonic case]\label{coro thm on synthesis 1.5 omnisonic case}
    Let $(\gam, \mu, \sig) \in \mathcal{Q}$ with $\mathcal{Q}$ as in Definition~\ref{parameters definition}.  For each $\N\ni s\ge\max\tcb{3,2+\tfloor{d/2}}$ there exists $\R^+\ni\rho_{\m{S}s}\le\tilde{\rho}_{\m{S}s}$, depending only on $s$ and $\mathcal{Q}$, such that if $\ep\in(0,\rho_{\m{S}s}]$ and
    \begin{equation}\label{a new condition for tuba mirum}
        (v_0,\eta_0,\tcb{\tau_{0i}}_{i=0}^\ell,\tcb{\varphi_{0i}}_{i=0}^\ell)\in B_{\X_r\times\W_r}(0,\ep)\cap\tp{\X_\infty\times\W_\infty} 
        \text{ for } r=\max\tcb{6,5+\tfloor{d/2}},
    \end{equation}
    then the following hold.
    \begin{enumerate}
        \item Upon setting
        \begin{equation}\label{dictionary}
            \uptau_0=\sum_{i=0}^\ell\eta_0^i\tau_{0i},
            \quad\be_0=\m{Tr}\uptau_0,
            \text{ and }
            \upvarphi_0=\sum_{i=1}^\ell i\eta_0^{i-1}\tp{\tau_{0i}\grad\eta_0+\varphi_{0i}},
        \end{equation}
        the linear map~\eqref{dies illa} is well-defined, bounded, and an isomorphism.
        \item  For all $(h,f)\in\Y_s$ the unique $(v,\eta)\in{{_{v_0,\eta_0,\be_0}^{\gam,\mu,\sig}}\X_s}$ satisfying the system~\eqref{equations of the first synthesis} obeys the tame estimate
    \begin{equation}\label{estimate of the first synthesis x}
        \tnorm{v,\eta}_{{_{\mu,\sig}}\X_s}\lesssim\tnorm{h,f}_{\Y_s}+\tnorm{v_0,\eta_0,\tcb{\tau_{0i}}_{i=0}^\ell,\tcb{\varphi_{0i}}_{i=0}^\ell}_{\X_{3+s}\times\W_{3+s}}\tnorm{h,f}_{\Y_{\max\tcb{3,2+\tfloor{d/2}}}},
    \end{equation}
    where the implicit constant depends only on $s$, $d$, and $\mathcal{Q}$. 
    \end{enumerate}
\end{coro}
\begin{proof}
    Thanks to Corollary~\ref{coro on more algebra properties} and Proposition~\ref{corollary on tame estimates on simple multipliers}, we make take $0<\ep\le\rho_{\m{S}s}$ sufficiently small to guarantee that $\uptau_0\in B_{H^r}(0,\tilde{\rho}_{\m{S}s})\cap H^\infty(\R^d;\R^{d\times d})$, $\be_0\in B_{H^r}(0,\tilde{\rho}_{\m{S}s})\cap H^\infty(\R^d)$, and $\upvarphi_0\in B_{H^r}(0,\tilde{\rho}_{\m{S}s})\cap H^\infty(\R^d;\R^d)$, where the parameter $\tilde{\rho}_{\m{S}s}\in\R^+$ is granted by Proposition~\ref{prop on the first part of the synthesis}. Therefore, by invoking this latter result, we obtain the first item and estimate~\eqref{estimate of the first synthesis}. To swap to the claimed estimate~\eqref{estimate of the first synthesis x}, one simply need note that another application of Corollary~\ref{coro on more algebra properties} and Proposition~\ref{corollary on tame estimates on simple multipliers} permits one the bound
    \begin{equation}
        \tnorm{v_0,\eta_0,\be_0,\uptau_0,\upvarphi_0}_{\X_{3+s}\times H^{3+s}\times H^{3+s}\times H^{3+s}}\lesssim\tnorm{v_0,\eta_0,\tcb{\tau_{0i}}_{i=0}^\ell,\tcb{\varphi_{0i}}_{i=0}^\ell}_{\X_{3+s}\times\W_{3+s}}.
    \end{equation}
\end{proof}

 \begin{rmk}\label{rmk on nondecreasing}
     We are free to assume that the sequence $\tcb{\rho_{\m{S}s}}_{s=3+\tfloor{d/2}}^\infty\subset(0,\rho_{\m{WD}}]$ generated by by Corollary~\ref{coro thm on synthesis 1.5 omnisonic case} is nonincreasing.
 \end{rmk}

 We are now ready for the first of our main theorems of our linear analysis.

\begin{thm}[Synthesis III, omnisonic case]\label{thm on synthesis, II}
    Set $b=\max\tcb{3,2+\tfloor{d/2}}$, and let $\Bar{\Uppsi} : \E_{b} \to \F_{b-3}$ be the map given in~\eqref{definition of uppsi bar}, which is well-defined and smooth in light of the first item of Theorem \ref{thm on smooth-tameness}. Let $\N\ni s\ge b$ and $\rho_{\m{S}s}\in \R^+$ be given by Corollary~\ref{coro thm on synthesis 1.5 omnisonic case}.  Suppose that  $0<\ep\le \min\{\rho_{\m{S}s},\rho_{\m{WD}} \}$ (recall that the latter is defined in Lemma~\ref{the preliminary lemma for you and I}),  $I\Subset\R^+$ and $J,K\Subset\R$ are bounded open intervals, and
    \begin{equation}\label{thm on synthesis, II p0}
        \bf{p}_0=(\gam_0,\mu_0,\sig_0,\tcb{\tau_{0i}}_{i=0}^\ell,\tcb{\varphi_{0i}}_{i=0}^\ell,w_0,\xi_0,v_0,\eta_0)\in [I\times J\times K\times B_{\W_r}(0,\ep)\times B_{\X_r\times\X_r}(0,\ep)]\cap\E_\infty,
    \end{equation}
    where $r \in \N$ is as in Definition \ref{parameters definition}. We set $\be_0=\m{Tr}\sum_{i=0}^\ell\eta_0^i\tau_{0i}\in H^\infty(\R^d)$. The following hold.
    \begin{enumerate}
        \item  The restriction of $D\Bar{\Uppsi}(\bf{p}_0) $ to $\R^3\times \W_s\times \X_s\times[{_{v_0,\eta_0,\be_0}^{\gam_0,\mu_0,\sig_0}}\X_s]$ is well-defined and takes values in $\F_s$, and the induced linear map
        \begin{equation}\label{adapted isomorphism of the nonlinear operators derivative}
            D\Bar{\Uppsi}(\bf{p}_0):\R^3\times \W_s\times \X_s\times[{_{v_0,\eta_0,\be_0}^{\gam_0,\mu_0,\sig_0}}\X_s]\to\F_s
        \end{equation}
        is  an isomorphism. 
        \item Assume that $\bf{f}\in\F_s$ and $\bf{e}\in \R^3\times \W_s\times \X_s\times[{_{v_0,\eta_0,\be_0}^{\gam_0,\mu_0,\sig_0}}\X_s]$ is defined via $\bf{e}=(D\Bar{\Uppsi}(\bf{p}_0))^{-1}\bf{f}$. We then have the tame estimate
        \begin{equation}\label{fly me to the moon}
            \tnorm{\bf{e}}_{\E_s}\lesssim\tnorm{\bf{f}}_{\F_s}+\tbr{\tnorm{\bf{p}_0}_{\E_{s+3}}}\tnorm{\bf{f}}_{\F_{b}}
        \end{equation}
        for an implicit constant depending only on $s$, $d$, $I$, $J$, and $K$.
    \end{enumerate}
\end{thm}
\begin{proof}

   We break the proof into two steps. In the first step we use~\eqref{most banking institutions offer} to view $\Upupsilon:\R^3\times [\X_{b}\times\W_{b}] \to \Y_{b-3}\times \W_{b-3}$, defined in~\eqref{upsilon is found here}, as a smooth map with a domain comprised of two factors (namely, $\R^3$ and  $\X_{b}\times\W_{b}$) and study the derivative with respect to each factor evaluated at the point $\bf{q}_0=(\gam_0,\mu_0,\sig_0,v_0,\eta_0,\tcb{\tau_{0i}}_{i=0}^\ell,\tcb{\varphi_{0i}}_{i=0}^\ell)$.  We begin with $D_1\Upupsilon(\bf{q}_0) : \R^3 \to \Y_{b-3} \times \W_{b-3}$.  The inclusion~\eqref{most banking institutions offer} and the definition of strong $3-$tame smoothness (see Definition~\ref{defn of smooth tame maps}) requires that
    \begin{equation}
        \tnorm{D_1\Upupsilon(\bf{q}_0)(\gam,\mu,\sig)}_{\Y_{t-3}\times\W_{t-3}}
        =
        \tnorm{D\Upupsilon(\bf{q}_0)(\gam,\mu,\sig,0,0,0,0,0)}_{\Y_{t-3}\times\W_{t-3}}
        \lesssim
        \tbr{\tnorm{\bf{q}_0}_{\R^3\times\W_{t}\times\X_{t}}} \norm{\gam,\mu,\sig}_{\R^3}
    \end{equation}
    for every $t \ge b$, and so we actually have that 
   \begin{equation}\label{comer el pulpo}
       D_1\Upupsilon(\bf{q}_0)(\gam,\mu,\sig) \in \Y_\infty \times \W_{\infty}
   \end{equation}
   for all $(\gamma,\mu,\sig) \in \R^3$.

    Next, we turn our attention to $D_2\Upupsilon(\bf{q}_0) :\X_{b}\times\W_{b} \to \Y_{b-3}\times \W_{b-3}$.  By assumption, we have that $s \ge b$, and so  ${_{v_0,\eta_0,\be_0}^{\gam_0,\mu_0,\sig_0}}\X_s \hookrightarrow \X_{b}$, which means that the space ${_{v_0,\eta_0,\be_0}^{\gam_0,\mu_0,\sig_0}}\X_s \times \W_s$ lies within the domain of  $D_2\Upupsilon(\bf{q}_0)$.  We claim that for $s \ge b$ and $\ep\le\rho_{\m{S}s}$ we have the inclusion $D_2\Upupsilon(\bf{q}_0)({_{v_0,\eta_0,\be_0}^{\gam_0,\mu_0,\sig_0}}\X_s\times\W_s) \subseteq \Y_s\times\W_s$ and that the induced map     \begin{equation}\label{comer un sandwich}
        D_2\Upupsilon(\bf{q}_0): {_{v_0,\eta_0,\be_0}^{\gam_0,\mu_0,\sig_0}}\X_s\times\W_s
        \to \Y_s\times\W_s.
    \end{equation}
        is a bounded isomorphism obeying the tame bound
    \begin{multline}\label{bound for step 1}
        \tnorm{v,\eta,\tcb{\tau_i}_{i=0}^\ell,\tcb{\varphi_i}_{i=0}^\ell}_{{_{\mu_0,\sig_0}}\X_s\times\W_s}\lesssim\tnorm{D_2\Upupsilon(\bf{q}_0)[v,\eta,\tcb{\tau_i}_{i=0}^\ell,\tcb{\varphi_i}_{i=0}^\ell]}_{\Y_s\times\W_s}\\+\tnorm{v_0,\eta_0,\tcb{\tau_{0i}}_{i=0}^\ell,\tcb{\varphi_{0i}}_{i=0}^\ell}_{\X_{3+s}\times\W_{3+s}}\tnorm{D_2\Upupsilon(\bf{q}_0)[v,\eta,\tcb{\tau_i}_{i=0}^\ell,\tcb{\varphi_i}_{i=0}^\ell]}_{\Y_b\times\W_b}.
    \end{multline}

    We begin the proof of the claim by establishing the stated inclusion.  Suppose that 
    \begin{equation}\label{comer las mandarinas}
    D_2\Upupsilon(\bf{q}_0)[v,\eta,\tcb{\tau_i}_{i=0}^\ell,\tcb{\varphi_i}_{i=0}^\ell]=(h,f,\tcb{\tilde{\tau}_i}_{i=0}^\ell,\tcb{\tilde{\varphi}_i}_{i=0}^\ell)    
    \end{equation}
    for $(v,\eta,\tcb{\tau_{i}}_{i=0}^\ell,\tcb{\varphi_{i}}_{i=0}^\ell) \in  {_{v_0,\eta_0,\be_0}^{\gam_0,\mu_0,\sig_0}}\X_s\times\W_s$.
    Since $\Upupsilon$ acts as the identity map on the $\W_s$ factor, we have that 
    \begin{equation}\label{comer atun}
    (\tcb{\tilde{\tau}_{i}}_{i=0}^\ell,\tcb{\tilde{\varphi}_i}_{i=0}^\ell)= (\tcb{\tau_i}_{i=0}^\ell,\tcb{\varphi_i}_{i=0}^\ell) \in \W_s.    
    \end{equation}
    The remaining piece of \eqref{comer las mandarinas} manifests as the equation
    \begin{multline}\label{identity smells funny}
        D_2\Uppsi(\gam_0,\mu_0,\sig_0,v_0,\eta_0)[v,\eta]+\sum_{i=0}^\ell\tp{0,\eta_0^i\tp{\m{Tr}\tau_{0i}I+(\mu_0+\sig_0)\tau_{0i}}\grad\eta+i\eta_0^{i-1}\tp{\tau_{0i}\grad\eta_0+\varphi_{0i}}\eta}\\=(h,f)-\sum_{i=0}^\ell\eta_0^i(0,(\m{Tr}\tau_i I+(\mu_0+\sig_0)\tau_i)\grad\eta_0+\varphi_i),
    \end{multline}
    which, after defining $\uptau_0$, $\be_0$, and $\upvarphi_0$ as in~\eqref{dictionary} is equivalent to
    \begin{multline}\label{comer los pollos}
        L^{\gam,\mu,\sig}_{v_0,\eta_0,\be_0}(v,\eta)+(0,\upvarphi_0\eta+(\mu_0+\sig_0)\uptau_0\grad\eta)+(0,R^{\gam,\mu,\sig}_{v_0,\eta_0}[v,\eta])\\=(h,f)-\sum_{i=0}^\ell\eta_0^i(0,(\m{Tr}\tau_i I+(\mu_0+\sig_0)\tau_i)\grad\eta_0+\varphi_i),
    \end{multline}
    thanks to the the first item of Proposition~\ref{prop on derivative splitting} and Definition~\ref{definition of operators and adapted domains I}.  Corollary~\ref{coro thm on synthesis 1.5 omnisonic case} then guarantees that $(h,f)-\sum_{i=0}^\ell\eta_0^i(0,(\m{Tr}\tau_i I+(\mu_0+\sig_0)\tau_i)\grad\eta_0+\varphi_i) \in \Y_s$, from which it immediately follows that $(h,f) \in \Y_s$.  Hence, $D_2\Upupsilon(\bf{q}_0)({_{v_0,\eta_0,\be_0}^{\gam_0,\mu_0,\sig_0}}\X_s\times\W_s) \subseteq \Y_s\times\W_s$.

    To complete the proof of the claim, it remains to show that \eqref{comer un sandwich} is an isomorphism obeying the estimate \eqref{bound for step 1}.  Once more we consider the equation \eqref{comer las mandarinas}, but now with the aim of solving it for given data $(h,f,\tcb{\tilde{\tau}_i}_{i=0}^\ell,\tcb{\tilde{\varphi}_i}_{i=0}^\ell) \in \Y_s \times \W_s$.  The identity \eqref{comer atun} remains true, and so again we reduce to \eqref{comer los pollos}.   Applying Corollary~\ref{coro thm on synthesis 1.5 omnisonic case}, we obtain the existence and uniqueness of $(v,\eta)\in{^{\gam_0,\mu_0,\sig_0}_{v_0,\eta_0,\be_0}}\X_s$ satisfying~\eqref{comer los pollos}. Furthermore,  the proposition guarantees that the solution obeys the bound
    \begin{multline}\label{initial bound on the solution}
        \tnorm{v,\eta}_{{_{\mu_0,\sig_0}}\X_s}\lesssim\bnorm{(h,f)-\sum_{i=0}^\ell\eta_0^i(0,(\m{Tr}\tau_i I+(\mu_0+\sig_0)\tau_i)\grad\eta_0+\varphi_i)}_{\Y_s}\\+\tnorm{v_0,\eta_0,\tcb{\tau_{0i}}_{i=0}^\ell,\tcb{\varphi_{0i}}_{i=0}^\ell}_{\X_{s+3}\times\W_{s+3}}\bnorm{(h,f)-\sum_{i=0}^\ell\eta_0^i(0,(\m{Tr}\tau_i I+(\mu_0+\sig_0)\tau_i)\grad\eta_0+\varphi_i)}_{\Y_b},
    \end{multline}
    but thanks to Corollary~\ref{coro on more algebra properties} and Proposition~\ref{corollary on tame estimates on simple multipliers}, estimate~\eqref{initial bound on the solution} immediately provides the bound
    \begin{multline}\label{comer camarones}
        \tnorm{v,\eta}_{{_{\mu_0,\sig_0}}\X_s}\lesssim\tnorm{h,f,\tcb{\tau_i}_{i=0}^\ell,\tcb{\varphi_i}_{i=0}^\ell}_{\Y_s\times\W_s}\\+\tnorm{v_0,\eta_0,\tcb{\tau_{0i}}_{i=0}^\ell,\tcb{\varphi_{0i}}_{i=0}^\ell}_{\X_{s+3}\times\W_{s+3}}\tnorm{h,f,\tau,\tcb{\varphi_i}_{i=0}^\ell,\tcb{\psi_i}_{i=0}^\ell}_{\Y_b\times\W_b}.
    \end{multline}
    As usual, the implicit constants here depend only on $s$, $d$, $I$, $J$, and $K$.  Upon combining the above existence and uniqueness results, \eqref{comer atun}, and estimate \eqref{comer camarones} we then readily deduce that \eqref{comer un sandwich} is an isomorphism obeying the estimate \eqref{bound for step 1}, which completes the proof of the claim and concludes the first step of the proof.
      
    We now move on to the second step of the proof, in which the goal is to attain the theorem statement by building on the results of the first step.   Let $\bf{p}_0$ be as in \eqref{thm on synthesis, II p0} and let $\bf{q}_0$ be as in the first step.  The system 
    \begin{equation}\label{the og equation}
        D\Bar{\Uppsi}(\bf{p}_0)(\gam,\mu,\sig,\tcb{\tau_i}_{i=0}^\ell,\tcb{\varphi_i}_{i=0}^\ell,w,\xi,v,\eta)=(\tilde{\gam},\tilde{\mu},\tilde{\sig},h,f,\tcb{\tilde{\tau}_i}_{i=0}^\ell,\tcb{\tilde{\varphi}_i}_{i=0}^\ell,u,\zeta)
    \end{equation}
    is equivalent to the following system of four equations:
    \begin{equation}\label{remaining equation 1}
        (\gam,\mu,\sig) = (\tilde{\gam},\tilde{\mu},\tilde{\sig}),
    \end{equation}
    \begin{equation}\label{remaining equation 2}
       D_1\Upupsilon(\bf{q}_0)(\gam,\mu,\sig) +  D_2\Upupsilon(\bf{q}_0)(v,\eta,\tcb{\tau_i}_{i=0}^\ell,\tcb{\varphi_i}_{i=0}^\ell)= (h,f,\tcb{\tilde{\tau}_i}_{i=0}^\ell,\tcb{\tilde{\varphi}_i}_{i=0}^\ell),
    \end{equation}
     \begin{equation}\label{remaining equation 3}
        w=\tbr{\mu_0\grad}^2v-u+2\mu\mu_0|\grad|^2v_0,
    \end{equation}
    and
    \begin{equation}\label{remaining equation 4}
        \xi=(1+(\mu_0+\sig_0)|\grad|)\tbr{\sig_0\grad}^2\eta-\zeta+(\mu+\sig)|\grad|\tbr{\sig_0\grad}^2\eta_0+(1+(\mu_0+\sig_0)|\grad|)\sig\sig_0|\grad|^2\eta_0.
    \end{equation}
    For brevity we denote in what follows
    \begin{equation}
        \bf{e}=(\gam,\mu,\sig,\tcb{\tau_i}_{i=0}^\ell,\tcb{\varphi_i}_{i=0}^\ell,w,\xi,v,\eta)
        \text{ and }
        \bf{f}=(\tilde{\gam},\tilde{\mu},\tilde{\sig},h,f,\tcb{\tilde{\tau}_i}_{i=0}^\ell,\tcb{\tilde{\varphi}_i}_{i=0}^\ell,u,\zeta).
    \end{equation}

    If $\bf{e} \in \R^3\times \W_s\times \X_s\times[{_{v_0,\be_0,\eta_0}^{\gam_0,\mu_0,\sig_0}}\X_s]$, then we use the first step on \eqref{remaining equation 2} and trivial estimates on \eqref{remaining equation 1}, \eqref{remaining equation 3}, and \eqref{remaining equation 4} to see that $\bf{f} \in \F_s$ and 
    \begin{equation}
        \norm{\bf{f}}_{\F_s} \lesssim \norm{\bf{e} }_{\W_s\times \X_s\times[{_{v_0,\be_0,\eta_0}^{\gam_0,\mu_0,\sig_0}}\X_s]}.
    \end{equation}
    This shows that the linear map \eqref{adapted isomorphism of the nonlinear operators derivative} is well-defined and bounded.  

    Now suppose that  $\bf{f} \in \F_s$.  We seek to find a unique $\bf{e} \in \W_s\times \X_s\times[{_{v_0,\be_0,\eta_0}^{\gam_0,\mu_0,\sig_0}}\X_s]$ solving \eqref{the og equation}.  The identity \eqref{remaining equation 1} uniquely selects $(\gam,\mu,\sig)$, and then \eqref{remaining equation 2} and the work from the previous step uniquely selects $(v,\eta,\tcb{\tau_i}_{i=0}^\ell,\tcb{\varphi_i}_{i=0}^\ell) \in {_{v_0,\eta_0,\be_0}^{\gam_0,\mu_0,\sig_0}}\X_s\times\W_s$.  From here it's a simple matter to use \eqref{remaining equation 3} and \eqref{remaining equation 4} to uniquely select $(w,\xi) \in \X_s$.  This process yields the desired $\bf{e}$ and completes the proof of the first item.  For the second item we employ~\eqref{bound for step 1}  to see that the pair $(v,\eta)$ determined by this process obeys the bound  
    \begin{multline}\label{along the way}
        \tnorm{v,\eta}_{{_{\mu_0,\sig_0}}\X_s}\lesssim\tnorm{\gam,\mu,\sig,h,f,\tcb{\tau_i}_{i=0}^\ell,\tcb{\varphi_i}_{i=0}^\ell}_{\R^3\times\Y_s\times\W_s}\\+\tbr{\tnorm{\bf{p}_0}_{\E_{s+3}}}\tnorm{\gam,\mu,\sig,h,f,\tcb{\tau_i}_{i=0}^\ell,\tcb{\varphi_i}_{i=0}^\ell}_{\R^3\times\Y_b\times\W_b},
    \end{multline}
    which, together with \eqref{remaining equation 3} and \eqref{remaining equation 4},  then provides the estimate  
    \begin{equation}\label{w and xi estimates}
        \tnorm{w,\xi}_{\X_s}\lesssim\tnorm{\bf{f}}_{\F_s}+\tbr{\tnorm{\bf{p}_0}_{\E_{3+s}}}\tnorm{\bf{f}}_{\F_b}.
    \end{equation}
    Combining these two estimates with the trivial bounds obtained from \eqref{remaining equation 1} then completes the proof of the second item.
\end{proof}

The remainder of this subsection is meant to provide subsonic analogs of the previous results. In the following definition, we recall that the space ${_{\mu,\sig}}\Xs_s$ is defined in~\eqref{sybil} and the operator $P^{\gam,\mu,\sig}_{v_0,\eta_0}$ is given in the first item of Proposition~\ref{prop on derivative splitting}.

\begin{defn}[Operators and adapted domains, subsonic case]\label{definition of operators and adapted domains II}
Given $\gam$, $\mu$, $\sig$, $v_0$, and $\eta_0$ as in Definition~\ref{parameters definition} and $\N\ni s\ge 3$ we define the space ${^{\gam,\mu,\sig}_{v_0,\eta_0}}\Xs_s=\tcb{(v,\eta)\in{{_{\mu,\sig}}\Xs}\;:\;P^{\gam,\mu,\sig}_{v_0,\eta_0}[v,\eta]\in\Y_s}$ and equip it with the graph norm
\begin{equation}\label{subsonic space graph norm}
    \tnorm{v,\eta}_{{^{\gam,\mu,\sig}_{v_0,\eta_0}}\Xs_s}=\tnorm{v,\eta}_{{_{\mu,\sig}}\Xs_s}+\tnorm{P^{\gam,\mu,\sig}_{v_0,\eta_0}[v,\eta]}_{\Y_s}.
\end{equation}
\end{defn}

The following result is the subsonic refinement of Proposition~\ref{prop on the first part of the synthesis}, as such the subscript $\m{SS}$ is an acronym meaning `subsonic synthesis'.

\begin{prop}[Synthesis I, subsonic case]\label{prop on syn 1 sub}
    Let $(\gam, \mu, \sig) \in \mathcal{Q}$ with $\mathcal{Q}=I\times J\times K$, as in Definition~\ref{parameters definition}, but suppose additionally that $I\Subset(0,1)$. For each $\N\ni s\ge\max\tcb{3,2+\tfloor{d/2}}$ there exists $\R^+\ni \tilde{\rho}_{\m{SS}s}\le\rho_{\m{SAP}s}$, depending only on $s$ and $\mathcal{Q}$, such that if $\R^+\ni\ep\in(0,\tilde{\rho}_{\m{SS}s}]$,
    \begin{equation}\label{subsonic data special inclusion}
        (v_0,\eta_0,\uptau_0,\upvarphi_0)\in B_{\Xs_r\times H^r\times H^r}(0,\ep)\cap\tp{\X_\infty\times H^\infty(\R^d;\R^{d\times d})\times H^\infty(\R^d;\R^d)},
    \end{equation}
    with $r=\max\tcb{6,5+\tfloor{d/2}}$, then the following hold.
    \begin{enumerate}
        \item The following linear map is well-defined and bounded:
        \begin{equation}\label{dies illa sub}
            {^{\gam,\mu,\sig}_{v_0,\eta_0}}\Xs_s\ni(v,\eta)\mapsto P^{\gam,\mu,\sig}_{v_0,\eta_0}[v,\eta]+(0,\upvarphi_0\eta+\uptau_0\grad\eta)+(0,R^{\gam,\mu,\sig}_{v_0,\eta_0}[v,\eta])\in\Y_s.
        \end{equation}
        Moreover, this map is an isomorphism.
        \item  For all $(h,f)\in\Y_s$ the unique $(v,\eta)\in{{_{v_0,\eta_0}^{\gam,\mu,\sig}}\Xs_s}$ satisfying the system
    \begin{equation}\label{equations of the first synthesis sub}
    P^{\gam,\mu,\sig}_{v_0,\eta_0}[v,\eta]+(0,\upvarphi_0\eta+\uptau_0\grad\eta)+(0,R^{\gam,\mu,\sig}_{v_0,\eta_0}[v,\eta])=(h,f)
    \end{equation}
    obeys the tame estimate
    \begin{equation}\label{estimate of the first synthesis sub}
        \tnorm{v,\eta}_{{_{\mu,\sig}}\Xs_s}\lesssim\tnorm{h,f}_{\Y_s}+\tnorm{v_0,\eta_0,\uptau_0,\upvarphi_0}_{\Xs_{2+s}\times H^{2+s}\times H^{2+s}}\tnorm{h,f}_{\Y_{\max\tcb{3,2+\tfloor{d/2}}}},
    \end{equation}
    where the implicit constant depends only on $s$, $d$, and $\mathcal{Q}$. 
    \end{enumerate}
\end{prop}
\begin{proof}
    The proof exactly follows the strategy we employed to prove Proposition~\ref{prop on the first part of the synthesis}, namely a double use of the method of continuity.  As such, we will only provide a sketch of the argument here.

    It is routine to verify that the map~\eqref{dies illa sub} is well-defined and bounded. In order to see that this map is an isomorphism obeying estimate~\eqref{estimate of the first synthesis sub}, we again work in two steps.  In the first step we apply the method continuity to an auxiliary homotopy of operators, which are defined as follows. For $\N\ni s\ge 3$ we define the space ${^{\gam,\mu,\sig}_{v_0,\eta_0}}\tilde{\Xs_s}=\tcb{(v,\eta)\in{_{\mu,\sig}}\Xs_s\;:\;Q^{\gam,\mu,\sig}_{v_0,\eta_0}[v,\eta]\in\Y_s}$, where $Q^{\gam,\mu,\sig}_{v_0,\eta_0}$ is the operator from the second item of Proposition~\ref{prop on derivative splitting}, and equip it with the graph norm 
    \begin{equation}
        \tnorm{v,\eta}_{{_{v_0,\eta_0}^{\gam,\mu,\sig}}\tilde{\Xs_s}}=\tnorm{v,\eta}_{{_{\mu,\sig}}\Xs_s}+\tnorm{Q^{\gam,\mu,\sig}_{v_0,\eta_0}[v,\eta]}_{\Y_s}.
    \end{equation}
    The auxiliary mappings are then defined, for $t\in[0,1]$, to be 
    \begin{equation}\label{well-defined and bounded here are the mappings ive never defined a mapping thats not well defined because tautology}
        {^{\gam,\mu,\sig}_{v_0,\eta_0}}\tilde{\Xs_s}\ni(v,\eta)\mapsto Q^{\gam,\mu,\sig}_{v_0,\eta_0}[v,\eta]+t\tp{0,\upvarphi_0\eta+\uptau_0\grad\eta}+t(0,S^{\gam,\mu,\sig}_{v_0,\eta_0}[v])\in\Y_s.
    \end{equation}
    By combining the a priori estimates of Theorem~\ref{a priori estimates subsonic principal part}, the existence of Theorem~\ref{main theorem of all of math in the math document where is it}, and the remainder estimates of Proposition~\ref{prop on remainder estimates} we verify that the mappings~\eqref{well-defined and bounded here are the mappings ive never defined a mapping thats not well defined because tautology} are well-defined and bounded and find $\tilde{\rho}_{\m{SS}s}\in\R^+$ as in the proposition statement such that for $\ep\le\tilde{\rho}_{\m{SS}s}$ the maps~\eqref{well-defined and bounded here are the mappings ive never defined a mapping thats not well defined because tautology} obey the requisite tame closed range estimates uniformly in $t\in[0,1]$.  The method of continuity then applies, and we may indeed argue as in the first step of the proof of Proposition~\ref{prop on the first part of the synthesis}.

    The second step of the proof builds on the previous step by considering a different operator homotopy, namely
    \begin{equation}\label{Stricte discussurus}
        {^{\gam,\mu,\sig}_{v_0,\eta_0}}\Xs_s\ni(v,\eta)\mapsto P^{\gam,\mu,\sig}_{v_0,\eta_0}[v,\eta]+(0,\upvarphi_0\eta+\uptau_0\grad\eta)+t(0,R^{\gam,\mu,\sig}_{v_0,\eta_0}[v,\eta])\in\Y_s,\quad t\in[0,1].
    \end{equation}
    Supplementing the previous step with more remainder estimates from Proposition~\ref{prop on remainder estimates} allows us to take $\uprho_{\m{SS}s}$ smaller, if necessary, and deduce the claimed bounds~\eqref{estimate of the first synthesis sub} and also invoke the method of continuity on~\eqref{Stricte discussurus}, leading us to a complete proof of the first and second items.
\end{proof}

\begin{coro}[Synthesis II, subsonic case]\label{coro 1.5 synthesis subsonic case}
    Let $(\gam, \mu, \sig) \in \mathcal{Q}$ with $\mathcal{Q}=I\times J\times K$, as in Definition~\ref{parameters definition}, but suppose additionally that $I\Subset(0,1)$.  For each $\N\ni s\ge\max\tcb{3,2+\tfloor{d/2}}$ there exists $\R^+\ni \rho_{\m{SS}s}\le\tilde{\rho}_{\m{SS}s}$, depending only on $s$ and $\mathcal{Q}$, such that if $\R^+\ni\ep\in(0,\rho_{\m{SS}s}]$ and $(v_0,\eta_0,\tcb{\tau_{0i}}_{i=0}^\ell,\tcb{\varphi_{0i}}_{i=0}^\ell)$ satisfy~\eqref{a new condition for tuba mirum} then the following hold.
    \begin{enumerate}
        \item Upon defining $\uptau_0$, $\be_0$, and $\upvarphi_0$ as in~\eqref{dictionary}, we have that the linear map~\eqref{dies illa sub} is well-defined, bounded, and an isomorphism.
        \item  For all $(h,f)\in\Y_s$ the unique $(v,\eta)\in{{_{v_0,\eta_0}^{\gam,\mu,\sig}}\Xs_s}$ satisfying~\eqref{equations of the first synthesis sub} obeys the tame estimate
    \begin{equation}\label{estimate of the first synthesis sub x}
        \tnorm{v,\eta}_{{_{\mu,\sig}}\Xs_s}\lesssim\tnorm{h,f}_{\Y_s}+\tnorm{v_0,\eta_0,\tcb{\tau_{0i}}_{i=0}^\ell,\tcb{\varphi_{0i}}_{i=0}^\ell}_{\Xs_{2+s}\times\W_{2+s}}\tnorm{h,f}_{\Y_{\max\tcb{3,2+\tfloor{d/2}}}},
    \end{equation}
    where the implicit constant depends only on $s$, $d$, and $\mathcal{Q}$. 
    \end{enumerate}
\end{coro}
\begin{proof}
    We derive these claims from Proposition~\eqref{prop on syn 1 sub} by arguing as in the proof of Corollary~\ref{coro thm on synthesis 1.5 omnisonic case}.  
\end{proof}

\begin{thm}[Synthesis III, subsonic case]\label{thm on syn 2 sub}
    Set $b=\max\tcb{3,2+\tfloor{d/2}}$, and let $\Bar{\Uppsi}^{\m{sub}} : \E_{b} \to \F_{b-2}$ be the map given in~\eqref{uppsi bar sub}, which is well-defined and smooth in light of the second item of Theorem~\ref{thm on smooth-tameness}. Let $\N\ni s\ge b$ and $\rho_{\m{SS}s}\in \R^+$ be given by Corollary~\ref{coro 1.5 synthesis subsonic case}.  Suppose that  $0<\ep\le \min\{\rho_{\m{SS}s},\rho_{\m{WD}} \}$ (recall that the latter is defined in Lemma~\ref{the preliminary lemma for you and I}),  $I\Subset(0,1)$ and $J,K\Subset\R$ are bounded open intervals, and
    \begin{equation}\label{thm on synthesis, II p0 sub}
        \bf{p}_0=(\gam_0,\mu_0,\sig_0,\tcb{\tau_{0i}}_{i=0}^\ell,\tcb{\varphi_{0i}}_{i=0}^\ell,w_0,\xi_0,v_0,\eta_0)\in [I\times J\times K\times B_{\W_r}(0,\ep)\times B_{\Xs_r\times\Xs_r}(0,\ep)]\cap\Es_\infty,
    \end{equation}
    where $r=\max\tcb{6,5+\tfloor{d/2}}$. The following hold.
    \begin{enumerate}
        \item  The restriction of $D\Bar{\Uppsi}^{\m{sub}}(\bf{p}_0) $ to $\R^3\times \W_s\times \Xs_s\times[{_{v_0,\eta_0}^{\gam_0,\mu_0,\sig_0}}\Xs_s]$ is well-defined and takes values in $\Fs_s$, and the induced linear map
        \begin{equation}\label{adapted isomorphism of the nonlinear operators derivative sub}
            D\Bar{\Uppsi}^{\m{sub}}(\bf{p}_0):\R^3\times \W_s\times \Xs_s\times[{_{v_0,\eta_0}^{\gam_0,\mu_0,\sig_0}}\Xs_s]\to\Fs_s
        \end{equation}
        is  an isomorphism. 
        \item Assume that $\bf{f}\in\Fs_s$ and $\bf{e}\in \R^3\times \W_s\times \Xs_s\times[{_{v_0,\eta_0}^{\gam_0,\mu_0,\sig_0}}\Xs_s]$ is defined via $\bf{e}=(D\Bar{\Uppsi}^{\m{sub}}(\bf{p}_0))^{-1}\bf{f}$. We then have the tame estimate
        \begin{equation}\label{fly me to the moon sub}
            \tnorm{\bf{e}}_{\Es_s}\lesssim\tnorm{\bf{f}}_{\Fs_s}+\tbr{\tnorm{\bf{p}_0}_{\Es_{s+2}}}\tnorm{\bf{f}}_{\Fs_{b}}
        \end{equation}
        for an implicit constant depending only on $s$, $d$, $I$, $J$, and $K$.
    \end{enumerate}
\end{thm}
\begin{proof}
    The same strategy used in the proof of Theorem~\ref{thm on synthesis, II} applies here, with minor modifications; for instance, instead of the omnisonic results such as Corollary~\ref{coro thm on synthesis 1.5 omnisonic case}, we instead invoke the corresponding subsonic result from Corollary~\ref{coro 1.5 synthesis subsonic case}.  In the interest of brevity, we omit further details.
\end{proof}

 % _+__+_ -_+__+_ -_+__+_ -_+__+_ -_+__+_ -_+__+_ -_+__+_ -_+__+_ -_+__+_ -_+__+_ -_+__+_ -_+__+_ -_+__+_ -
\section{Conclusion}\label{section on conclusion}
% _+__+_ -_+__+_ -_+__+_ -_+__+_ -_+__+_ -_+__+_ -_+__+_ -_+__+_ -_+__+_ -_+__+_ -_+__+_ -_+__+_ -_+__+_ -

In this section we piece together the nonlinear analysis of Section~\ref{section on nonlinear analysis} and the linear analysis of Section~\ref{section on linear analysis} along with the Nash-Moser inverse function theorem of Appendix~\ref{subsection on a NMIFT} to prove Theorems~\ref{main theorem no. 2} and~\ref{main theorem no. 3} along with Corollaries~\ref{coro on fixed physical parameters} and~\ref{coro on limits}. 

% _+__+_ -_+__+_ -_+__+_ -_+__+_ -_+__+_ -_+__+_ -_+__+_ -_+__+_ -_+__+_ -_+__+_ -_+__+_ -_+__+_ -_+__+_ -
\subsection{Abstract construction}
% _+__+_ -_+__+_ -_+__+_ -_+__+_ -_+__+_ -_+__+_ -_+__+_ -_+__+_ -_+__+_ -_+__+_ -_+__+_ -_+__+_ -_+__+_ -

In the following first formulations of the main theorem for the omnisonic and subsonic cases, we will utilize the following variations on the main nonlinear operators $\Bar{\Uppsi}$ and $\Bar{\Uppsi}^{\m{sub}}$, which we recall are defined in~\eqref{definition of uppsi bar} and~\eqref{uppsi bar sub}, respectively.  For $b = \max\tcb{3,2+\tfloor{d/2}} \in \N$ and any fixed $\mathfrak{q}=(\Bar{\gam},\Bar{\mu},\Bar{\sig})\in\R^+\times\R^2$ we define the operators $\Bar{\Uppsi}_\mathfrak{q}:\E_{b}\to\F_{b-3}$ and $\Bar{\Uppsi}_{\mathfrak{q}}^{\m{sub}}:\Es_b\to\Fs_{b-2}$ via the assignments
\begin{equation}\label{the fraky q dood}
    \Bar{\Uppsi}_{\mathfrak{q}}(\gam,\mu,\sig,\tcb{\tau_i}_{i=0}^\ell,\tcb{\varphi_i}_{i=0}^\ell,w,\xi,v,\eta)=\Bar{\Uppsi}(\gam+\Bar{\gam},\mu+\Bar{\mu},\sig+\Bar{\sig},\tcb{\tau_i}_{i=0}^\ell,\tcb{\varphi_i}_{i=0}^\ell,w,\xi,v,\eta)-(\Bar{\gam},\Bar{\mu},\Bar{\sig},0,0,0)
\end{equation}
and 
\begin{equation}\label{the fraky q sub dood}
    \Bar{\Uppsi}^{\m{sub}}_{\mathfrak{q}}(\gam,\mu,\sig,\tcb{\tau_i}_{i=0}^\ell,\tcb{\varphi_i}_{i=0}^\ell,w,\xi,v,\eta)=\Bar{\Uppsi}^{\m{sub}}(\gam+\Bar{\gam},\mu+\Bar{\mu},\sig+\Bar{\sig},\tcb{\tau_i}_{i=0}^\ell,\tcb{\varphi_i}_{i=0}^\ell,w,\xi,v,\eta)-(\Bar{\gam},\Bar{\mu},\Bar{\sig},0,0,0)
\end{equation}
The utility of these definitions is that $\Bar{\Uppsi}_{\mathfrak{q}}(0)=\Bar{\Uppsi}^{\m{sub}}_{\mathfrak{q}}(0)=0$.

\begin{thm}[Abstract formulation, omnisonic case]\label{thm1 on traveling waves} Let $\varsigma=19+2\tfloor{d/2}$. For each $\mathfrak{q}=(\Bar{\gam},\Bar{\mu},\Bar{\sig})\in\R^+\times\R^2$, there exists a nonincreasing sequence $\tcb{\uprho_{s}(\mathfrak{q})}_{s=0}^\infty\subset(0,\min\tcb{\rho_{\m{WD}},\Bar{\gam}/2}]$ and $\upkappa(\mathfrak{q})\in\R^+$
with the following properties.
    \begin{enumerate}
        \item \emph{Existence and uniqueness:} Given $\bf{f}\in B_{\F_\varsigma}(0,\uprho_0(\mathfrak{q}))$ there exists a unique $\bf{e}\in B_{\E_\varsigma}(0,\upkappa(\mathfrak{q})\uprho_0(\mathfrak{q}))$ such that $\Bar{\Uppsi}_{\mathfrak{q}}(\bf{e})=\bf{f}$. This induces the local inverse map
        \begin{equation}\label{the local inverse maps}
            \Bar{\Uppsi}^{-1}_{\mathfrak{q}}:B_{\F_{\varsigma}}(0,\uprho_0(\mathfrak{q})) \to 
            \Bar{\Uppsi}^{-1}_{\mathfrak{q}}(B_{\F_{\varsigma}}(0,\uprho_0(\mathfrak{q}))) \subseteq
            B_{\E_{\varsigma}}(0,\upkappa(\mathfrak{q})\uprho_0(\mathfrak{q})).
        \end{equation}
        \item \emph{Higher regularity, given low norm smallness:} If $s\in\N$ and $\bf{f}\in \F_{s+\varsigma}\cap B_{\F_{\varsigma}}(0,\uprho_s(\mathfrak{q}))$, then we have the inclusion $\Bar{\Uppsi}_{\mathfrak{q}}^{-1}(\bf{f})\in\E_{s+\varsigma}$ and the tame estimate
        \begin{equation}\label{tame estimates on the local inverse map}
        \tnorm{\Bar{\Uppsi}_{\mathfrak{q}}^{-1}(\bf{f})}_{\E_{\varsigma+s}}\lesssim\tnorm{\bf{f}}_{\F_{\varsigma+s}},
        \end{equation}
        for an implicit constant depending only on $s$, the dimension, and $\mathfrak{q}$. 
        
        \item \emph{Continuous dependence:} For every $s\in\N$, the restricted map
        \begin{equation}\label{high regularity local inverse map}
            \Bar{\Uppsi}_{\mathfrak{q}}^{-1}:\F_{s+\varsigma}\cap B_{\F_{\varsigma}}(0,\uprho_s(\mathfrak{q}))\to\E_{\varsigma+s}\cap B_{\E_{\varsigma}}(0,\upkappa(\mathfrak{q})\uprho_0(\mathfrak{q}))
        \end{equation}
        is continuous with respect to the norms on $\F_{s+\varsigma}$ and $\E_{s+\varsigma}$.  
    \end{enumerate}
\end{thm}
\begin{proof}
    Our aim is to show that the hypotheses of Theorem~\ref{thm on nmh} are satisfied by the map $\Bar{\Uppsi}_{\mathfrak{q}}$, defined in~\eqref{the fraky q dood}, acting between the Banach scales $\pmb{\E}$ and $\pmb{\F}$, which we recall are defined in~\eqref{the banach scales for our problem}. Evidently, these Banach scales consist of reflexive (in fact, Hilbert) spaces and, thanks to Lemma~\ref{lem on LP smoothability of the Banach scales}, we know that $\pmb{\E}$ and $\pmb{\F}$ are LP-smoothable in the sense of Definition~\ref{defn of smoothable and LP-smoothable Banach scales}.

    We now assess how the triple $(\pmb{\E},\pmb{\F},\Bar{\Uppsi}_{\mathfrak{q}})$ satisfies the LRI mapping hypotheses of Definition~\ref{defn of the mapping hypotheses} with parameters  
    \begin{equation}
        (\mu,r,R)=(3,6+\tfloor{d/2},33+3\tfloor{d/2}),
    \end{equation}
    which obey the inequalities $1\le\mu\le r<R<\infty$ and $2(r+\mu)+1<(r+R)/2$ as well as the identity $\varsigma=2(r+\mu)+1$. We then set $\del_r=\min\tcb{\rho_{\m{WD}},\Bar{\gamma}/2}$. By heeding to these definitions, Theorem~\ref{thm on smooth-tameness} and Corollary~\ref{lem on derivative estimates on tame maps}, we see that the first and second items of Definition~\ref{defn of the mapping hypotheses} are satisfied. The satisfaction of the third and final item of Definition~\ref{defn of the mapping hypotheses} is a consequence of Theorem~\ref{thm on synthesis, II} and Remark~\ref{rmk on nondecreasing} with $\del_R=\min\tcb{\del_r,\rho_{\m{S}R}}$ (the parameter $\rho_{\m{S}R}$ is granted by the theorem) and $I=(\Bar{\gam}/2,2\Bar{\gam})$, $J=(-2|\Bar{\mu}|,2|\Bar{\mu}|)$, and $K=(-2|\Bar{\sig}|,2|\Bar{\sig}|)$.

    Thus, the hypotheses of Theorem~\ref{thm on nmh} are satisfied and we therefore obtain $\ep=\uprho_0(\mathfrak{q})$ such that the first item holds, the map~\eqref{the local inverse maps} is continuous, and we have the estimate $\tnorm{\Bar{\Uppsi}_{\mathfrak{q}}^{-1}(\bf{f})}_{\E_\varsigma}\lesssim\tnorm{\bf{f}}_{\F_\varsigma}$ for $\bf{f}\in B_{\F_\varsigma}(0,\uprho_0(\mathfrak{q}))$.

    Finally, suppose that $s\in\N$. We again apply the Nash-Moser theorem, but this time with the parameter triple $(\mu,r,\tilde{R})=(\mu,r,R+s)$ and $\del_{\tilde{R}}=\min\tcb{\del_R,\rho_{\m{S}\tilde{R}}}$ and let $\uprho_s(\mathfrak{q})\in(0,\rho_{\m{WD}}]$ denote the smallness parameter $\ep$ provided by this application of Theorem~\ref{thm on nmh}. By local uniqueness, we are thus granted the second and third items.  
\end{proof}

We are now interested in the subsonic sharpening of Theorem~\ref{thm1 on traveling waves}, which we state and prove below.

\begin{thm}[Abstract formulation, subsonic case]\label{thm1 on traveling waves subsonic} Let $\varpi=17+2\tfloor{d/2}$. For each $\mathfrak{q}=(\Bar{\gam},\Bar{\mu},\Bar{\sig})\in(0,1)\times\R^2$, there exists nonincreasing $\tcb{\uprho^{\m{sub}}_{s}(\mathfrak{q})}_{s=0}^\infty\subset(0,\min\tcb{\rho_{\m{WD}},\Bar{\gam}/2,(1-\Bar{\gam})2}]$ and $\upkappa^{\m{sub}}(\mathfrak{q})\in\R^+$
with the following properties.
    \begin{enumerate}
        \item \emph{Existence and uniqueness:} For $\bf{f}\in B_{\Fs_\varpi}(0,\uprho^{\m{sub}}_0(\mathfrak{q}))$ there exists a unique $\bf{e}\in B_{\Es_\varpi}(0,\upkappa^{\m{sub}}(\mathfrak{q})\uprho^{\m{sub}}_0(\mathfrak{q}))$ such that $\Bar{\Uppsi}^{\m{sub}}_{\mathfrak{q}}(\bf{e})=\bf{f}$. This induces the local inverse map
        \begin{equation}\label{the local inverse maps sub}
            [\Bar{\Uppsi}^{\m{sub}}_{\mathfrak{q}}]^{-1}:B_{\Fs_{\varpi}}(0,\uprho^{\m{sub}}_0(\mathfrak{q})) \to 
            [\Bar{\Uppsi}^{\m{sub}}_{\mathfrak{q}}]^{-1}(B_{\Fs_{\varpi}}(0,\uprho^{\m{sub}}_0(\mathfrak{q}))) \subseteq
            B_{\Es_{\varpi}}(0,\upkappa^{\m{sub}}(\mathfrak{q})\uprho^{\m{sub}}_0(\mathfrak{q})).
        \end{equation}
        \item \emph{Higher regularity, given low norm smallness:} If $s\in\N$ and $\bf{f}\in \Fs_{s+\varpi}\cap B_{\Fs_{\varpi}}(0,\uprho^{\m{sub}}_s(\mathfrak{q}))$, then we have the inclusion $[\Bar{\Uppsi}^{\m{sub}}_{\mathfrak{q}}]^{-1}(\bf{f})\in\Es_{s+\varpi}$ and the tame estimate
        \begin{equation}\label{tame estimates on the local inverse map sub}
        \tnorm{[\Bar{\Uppsi}^{\m{sub}}_{\mathfrak{q}}]^{-1}(\bf{f})}_{\Es_{\varpi+s}}\lesssim\tnorm{\bf{f}}_{\Fs_{\varpi+s}},
        \end{equation}
        for an implicit constant depending only on $s$, the dimension, and $\mathfrak{q}$. 
        
        \item \emph{Continuous dependence:} For every $s\in\N$, the restricted map
        \begin{equation}\label{high regularity local inverse map sub}
            [\Bar{\Uppsi}_{\mathfrak{q}}^{\m{sub}}]^{-1}:\Fs_{s+\varpi}\cap B_{\Fs_{\varpi}}(0,\uprho^{\m{sub}}_s(\mathfrak{q}))\to\Es_{\varpi+s}\cap B_{\Es_{\varpi}}(0,\upkappa^{\m{sub}}(\mathfrak{q})\uprho^{\m{sub}}_0(\mathfrak{q}))
        \end{equation}
        is continuous with respect to the norms on $\Fs_{s+\varpi}$ and $\Es_{s+\varpi}$.  
    \end{enumerate}
\end{thm}
\begin{proof}
    The proof is much the same as that of Theorem~\ref{thm1 on traveling waves}; the main differences are that we use the second item of Theorem~\ref{thm on smooth-tameness} in place of the first and Theorem~\ref{thm on syn 2 sub} instead of Theorem~\ref{thm on synthesis, II}. The reason for the difference between $\varsigma$ and $\varpi$ is because the map $\Bar{\Uppsi}$ is $3$-tame, while $\Bar{\Uppsi}^{\m{sub}}$ is only $2$-tame. We omit further details.
\end{proof}

% _+__+_ -_+__+_ -_+__+_ -_+__+_ -_+__+_ -_+__+_ -_+__+_ -_+__+_ -_+__+_ -_+__+_ -_+__+_ -_+__+_ -_+__+_ -
\subsection{PDE construction}\label{section on PDE construction}
% _+__+_ -_+__+_ -_+__+_ -_+__+_ -_+__+_ -_+__+_ -_+__+_ -_+__+_ -_+__+_ -_+__+_ -_+__+_ -_+__+_ -_+__+_ -

The goal of this subsection is to take the abstract formulations of Theorems~\ref{thm1 on traveling waves} and~\ref{thm1 on traveling waves subsonic} and transform them into Theorems~\ref{main theorem no. 2} and~\ref{main theorem no. 3} along with a few corollaries.

\begin{proof}[Proof of Theorem~\ref{main theorem no. 2}]
    For each $\mathfrak{q}=(\Bar{\gam},\Bar{\mu},\Bar{\sig})\in\R^+\times[0,\infty)^2$ we may invoke Theorem~\ref{thm1 on traveling waves} and acquire a nonincreasing sequence $\tcb{\uprho_{\mathfrak{s}}(\mathfrak{q})}_{\mathfrak{s}=0}^\infty\subset(0,\min\tcb{\rho_{\m{WD}},\Bar{\gam}/2}]$ and $\upkappa(\mathfrak{q})\in\R^+$ such that the various conclusions of the theorem hold. We then define the open and relatively open sets
    \begin{equation}
        \mathcal{V}(\mathfrak{q})=B_{\X_\varsigma}(0,\upkappa(\mathfrak{q})\uprho_0(\mathfrak{q}))\subset\X_\varsigma.
    \end{equation}
    and, for $\N \ni s \ge \varsigma$,
    \begin{equation}\label{definition of the omnisonic set Us}
        \mathcal{U}_s=\bcb{(\gam,\mu,\sig,\tcb{\tau_i}_{i=0}^\ell,\tcb{\varphi_i}_{i=0}^\ell)\in\bigcup_{\mathfrak{q}\in\R^+\times[0,\infty)^2}B_{\R^3\times\W_\varsigma}((\mathfrak{q},0),\uprho_{s-\varsigma}(\mathfrak{q}))\;:\mu\ge0,\;\sig\ge0}\subset\R^3\times\W_{\varsigma}.
    \end{equation}
    The first properties of $\mathcal{U}_s$ and $\mathcal{V}(\mathfrak{q})$ in~\eqref{thm3 Us} and~\eqref{thm4 Vq} are now clear.

    For the first item, we apply the first conclusion of Theorem~\ref{thm1 on traveling waves}. Indeed, given $(\gam,\mu,\sig,\tcb{\tau_i}_{i=0}^\ell,\tcb{\varphi_i}_{i=0}^\ell)\in\mathcal{U}_\varsigma$ we define $(v,\eta)\in\mathcal{V}(\gam,\mu,\sig)$ to be the last two components of the tuple $\Bar{\Uppsi}_{(\gam,\mu,\sig)}^{-1}(0,0,\tcb{\tau_i}_{i=0}^\ell,\tcb{\varphi_i}_{i=0}^\ell,0,0)$. This grants us existence and uniqueness, and hence the second item is proved.

    It remains to prove the second and third items, which we do in tandem. It is clear from Theorem~\ref{thm1 on traveling waves} that, given $(\gam,\mu,\sig,\tcb{\tau_i}_{i=0}^\ell,\tcb{\varphi_i}_{i=0}^\ell)\in\mathcal{U}_s\cap(\R^3\times\W_s)$, the solution $(v,\eta)$ belongs also to $\X_{s}$ and the map
            \begin{equation}\label{vanilla solution map but thats okay vanilla is my favorite flavor}
            \mathcal{U}_s\cap(\R^3\times\W_s)\ni(\gam,\mu,\sig,\tcb{\tau_i}_{i=0}^\ell,\tcb{\varphi_i}_{i=0}^\ell)\mapsto(v,\eta)\in\X_s
        \end{equation}
     is continuous. It is at this point where we crucially utilize the last two components of the mapping $\Bar{\Uppsi}$ from~\eqref{definition of uppsi bar}.  Indeed, we have that $\Bar{\Uppsi}(\gam,\mu,\sig,\tcb{\tau_i}_{i=0}^\ell,\tcb{\varphi_i}_{i=0}^\ell,w,\xi,v,\eta)=(\gam,\mu,\sig,0,\tcb{\tau_i}_{i=0}^\ell,\tcb{\varphi_i}_{i=0}^\ell,0,0)$ and, by arguments similar to the above, the map
    \begin{equation}
        \mathcal{U}_s\cap(\R^3\times\W_s)\ni(\gam,\mu,\sig,\tcb{\tau_i}_{i=0}^\ell,\tcb{\varphi_i}_{i=0}^\ell)\mapsto(w,\xi)\in\X_s
    \end{equation}
    is also continuous; however, we have the identities
    \begin{equation}
        w=\tbr{\mu\grad}^2v\quad\text{and}\quad\xi=(1+(\mu+\sig)|\grad|)\tbr{\sig\grad}^2\eta
    \end{equation}
    which, in turn, lead to the inclusion $(v,\eta)\in{_{\mu,\sig}}\X_s$ (which implies~\eqref{reg prom estimate}) and the continuity of the maps~\eqref{the super maps}.
\end{proof}

\begin{proof}[Proof of Theorem~\ref{main theorem no. 3}]
    The proof is very similar to the proof of Theorem~\ref{main theorem no. 2}; the main difference is that we use Theorem~\ref{thm1 on traveling waves subsonic} in place of Theorem~\ref{thm1 on traveling waves}.  We omit further details for the sake of brevity.
\end{proof}

We now shall enumerate some important consequences of Theorems~\ref{main theorem no. 2} and~\ref{main theorem no. 3}.

\begin{coro}[Well-posedness for fixed physical parameters]\label{coro on fixed physical parameters}
    Let $(\gam,\mu,\sigma)\in\R^+\times[0,\infty)^2$. The following hold.
    \begin{enumerate}
        \item There exists a nonincreasing sequence of open sets 
        \begin{equation}\label{slice open set}
            \tcb{0}^{\ell+1}\times\tcb{0}^{\ell+1}\subset\mathcal{U}_s(\gam,\mu,\sig)\subset(H^\varsigma(\R^d;\R^{d\times d}))^{\ell+1}\times(H^\varsigma(\R^d;\R^d))^{\ell+1},\quad \N\ni s\ge\varsigma
        \end{equation}
        such that the following hold.
        \begin{enumerate}
            \item For all $(\tcb{\tau_i}_{i=0}^\ell,\tcb{\varphi_i}_{i=0}^\ell)\in\mathcal{U}_\varsigma(\gam,\mu,\sig)$ there exists a unique $(v,\eta)\in\mathcal{V}(\gam,\mu,\sig)$ such that system~\eqref{traveling wave formulation of the equation} is classically satisfied by the latter tuple with physical parameters $(\gam,\mu,\sig)$ and data $\mathcal{F}$ determined from $\eta$, $\mu+\sig$, and $(\tcb{\tau_i}_{i=0}^\ell,\tcb{\varphi_i}_{i=0}^\ell)$ via the omnisonic case of~\eqref{conditions on the forcing}.
            \item If $\N\ni s\ge\varsigma$ and $(\tcb{\tau_i}_{i=0}^\ell,\tcb{\varphi_i}_{i=0}^\ell)\in\mathcal{U}_s\cap\tp{(H^s(\R^d;\R^{d\times d}))^{\ell+1}\times(H^s(\R^d;\R^d))^{\ell+1}}$
        then the corresponding solution $(v,\eta)\in\mathcal{V}(\gam,\mu,\sig)$ produced by the previous item in fact obeys the inclusions $\tbr{\mu\grad}^2 v\in H^{s}(\R^d;\R^d)$ and $(1+(\mu+\sig)|\grad|)\tbr{\sig\grad}^2\eta\in\mathcal{H}^s(\R^d)$; moreover the mapping from the data to the solution is continuous for the spaces determined via this inclusion.
        \end{enumerate}
        \item If we assume additionally that $\gam<1$, then there exists another nonincreasing sequence of open sets
        \begin{equation}\label{subslice open set}
           \tcb{0}^{\ell+1}\times\tcb{0}^{\ell+1}\subset\mathcal{U}^{\m{sub}}_s(\gam,\mu,\sig)\subset(H^\varpi(\R^d;\R^{d\times d}))^{\ell+1}\times(H^\varpi(\R^d;\R^d))^{\ell+1},\quad\N\ni s\ge\varpi
        \end{equation}
        such that the following hold.
        \begin{enumerate}
            \item For all $(\tcb{\tau_i}_{i=0}^\ell,\tcb{\varphi_i}_{i=0}^\ell)\in\mathcal{U}^{\m{sub}}_\varpi$ there exists a unique $(v,\eta)\in\mathcal{V}^{\m{sub}}(\gam,\mu,\sig)$ such that system~\eqref{traveling wave formulation of the equation} is classically satisfied by the latter tuple with physical parameters $(\gam,\mu,\sig)$ and data $\mathcal{F}$ determined from $\eta$ and $(\tcb{\tau_i}_{i=0}^\ell,\tcb{\varphi_i}_{i=0}^\ell)$ via the subsonic case of~\eqref{conditions on the forcing}.
            \item If $\N\ni s\ge\varpi$ and $(\tcb{\tau_i}_{i=0}^\ell,\tcb{\varphi_i}_{i=0}^\ell)\in\mathcal{U}^{\m{sub}}_s\cap\tp{(H^s(\R^d;\R^{d\times d}))^{\ell+1}\times(H^s(\R^d;\R^d))^{\ell+1}}$ then the corresponding solution $(v,\eta)\in\mathcal{V}^{\m{sub}}(\gam,\mu,\sig)$ produced by the previous item in fact obeys the inclusions $\tbr{\mu\grad}^2 v\in H^{s}(\R^d;\R^d)$ and $\tbr{\sig\grad}^2\eta\in\mathcal{H}^{1+s}(\R^d)$; moreover the mapping from the data to the solution is continuous for the spaces determined via this inclusion.
        \end{enumerate}
    \end{enumerate}
\end{coro}
\begin{proof}
    Throughout the proof we shall let $\m{int}$ denote the topological interior. We begin by defining the open sets of~\eqref{slice open set} and~\eqref{subslice open set}; we set:
    \begin{equation}\label{omni for you}
        \mathcal{U}_s(\gam,\mu,\sig)=\m{int}\tcb{(\tcb{\tau_i}_{i=0}^\ell,\tcb{\varphi_i}_{i=0}^\ell)\;:\;(\gam,\mu,\sig,\tcb{\tau_i}_{i=0}^\ell,\tcb{\varphi_i}_{i=0}^\ell)\in\mathcal{U}_s},
    \end{equation}
    and, if $\gam<1$, similarly define
    \begin{equation}\label{sub for you}
        \mathcal{U}^{\m{sub}}_s(\gam,\mu,\sig)=\m{int}\tcb{(\tcb{\tau_i}_{i=0}^\ell,\tcb{\varphi_i}_{i=0}^\ell)\;:\;(\gam,\mu,\sig,\tcb{\tau_i}_{i=0}^\ell,\tcb{\varphi_i}_{i=0}^\ell)\in\mathcal{U}^{\m{sub}}_s}.
    \end{equation}
    The sets in~\eqref{omni for you} and~\eqref{sub for you} are manifestly open and nonempty since due to the fact that $\mathcal{U}_s$ and $\mathcal{U}_s^{\m{sub}}$ are are relatively open and satisfy the nondegeneracy conditions of Theorems~\ref{main theorem no. 2} and~\ref{main theorem no. 3}, respectively. With these open sets in hand, we see that the first and second items of above are just restrictions of the conclusions of Theorems~\ref{main theorem no. 2} and~\ref{main theorem no. 3} to the slices defined above.
\end{proof}

\begin{coro}[On the limits of vanishing viscosity and surface tension]\label{coro on limits}
    Let $\gam,R_{\m{visc}},R_{\m{surf}}\in\R^+$. The following hold.
    \begin{enumerate}
        \item There exists a nonincreasing sequence of open sets
        \begin{equation}
            \tcb{0}^{\ell+1}\times\tcb{0}^{\ell+1}\subset\tilde{\mathcal{U}}_s(\gam,R_{\m{visc}},R_{\m{surf}})\subset(H^\varsigma(\R^d;\R^{d\times d}))^{\ell+1}\times(H^\varsigma(\R^d;\R^d))^{\ell+1},\quad \N\ni s\ge\varsigma
        \end{equation}
        such that the following hold.
        \begin{enumerate}
            \item We have the inclusion $\tilde{\mathcal{U}}_s(\gam,R_{\m{visc}},R_{\m{surf}})\subseteq\bigcap_{0\le\mu\le R_{\m{visc}}}\bigcap_{0\le\sig\le R_{\m{surf}}}\mathcal{U}_s(\gam,\mu,\sig)$.
            \item For each fixed $(\tcb{\tau_i}_{i=0}^\ell,\tcb{\varphi_i}_{i=0}^\ell)\in\tilde{\mathcal{U}}_s(\gam,R_{\m{visc}},R_{\m{surf}})\cap\tp{H^s(\R^d;\R^{d\times d}))^{\ell+1}\times(H^s(\R^d;\R^d))^{\ell+1}}$, where $\N\ni s\ge\varsigma$, there exists a continuous mapping
            \begin{equation}\label{parameter tracing}
                [0,R_{\m{visc}}]\times[0,R_{\m{surf}}]\ni(\mu,\sig)\mapsto(v,\eta)\in H^s(\R^d;\R^d)\times\mathcal{H}^s(\R^d;\R^d)
            \end{equation}
            where $(v,\eta)$ solve~\eqref{traveling wave formulation of the equation} as in Theorem~\ref{main theorem no. 2}.
            \item The following derivatives of each mapping in~\eqref{parameter tracing} are well-defined and continuous between the stated spaces:
            \begin{equation}\label{v boy}
                [0,R_{\m{visc}}]\times[0,R_{\m{surf}}]\ni(\mu,\sigma)\mapsto(\mu\grad v,\mu^2\grad^2v)\in H^{s}(\R^d;\R^{d^2})\times H^s(\R^d;\R^{d^3})
            \end{equation}
            and
            \begin{equation}\label{eta boy}
                [0,R_{\m{visc}}]\times[0,R_{\m{surf}}]\ni(\mu,\sigma)\mapsto(\mu+\sig)(\grad\eta,\sig\grad^2\eta,\sig^2\grad^3\eta)\in H^s(\R^d;\R^d)\times H^s(\R^d;\R^{d^2})\times H^s(\R^d;\R^{d^3}).
            \end{equation}
            
        \end{enumerate}
        \item If we assume additionally that $\gam<1$, then there exists an additional nonincreasing sequence of open sets
                \begin{equation}
            \tcb{0}^{\ell+1}\times\tcb{0}^{\ell+1}\subset\tilde{\mathcal{U}}^{\m{sub}}_s(\gam,R_{\m{visc}},R_{\m{surf}})\subset(H^\varpi(\R^d;\R^{d\times d}))^{\ell+1}\times(H^\varpi(\R^d;\R^d))^{\ell+1},\quad \N\ni s\ge\varpi
        \end{equation}
        such that the following hold.
        \begin{enumerate}
            \item We have the inclusion $\tilde{\mathcal{U}}^{\m{sub}}_s(\gam,R_{\m{visc}},R_{\m{surf}})\subseteq\bigcap_{0\le\mu\le R_{\m{visc}}}\bigcap_{0\le\sig\le R_{\m{surf}}}\mathcal{U}^{\m{sub}}_s(\gam,\mu,\sig)$.
            \item For each $(\tcb{\tau_i}_{i=0}^\ell,\tcb{\varphi_i}_{i=0}^\ell)\in\tilde{\mathcal{U}}^{\m{sub}}_s(\gam,R_{\m{visc}},R_{\m{surf}})\cap\tp{H^s(\R^d;\R^{d\times d}))^{\ell+1}\times(H^s(\R^d;\R^d))^{\ell+1}}$, where $\N\ni s\ge\varpi$, there exists a continuous mapping
            \begin{equation}\label{parameter tracing sub}
                [0,R_{\m{visc}}]\times[0,R_{\m{surf}}]\ni(\mu,\sig)\mapsto(v,\eta)\in H^s(\R^d;\R^d)\times\mathcal{H}^{1+s}(\R^d;\R^d)
            \end{equation}
            where $(v,\eta)$ solve~\eqref{traveling wave formulation of the equation} as in Theorem~\ref{main theorem no. 3}.
            \item The following derivatives of each mapping in~\eqref{parameter tracing sub} are well-defined and continuous between the stated spaces:
            \begin{equation}\label{v boy sub}
                [0,R_{\m{visc}}]\times[0,R_{\m{surf}}]\ni(\mu,\sigma)\mapsto(\mu\grad v,\mu^2\grad^2v)\in H^{s}(\R^d;\R^{d^2})\times H^s(\R^d;\R^{d^3})
            \end{equation}
            and
            \begin{equation}\label{eta boy sub}
            [0,R_{\m{visc}}]\times[0,R_{\m{surf}}]\ni(\mu,\sigma)\mapsto(\sig\grad\eta,\sig^2\grad^2\eta)\in H^{1+s}(\R^d;\R^d)\times H^{1+s}(\R^d;\R^{d^2}).
            \end{equation}
            
        \end{enumerate}
    \end{enumerate}
\end{coro}
\begin{proof}
    Given $\gam$, $R_{\m{visc}}$, and $R_{\m{surf}}$ as in the statement of the corollary, we define
    \begin{equation}\label{you wish}
        \tilde{\mathcal{U}}_s(\gam,R_{\m{visc}},R_{\m{surf}})=\m{int}\bigcap_{0\le\mu\le R_{\m{visc}}}\bigcap_{0\le\sig\le R_{\m{surf}}}\mathcal{U}_s(\gam,\mu,\sig)
    \end{equation}
    and, provided $\gam<1$,
    \begin{equation}\label{you wish sub}
        \tilde{\mathcal{U}}^{\m{sub}}_s(\gam,R_{\m{visc}},R_{\m{surf}})=\m{int}\bigcap_{0\le\mu\le R_{\m{visc}}}\bigcap_{0\le\sig\le R_{\m{surf}}}\mathcal{U}^{\m{sub}}_s(\gam,\mu,\sig),
    \end{equation}
    where again we use $\m{int}$ to denote the topological interior. We note that the set $\tilde{\mathcal{U}}_s(\gam,R_{\m{visc}},R_{\m{surf}})$ is indeed non-empty since inspection of the definition of the set $\mathcal{U}_s$ in~\eqref{definition of the omnisonic set Us} allows one to conclude that
    \begin{equation}
        \m{dist}\bp{\tcb{\gam}\times[0,R_{\m{visc}}]\times[0,R_{\m{surf}}]\times\tcb{0},\pd\bp{\bigcup_{\mathfrak{q}\in\R^+\times[0,\infty)^2}B_{\R^3\times\W_\varsigma}((\mathfrak{q},0),\uprho_{s-\varsigma}(\mathfrak{q}))}}>0.
    \end{equation}
    Hence $\tilde{\mathcal{U}}_s(\gam,R_{\m{visc}},R_{\m{surf}})$ contains balls of positive radii of length at most the above distance. A similar argument verifies that $\tilde{\mathcal{U}}^{\m{sub}}_s(\gam,R_{\m{visc}},R_{\m{surf}})$ is also non-empty. Armed with the open sets of~\eqref{you wish} and~\eqref{you wish sub}, we find that the conclusions stated above are immediate consequences of Theorems~\ref{main theorem no. 2} and~\ref{main theorem no. 3}.
\end{proof}

% _+__+_ -_+__+_ -_+__+_ -_+__+_ -_+__+_ -_+__+_ -_+__+_ -_+__+_ -_+__+_ -_+__+_ -_+__+_ -_+__+_ -_+__+_ -
\appendix
% _+__+_ -_+__+_ -_+__+_ -_+__+_ -_+__+_ -_+__+_ -_+__+_ -_+__+_ -_+__+_ -_+__+_ -_+__+_ -_+__+_ -_+__+_ -
\section{Analysis background}
% _+__+_ -_+__+_ -_+__+_ -_+__+_ -_+__+_ -_+__+_ -_+__+_ -_+__+_ -_+__+_ -_+__+_ -_+__+_ -_+__+_ -_+__+_ -

% _+__+_ -_+__+_ -_+__+_ -_+__+_ -_+__+_ -_+__+_ -_+__+_ -_+__+_ -_+__+_ -_+__+_ -_+__+_ -_+__+_ -_+__+_ -
\subsection{On the anisotropic Sobolev spaces}\label{appendix on the anisobros}
% _+__+_ -_+__+_ -_+__+_ -_+__+_ -_+__+_ -_+__+_ -_+__+_ -_+__+_ -_+__+_ -_+__+_ -_+__+_ -_+__+_ -_+__+_ -

For $\R\ni s\ge 0$ and $d\in\N^+$ we define the anisotropic Sobolev space
	\begin{equation}\label{how has this not yet been labeled}
		\mathcal{H}^s(\R^d)=\tcb{f\in\mathscr{S}^\ast(\R^d;\R)\;:\;\mathscr{F}[f]\in L^1_{\m{loc}}(\R^d;\C),\;\tnorm{f}_{\mathcal{H}^s} <\infty}\index{\textbf{Function spaces}!170@$\mathcal{H}^s$},
	\end{equation}
	equipped with the norm
	\begin{equation}\label{there's one for you nineteen for me}
 \tnorm{f}_{\mathcal{H}^s}
 =\bp{\int_{\R^d}\sp{|\xi|^{-2}(\xi_1^2+|\xi|^4)\mathds{1}_{B(0,1)}(\xi) + \tbr{\xi}^{2s}\mathds{1}_{\R^d\setminus B(0,1)}(\xi)} |\mathscr{F}[f](\xi)|^2\;\m{d}\xi}^{1/2}.
	\end{equation}
 These spaces were introduced in Leoni and Tice~\cite{leoni2019traveling}, where it was shown, in Proposition 5.3 and Theorem 5.6,  that $\mathcal{H}^s(\R^d)$ is a Hilbert space and $H^s(\R^d) \emb \mathcal{H}^s(\R^d) \emb H^s(\R^d) + C^\infty_0(\R^d)$, with equality in the first embedding if and only if $d=1$.

 We shall find the following characterizations useful.
 \begin{prop}[Characterizations of the anisotropic Sobolev spaces]\label{proposition on spatial characterization of anisobros}
     The following hold for $\R\ni s\ge 0$.
     \begin{enumerate}
         \item If $f\in\mathscr{S}^\ast(\R^d;\R)$ satisfies $\grad f\in H^{s-1}(\R^d;\R^d)$ and $\pd_1 f\in\dot{H}^{-1}(\R^d)$, then there exists a constant $c\in\R$ such that $f-c\in\mathcal{H}^s(\R^d)$.
         
        \item  Assume $d \ge 2$ and write $p_d = \frac{2(d+1)}{d-1}>2$. Then we have the equality 
        \begin{equation}\label{ziggy_stardust_1}
            \mathcal{H}^s(\R^d)  
             = \bcb{ f \in L^2(\R^d) + \bigcup_{p_d < p < \infty} L^p(\R^d) \;:\; \nabla f \in H^{s-1}(\R^d;\R^d) \text{ and } \partial_1 f \in \dot{H}^{-1}(\R^d)},
        \end{equation}
        with the equivalence of norms 
         \begin{equation}\label{the norm on the anisotropic Sobolev spaces}
             \tnorm{f}_{\mathcal{H}^s} \asymp \sqrt{\tnorm{\grad f}_{H^{s-1}}^2+\tsb{\pd_1f}_{\dot{H}^{-1}}^2},
         \end{equation}
         where the implied constants depend only on $d$ and $s$.
         \item With $d$ and $p_d$ as in the previous item, we have the equality
         \begin{equation}
             \mathcal{H}^s(\R^d)  
             = \bcb{ f \in L^2(\R^d) + \bigcup_{p_d < p < \infty} L^p(\R^d) \;:\; \nabla f \in H^{s-1}(\R^d;\R^d) \text{ and } \mathcal{R}_1 f \in L^2(\R^d)}
         \end{equation}
         with the equivalence of norms
         \begin{equation}\label{the norm on the anisotropic Sobolev spaces, 2}
             \tnorm{f}_{\mathcal{H}^s} \asymp \sqrt{\tnorm{\grad f}_{H^{s-1}}^2+\tnorm{\mathcal{R}_1f}_{L^2}^2},
         \end{equation}
         where $\mathcal{R}_1$ is the Riesz transform in the $e_1$-direction.
     \end{enumerate}
 \end{prop}
 \begin{proof}
     The first and second items are the content of Proposition B.4 in Stevenson and Tice~\cite{stevenson2023wellposedness}. That the third item holds follows from the fact that $\pd_1f\in\dot{H}^{-1}(\R^d)$ is equivalent to $\mathcal{R}_1 f\in L^2(\R^d)$ for any $f\in (L^2+L^q)(\R^d)$ for $p_d<q<\infty$.
 \end{proof}

The following result is Lemma B.6 in Stevenson and Tice~\cite{stevenson2023wellposedness}.

 \begin{prop}[LP-smoothability]\label{lem on lp smoothability of anisotropic Sobolev spaces}
		The Banach scale $\tcb{\mathcal{H}^s(\R^d)}_{s\in\N}$ is LP-smoothable in the sense of Definition~\ref{defn of smoothable and LP-smoothable Banach scales}.
	\end{prop}

We next discuss a high-low decomposition for which the following notational convention is set. For $\kappa\in\R^+$, we define the linear operators $\Uppi^{\kappa}_{\m{L}}$ and $\Uppi^{\kappa}_{\m{H}}$ on the subspace of $f \in \mathscr{S}^\ast(\R^d;\C)$ such that $\mathscr{F}[f]$ is locally integrable via 
	\begin{equation}\label{notation for the Fourier projection operators}
		\Uppi^{\kappa}_{\m{L}}f=\mathscr{F}^{-1}[\mathds{1}_{B(0,\kappa)}\mathscr{F}[f]]\text{  and 
 }\Uppi^\kappa_{\m{H}}f=(I-\Uppi^\kappa_{\m{L}})f.
	\end{equation}
 \begin{prop}[Frequency splitting for anisotropic Sobolev spaces]\label{proposition on frequency splitting}
		The following hold for $0\le s\in\R$, $f\in\mathcal{H}^s(\R^d)$, and $\kappa\in\R^+$.
		\begin{enumerate}
			\item We have the equivalence $\tnorm{f}_{\mathcal{H}^s}\asymp_{s,\kappa}\sqrt{\tnorm{\Uppi _{\m{L}}^\kappa f}^2_{\mathcal{H}^0}+\tnorm{\Uppi_{\m{H}}^\kappa f}^2_{H^s}}$.
			\item We have that $\Uppi_{\m{L}}^\kappa f\in C^\infty_0(\R^d)$ with the estimates $\tnorm{\Uppi_{\m{L}}^\kappa f}_{W^{k,\infty}}\lesssim_{k,\kappa}\tnorm{\Uppi_{\m{L}}^\kappa f}_{\mathcal{H}^0}$ for every $k \in \N$.
		\end{enumerate}
	\end{prop}
	\begin{proof}
		This is Theorem 5.5 in Leoni and Tice~\cite{leoni2019traveling}.
	\end{proof}

 Our next result, which is Proposition B.2 in Stevenson and Tice~\cite{stevenson2023wellposedness}, enumerates important algebra properties of the anisotropic Sobolev spaces. In what follows, we denote
 \begin{equation}
     \mathcal{H}^0_{(\kappa)}(\R^d)=\tcb{f\in\mathcal{H}^0(\R^d),\;\m{supp}\mathscr{F}[f]\subseteq\Bar{B(0,\kappa)}}.
 \end{equation}
 
	\begin{prop}[Algebra properties of anisotropic Sobolev spaces]\label{proposition on algebra properties}
		Suppose that $f_1,f_2\in\mathcal{H}^0_{(\kappa)}(\R^d)$ for some $\kappa\in\R^+$.  		The following hold.
		\begin{enumerate}
			\item The product $f_1f_2$ belongs to $\mathcal{H}_{(2\kappa)}^0(\R^d)$ and satisfies the estimate $\tnorm{f_1f_2}_{\mathcal{H}^0}\lesssim_{\kappa}\tnorm{f_1}_{\mathcal{H}^0}\tnorm{f_2}_{\mathcal{H}^0}$.
			\item Set
			\begin{equation}\label{the definition of rd}
				r_d=\begin{cases}1+\sfloor{\f{d+1}{d-1}},&\text{if }d>1,\\
					1&\text{if }d=1.
				\end{cases}
			\end{equation}
			Assume additionally that $f_3,\dots,f_{r_d}\in\mathcal{H}_{(\kappa)}^0(\R^d)$. Then the pointwise product $\prod_{j=1}^{r_d}f_j$ belongs to $(L^2\cap\mathcal{H}^0_{(r_d \cdot \kappa)})(\R^d)$ and satisfies
			\begin{equation}
				\bnorm{\prod_{j=1}^{r_d}f_j}_{L^2} \lesssim_{\kappa,d} \prod_{j=1}^{r_d}\tnorm{f_j}_{\mathcal{H}^0}.
			\end{equation}
			\item If   $g\in\mathcal{H}^s(\R^d)$, then the pointwise product $f_1g$ belongs to $\mathcal{H}^s(\R^d)$ and satisfies
			\begin{equation}
				\tnorm{f_1g}_{\mathcal{H}^s}\lesssim_{\kappa,s}\tnorm{f_1}_{\mathcal{H}^0}\tnorm{g}_{\mathcal{H}^s}.
			\end{equation}
		\end{enumerate}
	\end{prop}

 \begin{coro}[More algebra properties]\label{coro on more algebra properties}
     The following hold for $s\ge 1+\tfloor{d/2}$, $t\le s$, and $g\in\mathcal{H}^s(\R^d)$.
     \begin{enumerate}
         \item If $f\in H^t(\R^d)$, then the pointwise product $fg$ belongs to $H^t(\R^d)$ and obeys the estimate
         \begin{equation}
             \tnorm{fg}_{H^t}\lesssim\tnorm{f}_{H^t}\tnorm{g}_{\mathcal{H}^{1+\tfloor{d/2}}}+\begin{cases}
                 0&\text{if }t\le1+\tfloor{d/2},\\
                 
                 \tnorm{f}_{H^{1+\tfloor{d/2}}}\tnorm{g}_{\mathcal{H}^t}&\text{if }t>1+\tfloor{d/2}.
             \end{cases}.
         \end{equation}
         \item If $h\in\mathcal{H}^t(\R^d)$, then the pointwise product $gh$ belongs to $\mathcal{H}^t(\R^d)$ and obeys the estimate
         \begin{equation}
             \tnorm{hg}_{\mathcal{H}^t}\lesssim\tnorm{h}_{\mathcal{H}^t}\tnorm{g}_{\mathcal{H}^{1+\tfloor{d/2}}}+\begin{cases}
                 0&\text{if }t\le1+\tfloor{d/2},\\
                 \tnorm{h}_{\mathcal{H}^{1+\tfloor{d/2}}}\tnorm{g}_{\mathcal{H}^{t}}&\text{if }t>1+\tfloor{d/2}
             \end{cases}
         \end{equation}
         \item If $g\in\mathcal{H}^s(\R^d)$, then $g^{r_d}$, for $r_d$ as in~\eqref{the definition of rd}, belongs to $H^s(\R^d)$ and obeys the estimate
         \begin{equation}
             \tnorm{g^{r_d}}_{H^s}\lesssim\tnorm{g}_{\mathcal{H}^s}\tnorm{g}_{\mathcal{H}^{1+\tfloor{d/2}}}^{r_d-1}.
         \end{equation}
     \end{enumerate}
 \end{coro}
 \begin{proof}
 For the first item we decompose $g=\Uppi^1_{\m{L}}g+\Uppi^1_{\m{H}}g$. For the low piece $f\Uppi^1_{\m{L}}g$, we use the second item of Proposition~\ref{proposition on frequency splitting} along with the fourth item of Proposition~\ref{corollary on tame estimates on simple multipliers}. For the high piece $f\Uppi^1_{\m{H}}g$, we instead use the first and second items of the two aforementioned propositions, respectively.

 For the second item, we decompose both $g$ and $h$ into high and low pieces via $\Uppi^1_{\m{L}}$ and $\Uppi^1_{\m{H}}$ and inspect the four terms in the expanded product individually. For either of the mixed high-low terms, we invoke the third item of Proposition~\ref{proposition on algebra properties}. For the low-low term, we instead use the first item. Finally, for the high-high term, we use the first item of Proposition~\ref{proposition on frequency splitting} along with the second item of Proposition~\ref{corollary on tame estimates on simple multipliers}.

 For the third item, we use the binomial theorem to write
 \begin{equation}
     g^{r_d}=\sum_{\ell=0}^{r_d}\bn{r_d}{\ell}(\Uppi^1_\m{L}g)^\ell(\Uppi^1_{\m{H}}g)^{r_d-\ell}.
 \end{equation}
 The terms indexed with $\ell<r_d$ are handled via the first and second items of this result. The $\ell=r_d$ term is handled via the second item of Proposition~\ref{proposition on algebra properties}.
 \end{proof}

% _+__+_ -_+__+_ -_+__+_ -_+__+_ -_+__+_ -_+__+_ -_+__+_ -_+__+_ -_+__+_ -_+__+_ -_+__+_ -_+__+_ -_+__+_ -
\subsection{Tame structure abstraction}\label{tame structure abs}
% _+__+_ -_+__+_ -_+__+_ -_+__+_ -_+__+_ -_+__+_ -_+__+_ -_+__+_ -_+__+_ -_+__+_ -_+__+_ -_+__+_ -_+__+_ -

The goal of this sub-appendix is to review notions of Banach scales and tameness of mappings between them. This is a rapid treatment of the thorough development of Section 2.1 in Stevenson and Tice~\cite{stevenson2023wellposedness}.

\begin{defn}\label{defn on subsets of N}
		Given $N\in\N\cup\tcb{\infty}$ we define $\j{N}\subseteq\N$ to be the set $\j{N}=\tcb{n\in\N\;:\;n\le N}$\index{\textbf{Miscellaneous}!10@$\j{N}$}.
	\end{defn} 
 	
  \begin{defn}[Banach scales]\label{definition of Banach scales}
		Let $N\in\N^+ \cup\tcb{\infty}$ and let $\bf{E} = \{E^s\}_{s\in \j{N}}$ be a collection of Banach spaces over a common field, either $\R$ or $\C$.
		\begin{enumerate}
			\item We say that $\bf{E}$ is a Banach scale if for each $s \in \j{N-1}$ we have the non-expansive inclusion $E^{s+1} \hookrightarrow E^s$, i.e. $\tnorm{\cdot}_{E^s}\le\tnorm{\cdot}_{E^{s+1}}$.
			\item If $\bf{E}$ is a Banach scale, then we define the scale's terminal space to be $E^N = \bigcap_{s \in \j{N}} E^s$ and endow it with the Fr\'{e}chet topology induced by the collection of norms $\{\norm{\cdot}_{E^{s}}\}_{s\in \j{N}}$. Note that if $N<\infty$, then $E^N$ has the standard Banach topology from its norm.
			\item If $\bf{E}$ is a Banach scale, then we write $B_{E^r}(u,\delta) \subseteq E^r$ for the $E^r$-open ball of radius $\delta>0$, centered at $u \in E^r$.
			\item If $\bf{E}$ is a Banach scale and $E^N$ is dense in $E^s$ for each $s \in \j{N}$, then we say that $\bf{E}$ is terminally dense.
		\end{enumerate}
	\end{defn}

 	\begin{rmk}[Products of Banach scales]\label{scale_prod_rmk}
	Suppose that $\bf{E}_i = \{E^s_i\}_{s\in \j{N}}$ is a Banach scale for $i \in \{1,\dotsc,n\}$, each over the same field. Then $\bf{F} = \prod_{i=1}^n \bf{E}_i = \{ \prod_{i=1}^n E^s_i \}_{s\in \j{N}}$ is a Banach scale when each $\prod_{i=1}^n E^s_i$ is endowed with the norm
		\begin{equation}
			\norm{u_1,\dotsc,u_n}_{\prod_{i=1}^n E^s_i} = 
				\bp{\sum_{i=1}^n \norm{u_i}_{E^s_i}^2}^{1/2}.
		\end{equation}
	\end{rmk} 
	
	\begin{defn}[Tame maps]\label{defn of smooth tame maps}
		Let $\bf{E}=\tcb{E^s}_{s\in\j{N}}$ and $\bf{F}=\tcb{F^s}_{s\in\j{N}}$ be Banach scales over the same field, $\mu,r\in\j{N}$ with $\mu\le r$, $U\subseteq E^r$ be an open set, and $P:U\to F^0$.
		\begin{enumerate}
			\item We say that $P$ satisfies tame estimates of order $\mu$ and base $r$ (with respect to the Banach scales $\bf{E}$ and $\bf{F}$) if for all $\j{N}\ni s\ge r$, there exists a constant $C_s\in\R^+$ such that for all $f\in U\cap E^s$ we have the inclusion $P(f)\in F^{s-\mu}$ as well as the estimate
			\begin{equation}
				\tnorm{P(f)}_{F^{s-\mu}}\le C_s\tbr{\tnorm{f}_{E^{s}}}.
			\end{equation}
			
			\item We say that $P$ is $\mu$-tame with base $r$ if for each $f\in U$, there exists an $E^r$-open set $V\subseteq U$ such that $f \in V$ and the restricted map $P\res V:V \to F^0$ satisfies tame estimates of order $\mu$ and base $r$.
			
			\item We say that $P$ is strongly $\mu$-tame with base $r$ if on every $E^r$-open and bounded subset $V\subseteq U$, the restricted map $P\res V$ satisfies tame estimates of order $\mu$ and base $r$.
			
			\item We say that $P$ is (strongly) $\mu$-tamely $C^0$ with base $r$ if $P$ is (strongly) $\mu$-tame with base $r$ and if for every $\j{N}\ni s\ge r$ we have that $P:U\cap E^s\to F^{s-\mu}$
			is continuous as a map from $E^s$ to $F^{s-\mu}$. The collections of such maps will be denoted by $T^0_{\mu,r}(U,\bf{E};\bf{F})$ and ${_{\m{s}}}T^0_{\mu,r}(U,\bf{E};\bf{F})$, respectively.
			
			\item For $k\in\N^+$ we say that $P$ is (strongly) $\mu$-tamely $C^k$ with base $r$ if for all $\j{N}\ni s\ge r$ the map $P:U\cap E^{s}\to F^{s-\mu}$ is  $C^k$ and for all $j\in\tcb{0,1,\dots,k}$ the $j^{\m{th}}$ derivative map, thought of as mapping $D^jP:U\times\prod_{\ell=1}^jE^r \to F^0$, is (strongly) $\mu$-tame with base $r$ with respect to the Banach scales $\bf{E}^{1+j}$ and $\bf{F}$. In other words, $D^jP\in T^0_{\mu,r}(U\times\prod_{\ell=1}^jE^r,\bf{E}^{1+j};\bf{F})$ or, in the strong case, $D^jP\in {_{\m{s}}}T^0_{\mu,r}(U\times\prod_{\ell=1}^jE^r,\bf{E}^{1+j};\bf{F})$. The collections of such maps will be denoted by $T^k_{\mu,r}(U,\bf{E};\bf{F})$\index{\textbf{Function spaces}!65@$T^k_{\mu,r}(U,\bf{E};\bf{F})$} and ${_{\m{s}}}T^k_{\mu,r}(U,\bf{E};\bf{F})$\index{\textbf{Function spaces}!66@${_{\m{s}}}T^k_{\mu,r}(U,\bf{E};\bf{F})$}, respectively.
			
			\item The collections of (strongly) $\mu$-tamely smooth with base $r$ maps are $T^\infty_{\mu,r}(U,\bf{E};\bf{F})=\bigcap_{k\in\N} T^k_{\mu,r}(U,\bf{E};\bf{F})$ and ${_{\m{s}}}T^\infty_{\mu,r}(U,\bf{E};\bf{F})=\bigcap_{k\in\N} {_{\m{s}}}T^k_{\mu,r}(U,\bf{E};\bf{F})$.
			
			\item    When $U= E^r$ we will often use the abbreviated notation $T^k_{\mu,r}(\bf{E},\bf{F})$ and ${_{\m{s}}}T^k_{\mu,r}(\bf{E};\bf{F})$ in place of $T^k_{\mu,r}(E^r,\bf{E};\bf{F})$ and ${_{\m{s}}}T^k_{\mu,r}(U,\bf{E};\bf{F})$, respectively.
		\end{enumerate}
	\end{defn}

	The following result, which is proved in Lemma 2.5 in Stevenson and Tice~\cite{stevenson2023wellposedness}, is a useful characterization of tameness when there is multilinear structure present in the map.
	
	\begin{lem}[Multilinearity and tame maps]\label{lem on multilinear tame maps} Let $k\in\N^+$ and  $r,\mu\in\j{N}$ with $\mu\le r$.  Let $\bf{F}=\tcb{F^s}_{s\in\j{N}}$, $\bf{G}=\tcb{G^s}_{s\in\j{N}}$, and $\bf{E}_j=\tcb{E^s_j}_{s\in\j{N}}$, for $j\in\tcb{1,\dots,k}$, be Banach scales over the same field.  Let $U\subseteq G^r$ be an open set.  Then the following are equivalent for all maps $P:U\times\prod_{j=1}^kE^r_j\to F^0$ such that $P(g,\cdot)$ is $k$-multilinear for all $g\in U$.
		\begin{enumerate}
			\item $P\in T^0_{\mu,r}(U\times\prod_{j=1}^kE^r_j;\bf{G}\times\prod_{j=1}^k\bf{E}_j;\bf{F})$.
			\item The restriction of $P$ to $(U\cap G^s)\times\prod_{j=1}^kE^s_j$ is continuous as a map into $F^{s-\mu}$ for all $\j{N}\ni s\ge r$, and for all $g\in U$ there exists a $G^r$-open subset $V\subseteq U$ such that $g \in V$ and whenever $\j{N}\ni s\ge r$, $f\in V\cap G^s$, and $h_i\in E_i^s$ for $i\in\tcb{1,\dots,k}$ we have the estimate
			\begin{equation}\label{which he ate and donated to the national trust}
				\tnorm{P(f,h_1,\dots,h_k)}_{F^{s-\mu}}\lesssim\tbr{\tnorm{f}_{G^s}}\prod_{i=1}^k\tnorm{h_i}_{E_i^r}+\sum_{j=1}^k\tnorm{h_j}_{E_j^s}\prod_{i\neq j}\tnorm{h_i}_{E_i^r}.
			\end{equation}
		\end{enumerate}
		A similar equivalence holds for maps $P\in{_{\m{s}}}T^0_{\mu,r}(U\times\prod_{j=1}^k E^r_j;\bf{G}\times\prod_{j=1}^k\bf{E}_j;\bf{F})$ if we change the space in the first item to the space of strongly tame maps and we change the second item's  quantification of $V$ to `for all bounded $G^r$-open sets $V \subseteq U$.'
	\end{lem}

 \begin{rmk}\label{remark on purely multlinear maps}
		If $P\in T^0_{\mu,r}(\prod_{j=1}^k\bf{E}_j;\bf{F})$ is a $k$-multilinear mapping, then a simple modification of Lemma~\ref{lem on multilinear tame maps} shows that $P$ is actually strongly tame.  Moreover, by multilinearity, we have that $P$ is automatically a smooth function whose derivatives are also multilinear maps. This combines with the previous fact to show that actually $P\in {_{\m{s}}}T^\infty_{\mu,r}(\prod_{j=1}^k\bf{E}_j;\bf{F})$.
	\end{rmk}

 \begin{coro}[Derivative estimates on tame maps]\label{lem on derivative estimates on tame maps}
		Let $\bf{E}=\tcb{E^s}_{s\in\j{N}}$ and $\bf{F}=\tcb{F^s}_{s\in\j{N}}$ be Banach scales over a common field. Let $r,\mu\in\j{N}$ with $\mu\le r$, $U\subseteq E^r$ be an open set, and $P\in T^k_{\mu,r}(U,\bf{E};\bf{F})$.  Then for each $f_0\in U$, there exists an open set $f_0\in V\subseteq U$ with the property that for every $\j{N}\ni s\ge r$ there exists a constant $C_{s,V} \in \R^+$, depending on $s$ and $V$, such that for all $f\in V\cap E^s$ and all $g_1,\dots,g_k\in E^s$ we have the estimate
		\begin{equation}
			\tnorm{D^kP(f)[g_1,\dots,g_k]}_{F^{s-\mu}}\le C_{s,V} \sum_{\ell=1}^k\tnorm{g_\ell}_{E^s}\prod_{j\neq\ell}\tnorm{g_j}_{E^r}
			+ C_{s,V} \tbr{\tnorm{f}_{E^s}}\prod_{j=1}^k\tnorm{g_j}_{E^r}.
		\end{equation}
		Moreover, if $P\in{_{\m{s}}}T^k_{\mu,r}(U,\bf{E};\bf{F})$, then a similar assertion holds for every bounded and open subset $V \subseteq U$.
	\end{coro}
	\begin{proof}
		This is a direct application of Lemma~\ref{lem on multilinear tame maps}.
	\end{proof}

 \begin{prop}[Composition of tame maps]\label{lem on composition of tame maps}
		Let  $\bf{E}=\tcb{E^s}_{s\in\j{N}}$, $\bf{F}=\tcb{F^s}_{s\in\j{N}}$, and $\bf{G}=\tcb{G^s}_{s\in\j{N}}$ be a triple of Banach scales over a common field ($\R$ or $\C$), and let $k\in\N$ and $\mu,\mu', r,r'\in\j{N}$ be such that $\mu\le r$, $\mu'\le r'$, $r'+\mu\in\j{N}$.  Suppose that $U\subseteq E^r$, $U'\subseteq F^{r'}$ are open sets,  $P\in T^k_{\mu,r}(U,\bf{E};\bf{F})$, and $Q\in T^k_{\mu',r'}(U',\bf{F};\bf{G})$. Set $V=U\cap E^{\max\tcb{r,r'+\mu}}$.   If $P(V)\subseteq U'$, then $Q\circ P\in T^k_{\mu+\mu',\max\tcb{r,\mu+r'}}(V,\bf{E};\bf{G})$. A similar assertion holds for the composition of strongly tame maps.
	\end{prop}

	\begin{coro}[Tame products of tame $C^k$ maps]\label{lem on product of smooth tame maps} Let  $\ell,k\in\N$ with $\ell\ge1$.  For $j\in\tcb{1,\dots,\ell}$ let  $\bf{E}_j=\tcb{E^s_j}_{s\in\j{N}}$ and $\bf{F}_j=\tcb{E^s_j}_{s\in\j{N}}$ be Banach scales over a common field, and let $\bf{G}=\tcb{G^s}_{s\in\j{N}}$ be a Banach scale over the same field.  Let $r_1,\dots,r_\ell\in\j{N}$, $\mu_j\in\j{r_j}$ for $j\in\tcb{1,\dots,\ell}$, $r'\in\j{N}$, and $\mu'\in\j{r'}$.  Finally, suppose that $U_j\subseteq E^{r_j}_j$ is open, $P_j\in T^k_{\mu_j,r_j}(U_j,\bf{E}_j;\bf{F}_j)$, and $B\in T^0_{\mu',r'}(\prod_{j=1}^\ell\bf{F}_j;\bf{G})$ is $k$-multilinear. If we set
		\begin{equation}
			\mu=\mu'+\max\tcb{\mu_1,\dots,\mu_\ell}
			\text{ and }
			r=\max\tcb{r_1,\dots,r_\ell,\mu_1+r',\dots,\mu_\ell+r'}
		\end{equation}
		and assume that $r\in\j{N}$, then the product map satisfies the inclusion
		\begin{equation}
			B(P_1,\dots,P_\ell)\in T^k_{\mu,r}\bp{\prod_{j=1}^\ell U_j\cap E^{r}_j,\prod_{j=1}^\ell\bf{E}_j;\bf{G}}.
		\end{equation}
		A similar assertion holds for products of strongly tame maps.
	\end{coro}

 \begin{defn}[Smoothable and LP-smoothable Banach scales]\label{defn of smoothable and LP-smoothable Banach scales}
		Let $\bf{E} = \{E^s\}_{s\in \j{N}}$ be a Banach scale.
		\begin{enumerate}
			\item We say that $\bf{E}$ is smoothable if for each $j \in \N^+$ there exists a linear map $S_j : E^0 \to E^N$\index{\textbf{Linear maps}!100@$S_j$} satisfying the following smoothing conditions for every $u \in E^0$:
			\begin{equation}\label{smoothing_1}
				\norm{S_j u}_{E^s} \lesssim \norm{u}_{E^s} \text{ for all }s \in \j{N},
			\end{equation}
			\begin{equation}\label{smoothing_2}
				\norm{S_j u}_{E^t} \lesssim 2^{j(t-s)} \norm{S_j u}_{E^s}  \text{ for all } s,t \in \j{N} \text{ with } s < t,  
			\end{equation}
			\begin{equation}\label{smoothing_3}
				\norm{(I-S_j)u}_{E^s} \lesssim 2^{-j(t-s)} \norm{(I-S_j) u}_{E^t} \text{ for all } s,t \in \j{N} \text{ with } s < t,
			\end{equation}
			\begin{equation}\label{smoothing_4}
				\norm{(S_{j+1}-S_j)u}_{E^t} \lesssim 2^{j(t-s)}    \norm{(S_{j+1}-S_j)u}_{E^s} \text{ for all }s,t \in \j{N},
			\end{equation}
			where the implicit constants are independent of $j$ and are increasing in $s$ and $t$.

			\item We say that $\bf{E}$ is LP-smoothable if $\bf{E}$ is smoothable and the smoothing operators satisfy the following Littlewood-Paley condition:   for $s \in \j{N}$ and $u \in E^s$ we have that
			\begin{equation}\label{smoothing_LP_1}
				\lim_{j\to \infty} \tnorm{(I-S_j)u}_{E^s} =0, 
			\end{equation}
			and there exists a constant $A >0$, possibly depending on $s$, such that 
			\begin{equation}\label{smoothing_LP_2}
				A^{-1} \norm{u}_{E^s} \le \bp{ \sum_{j=0}^\infty \norm{\Updelta_j u}_{E^s}^2 }^{1/2} \le A \norm{u}_{E^s},
			\end{equation}
			where the operators $\{\Updelta_j\}_{j=0}^\infty$ are defined via $\Updelta_j=S_{j+1}-S_j$\index{\textbf{Linear maps}!101@$\Updelta_j$} with the convention that $S_0=0$.
		\end{enumerate}
	\end{defn}  
	
	\begin{exa}[A fixed Banach space]\label{example of a fixed Banach space}
		Suppose that $X$ is a Banach space and $N \in \N \cup \{\infty\}$. Then the scale $\tcb{X}_{s \in \j{N}}$, generated by $X$ alone, is LP-smoothable in the sense of Definition~\ref{defn of smoothable and LP-smoothable Banach scales}, provided we take $S_j=I$ for all $j\in\N^+$.
	\end{exa}
	
	\begin{exa}[Sobolev spaces on $\R^d$]\label{example of Euclidean Soblev spaces}
		Let $V$ be a finite dimensional real vector space and $N \in \N \cup \{\infty\}$.  The real Banach scales of Sobolev spaces $\tcb{H^s(\R^d;V)}_{s\in \j{N}}$ and $\tcb{\tp{\dot{H}^{-1}\cap H^s}(\R^d;V)}_{s\in\j{N}}$ are LP-smoothable for the smoothing operators $\tcb{S_j}_{j=0}^\infty$ given by $S_j=\mathds{1}_{B(0,2^j)}(\grad/2\pi\ii)$.
	\end{exa}

% _+__+_ -_+__+_ -_+__+_ -_+__+_ -_+__+_ -_+__+_ -_+__+_ -_+__+_ -_+__+_ -_+__+_ -_+__+_ -_+__+_ -_+__+_ -
\subsection{A variation on the Nash-Moser inverse function theorem}\label{subsection on a NMIFT}
% _+__+_ -_+__+_ -_+__+_ -_+__+_ -_+__+_ -_+__+_ -_+__+_ -_+__+_ -_+__+_ -_+__+_ -_+__+_ -_+__+_ -_+__+_ -

The following definition gives us a sense of the mapping hypotheses one needs to verify to invoke a Nash-Moser inverse function theorem.

	\begin{defn}[Mapping hypotheses]\label{defn of the mapping hypotheses}
		We say that a triple $(\bf{E},\bf{F},\Psi)$ satisfies the RI mapping hypotheses with parameters $(\mu,r,R) \in \j{N}^3$ if $\bf{E} = \{E^s\}_{s\in \j{N}}$ and $\bf{F} = \{F^s\}_{s\in \j{N}}$ are Banach scales over a common field, $1 \le \mu \le r < R < \infty$, $R+\mu \in \j{N}$, and there exists $0 < \delta_r \in \R$ such that  $\Psi : B_{E^r}(0,\delta_r) \to F^{r-\mu}$ is a map satisfying the following.
		\begin{enumerate}
			\item $\Psi(0)=0$.
			
			\item $\mu$-tamely $C^2$: For every $r-\mu \le s \in \j{N-\mu}$ we have that $\Psi : B_{E^r}(0,\delta_r) \cap E^{s+\mu} \to F^s$ is $C^2$, and for every $u_0 \in B_{E^r}(0,\delta_r) \cap E^{s+\mu}$ we have the tame estimate 
			\begin{equation}\label{C2 tame estimates}
				\norm{D^2 \Psi(u_0)[v,w] }_{F^s} \le C_1(s)\tp{\norm{v}_{E^{s+\mu}} \norm{w}_{E^r} 
					+ \norm{v}_{E^r} \norm{w}_{E^{s+\mu}}  
					+  \br{\norm{u_0}_{E^{s+\mu}} } \norm{v}_{E^r} \norm{w}_{E^r} }.
			\end{equation}
			Here the constant $C_1(s)$ is increasing in $s$. In other words $\Psi \in{_{\m{s}}}T^2_{\mu,r}(B_{E^r}(0,\del_r),\bf{E};\bf{F})$ according to the notation from Definition~\ref{defn of smooth tame maps}.
			
			\item Derivative inversion: There exists $\delta_R \in \R$ satisfying $0 < \delta_R \le \delta_r$ such that for every $u_0 \in B_{E^r}(0,\delta_R) \cap E^{N}$ there exists a bounded linear operator $L(u_0) : F^r \to E^r$ satisfying the following three conditions.
			\begin{enumerate}
				\item For every $s \in \N \cap [r,R]$  we have that the restriction of $L(u_0)$ to $F^s$ defines a bounded linear operator with values in $E^s$, i.e. $L(u_0) \in \mathcal{L}(F^s; E^s)$.
				\item $D\Psi(u_0) L(u_0) f = f$ for every $f \in F^r$.

				\item We have the tame estimate 
				\begin{equation}\label{tame estimates on the right inverse}
					\norm{L(u_0) f}_{E^s} \le C_2(s) \tp{\norm{f}_{F^s} + \br{\norm{u_0}_{E^{s+\mu}}} \norm{f}_{F^r}}
				\end{equation}
				for every $f \in F^s$ and $r \le s \le R$, where again the constant $C_2(s)$ is increasing in $s$.
			\end{enumerate}
		\end{enumerate}
		Here the use of the prefix RI- is meant to indicate that the maps $L(u_0)$ are only required to be right inverses.  We say that the triple $(\bf{E},\bf{F},\Psi)$ satisfies the LRI mapping hypotheses if condition $(b)$ in the third item is augmented by the left-inverse condition  $(b'):$  
		\begin{equation}
			L(u_0) D\Psi(u_0) v = v \text{ for every }v \in E^{r+\mu}.
		\end{equation}
	\end{defn}

We now state a slightly specified version of the Nash-Moser inverse function theorem. For a proof and a more general formulation, the reader is referred to Section 2 of Stevenson and Tice~\cite{stevenson2023wellposedness}.
	
	\begin{thm}[Inverse function theorem]\label{thm on nmh}
		Let $\bf{E} = \{E^s\}_{s\in \j{N}}$ and $\bf{F} = \{F^s\}_{s\in \j{N}}$ be Banach scales over the same field consisting of reflexive spaces, and assume that $\bf{E}$ and $\bf{F}$ are LP-smoothable (see Definition~\ref{defn of smoothable and LP-smoothable Banach scales}). Assume the triple $(\bf{E},\bf{F},\Psi)$ satisfies the LRI mapping hypotheses of Definition \ref{defn of the mapping hypotheses} with parameters $(\mu,r,R)$ satisfying $2(r+\mu) + 1 < (r+R)/2$, and set $\be=2(r+\mu)+1\in\j{N}$.   Then there exist $\ep,\kappa_1,\kappa_2 >0$ such that the following hold.
		\begin{enumerate}
			\item Existence of local inverse: For every $g \in B_{F^\beta}(0,\ep)$ there exists a unique $u \in B_{E^\beta}(0,\kappa_1 \ep)$ such that $\Psi(u) =g$. 

			\item Estimates of local inverse:  The induced bijection $\Psi^{-1} :  B_{F^\beta}(0,\ep) \to \Psi^{-1}(B_{F^\beta}(0,\ep)) \cap  B_{E^\beta}(0,\kappa_1 \ep)$ obeys the estimate
			\begin{equation}\label{nmh_01}
				\tnorm{\Psi^{-1}(g)}_{E^\beta} \le \kappa_1 \tnorm{g}_{F^\beta} \text{ for all }g \in B_{F^\beta}(0,\ep).
			\end{equation}
			Moreover, if $\N\ni\nu\le R+r-2\be$, then we have that
			\begin{equation}\label{derivatives_gain}
				\Psi^{-1}:B_{F^\be}(0,\ep)\cap F^{\be+\nu}\to E^{\be+\nu}
			\end{equation}
			with the estimate
			\begin{equation}\label{when the rain comes they run and hide}
				\tnorm{\Psi^{-1}(g)}_{E^{\be+\nu}}\lesssim\tnorm{g}_{F^{\be+\nu}},
			\end{equation}
			where the implied constant is independent of $g$.
			
			\item Continuity: For $s\in[\be,R+r-\be-\mu)\cap\N$ the map
			\begin{equation}
				\Psi^{-1}:B_{F^\be}(0,\ep)\cap F^s\to E^s
			\end{equation}
			is continuous.
		\end{enumerate}
	\end{thm}

% _+__+_ -_+__+_ -_+__+_ -_+__+_ -_+__+_ -_+__+_ -_+__+_ -_+__+_ -_+__+_ -_+__+_ -_+__+_ -_+__+_ -_+__+_ -
\subsection{Toolbox for smooth-tameness verification}\label{toolbox for smooth-tameness}
% _+__+_ -_+__+_ -_+__+_ -_+__+_ -_+__+_ -_+__+_ -_+__+_ -_+__+_ -_+__+_ -_+__+_ -_+__+_ -_+__+_ -_+__+_ -

We begin this appendix by recording some facts about products. In the subsequent three results, the following notation for specific Banach scales are used:
\begin{equation}\label{atomic Banach scale notation}
    \bf{H}(W)=\tcb{H^s(\R^d;W)}_{s\in\N},\quad\bf{W}(W)=\tcb{W^{s,\infty}(\R^d;W)}_{s\in\N},\quad\bm{\mathcal{H}}=\tcb{\mathcal{H}^s(\R^d)}_{s\in\N},
\end{equation}
for $W$ a finite dimensional real vector space. The following is  Lemma 3.3 from Stevenson and Tice~\cite{stevenson2023wellposedness}.

\begin{lem}[Smooth tameness of products, 1]\label{lemma on tameness of products 1}
		Let $V_1,V_2$ and $W$ be real finite dimensional vector spaces, $B:V_1\times V_2 \to W$ be a bilinear map. The following inclusions hold when viewing $B$ as a product for functions taking values in $V_1$ and $V_2$:  $B\in {_{\m{s}}}T^\infty_{0,1+\tfloor{d/2}}(\bf{H}(V_1)\times\bf{H}(V_2);\bf{H}(W))$, and $B\in {_{\m{s}}}T^\infty_{0,0}(\bf{H}(V_1)\times\bf{W}(V_2);\bf{H}(W))$.
	\end{lem}

 Our next product result is proved in Lemma 3.4 in Stevenson and Tice~\cite{stevenson2023wellposedness}.

 \begin{lem}[Smooth tameness of products, 2]\label{lem on smooth tameness of products 2}
		Let $m\in\N^+$, $V$, $W$, and $V_1,\dots,V_m$ be real finite dimensional vector spaces, $B:V_1\times\cdots\times V_m\times V\to W$ be $(m+1)$-linear. The $(m+1)$-linear map $P_B$ defined via $P_B\tp{(g_j,\psi_j)_{j=1}^m,\varphi}=B(g_1+\psi_1,\dots,g_m+\psi_m,\varphi)$ satisfies the inclusion
		\begin{equation}
			P_B\in{_{\m{s}}}T^\infty_{0,1+\tfloor{d/2}}\bp{\prod_{j=1}^m\tp{\bf{W}(V_j)\times\bf{H}(V_j)}\times\bf{H}(V);\bf{H}(W)}.
		\end{equation}
	\end{lem}

 \begin{lem}[Smooth tameness of products, 3]\label{lem on smooth tameness of products 3}
     Let $W$ be a real finite dimensional vector space. Then the product map $(\eta,u)\mapsto \eta u$ belongs to ${_{\m{s}}}T^\infty_{0,1+\tfloor{d/2}}(\bm{\mathcal{H}}\times \bf{H}(W);\bf{H}(W))$.  Also, the product map $(\eta,\xi)\mapsto\eta\xi$ belongs to ${_{\m{s}}}T^\infty_{0,1+\tfloor{d/2}}\tp{\bm{\mathcal{H}}\times\bm{\mathcal{H}};\bm{\mathcal{H}}}$.
 \end{lem}
 \begin{proof}
     The claimed mapping properties follow from Remark~\ref{remark on purely multlinear maps} and Corollary~\ref{coro on more algebra properties}.
 \end{proof}

 The remaining results of this sub-appendix are slightly more quantitative tame estimates.

 \begin{prop}[Tame estimates on simple multipliers]\label{corollary on tame estimates on simple multipliers}
		Let $k\in\N$ and $\varphi\in H^k(\R^d)$. The following hold.
		\begin{enumerate}
			\item Suppose that $\al,\be\in\N^d$ are such that $|\al|+|\be|=k$ and $\psi\in H^{\max\tcb{1+\tfloor{d/2},k}}(\R^d)$. Then the product $\pd^\al\psi\pd^\be\varphi$ belongs to $L^2(\R^d)$ and obeys the bound
			\begin{equation}\label{a bound for me and you}
				\tnorm{\pd^\al\psi\pd^\be\varphi}_{L^2}\lesssim\tnorm{\psi}_{H^{1+\tfloor{d/2}}}\tnorm{\psi}_{H^k}+\begin{cases}
					0&\text{if }k\le\tfloor{d/2},\\
					\tnorm{\psi}_{H^k}\tnorm{\varphi}_{H^{1+\tfloor{d/2}}}&\text{if }1+\tfloor{d/2}<k.
				\end{cases}
			\end{equation}
			\item Suppose that $\psi\in H^{\max\tcb{1+\tfloor{d/2},k}}(\R^d)$. Then the product $\psi\varphi$ belongs to $H^k(\R^d)$ and obeys the bound
			\begin{equation}
				\tnorm{\psi\varphi}_{H^k}\lesssim\tnorm{\psi}_{H^{1+\tfloor{d/2}}}\tnorm{\psi}_{H^k}+\begin{cases}
					0&\text{if }k\le1+\tfloor{d/2},\\
					\tnorm{\psi}_{H^k}\tnorm{\varphi}_{H^{1+\tfloor{d/2}}}&\text{if }1+\tfloor{d/2}<k.
				\end{cases}
			\end{equation}
			\item Suppose that $\al,\be\in\N^d$ are such that $|\al|+|\be|=k$ and $\psi\in W^{k,\infty}(\R^d)$. Then the product $\pd^\al\psi\pd^\be\varphi$ belongs to $L^2(\R^d)$ and obeys the bound
			\begin{equation}
				\tnorm{\pd^\al\psi\pd^\be\varphi}_{L^2}\lesssim\tnorm{\psi}_{L^\infty}\tnorm{\varphi}_{H^k}+\tnorm{\psi}_{W^{k,\infty}}\tnorm{\varphi}_{L^2}.
			\end{equation}
			\item Suppose that $\psi\in W^{k,\infty}(\R^d)$. Then the product $\psi\varphi$ belongs to $H^k(\R^d)$ and obeys the bound
			\begin{equation}
				\tnorm{\psi\varphi}_{H^k}\lesssim\tnorm{\psi}_{L^\infty}\tnorm{\varphi}_{H^k}+\tnorm{\psi}_{W^{k,\infty}}\tnorm{\varphi}_{L^2}.
			\end{equation}
		\end{enumerate}
	\end{prop}
	\begin{proof}
This is Corollary D.7 in Stevenson and Tice~\cite{stevenson2023wellposedness}.
 \end{proof}

\begin{prop}[Tame estimates on commutators]\label{prop tame commutator estimates}
		Let $k,m\in\N$ with $m\ge 1$, $j\in\tcb{1,\dots,d}$, and $\varphi\in H^{k+m-1}(\R^d)$. The following hold.
		\begin{enumerate}
			\item Suppose that $\psi\in H^{\max\tcb{2+\tfloor{d/2},k+m}}_{\loc}(\R^d)$ satisfies $\pd_j\psi\in H^{\max\tcb{1+\tfloor{d/2},k+m-1}}(\R^d)$. Then the commutator $[\pd_j^m,\psi]\varphi$ belongs to $H^{k}(\R^d)$ and obeys the bound
			\begin{equation}
				\tnorm{[\pd_j^m,\psi]\varphi}_{H^k}\lesssim\tnorm{\pd_j\psi}_{H^{1+\tfloor{d/2}}}\tnorm{\varphi}_{H^{k+m-1}}+\begin{cases}
					0&\text{if }k+m\le2+\tfloor{d/2},\\
					\tnorm{\pd_j\psi}_{H^{k+m-1}}\tnorm{\varphi}_{H^{1+\tfloor{d/2}}}&\text{if }2+\tfloor{d/2}<k+m.
				\end{cases}
			\end{equation}
			\item Suppose that $\psi\in W^{k+m,\infty}_{\loc}(\R^d)$ satisfies $\pd_j\psi\in W^{k+m-1,\infty}(\R^d)$. Then the commutator $[\pd_j^m,\psi]\varphi$ belongs to $H^k(\R^d)$ and obeys the bound
			\begin{equation}
				\tnorm{[\pd_j^m,\psi]\varphi}_{H^k}\lesssim\tnorm{\pd_j\psi}_{L^\infty}\tnorm{\varphi}_{H^{k+m-1}}+\tnorm{\pd_j\psi}_{W^{k+m-1,\infty}}\tnorm{\varphi}_{L^2}.
			\end{equation}
		\end{enumerate}
\end{prop}
 \begin{proof}
     This is Corollary D.8 in Stevenson and Tice~\cite{stevenson2023wellposedness}.
 \end{proof}

%  % _+__+_ -_+__+_ -_+__+_ -_+__+_ -_+__+_ -_+__+_ -_+__+_ -_+__+_ -_+__+_ -_+__+_ -_+__+_ -_+__+_ -_+__+_ -

% _+__+_ -_+__+_ -_+__+_ -_+__+_ -_+__+_ -_+__+_ -_+__+_ -_+__+_ -_+__+_ -_+__+_ -_+__+_ -_+__+_ -_+__+_ -
\section{The shallow water system}\label{appendix on the shallow water system}
% _+__+_ -_+__+_ -_+__+_ -_+__+_ -_+__+_ -_+__+_ -_+__+_ -_+__+_ -_+__+_ -_+__+_ -_+__+_ -_+__+_ -_+__+_ -

% _+__+_ -_+__+_ -_+__+_ -_+__+_ -_+__+_ -_+__+_ -_+__+_ -_+__+_ -_+__+_ -_+__+_ -_+__+_ -_+__+_ -_+__+_ -
\subsection{Derivation of the forced-traveling model}\label{Ian's appendix post appendectomy}
% _+__+_ -_+__+_ -_+__+_ -_+__+_ -_+__+_ -_+__+_ -_+__+_ -_+__+_ -_+__+_ -_+__+_ -_+__+_ -_+__+_ -_+__+_ -

In this subsection we present a derivation of the forced-traveling formulation of the shallow water equations, which is system~\eqref{traveling wave formulation of the equation}. While there are numerous derivations of the damped shallow water equations in the literature (see, for example, Bresch~\cite{MR2562163}), to best of the authors' knowledge, there does not exist a derivation with the goal of determining the correct or physically relevant form of the forcing, nor is there a derivation  for which the starting point is the traveling free boundary incompressible Navier-Stokes equations, rather than the dynamic equations. What we present here is a derivation, largely inspired by the one of Bresch~\cite{MR2562163}, in which these two gaps are filled.  The resulting derivation then informs the structure of the system we study throughout the main body of the paper.

%%%%%%%%%%%%%%%%%%%%%%%%%%%%%%%%%%%%%%%%%%%%%%%%%
%%%%%%%%%%%%%%%%%%%%%%%%%%%%%%%%%%%%%%%%%%%%%%%%%
\subsubsection{Free boundary Navier-Stokes in traveling wave form}
%%%%%%%%%%%%%%%%%%%%%%%%%%%%%%%%%%%%%%%%%%%%%%%%% 
%%%%%%%%%%%%%%%%%%%%%%%%%%%%%%%%%%%%%%%%%%%%%%%%%

We assume a viscous incompressible fluid with viscosity $\Bar{\mu}>0$ fills a layer-like domain in $\R^{n}$, where $n = d+1$ for $d \ge 1$ the dimension of the fluid's cross-section.   We posit a traveling wave ansatz  moving in the $e_1$ direction with speed $\bar{\gamma} \in \R$, 
and we assume that in the moving coordinate system the fluid occupies the stationary domain $\Omega_{\bar{\eta}} = \tcb{(x,y) \in \R^{n} \;:\: 0 < y <\bar{\eta}(x) }$ with free boundary $\Sigma_{\bar{\eta}} = \tcb{(x,y) \in \R^n \;:\; y = \bar{\eta}(x)}$ for an unknown free surface function $\bar{\eta} : \R^d \to \R^+$.  We assume the equilibrium height of the fluid is $H >0$ with the corresponding equilibrium domain $\Omega_H= \R^d \times (0,H)$, flat upper boundary $\Sigma_H$, and flat lower boundary $\Sigma_0$.  The equilibrium domain can then be conveniently mapped to the moving one via $\mathfrak{F} : \Omega \to \Omega_{\bar{\eta}}$ given by $\mathfrak{F}(x,y) = (x,\bar{\eta}(x) y/H)$.  Following the derivation in Section 1 of Leoni and Tice~\cite{leoni2019traveling} (see also \cite{koganemaru_tice_23}) with minor modification, we may use $\mathfrak{F}$ to pose the traveling wave free boundary Navier-Stokes system in the fixed domain $\Omega_H$:
\begin{equation}\label{flattened_system}
\begin{cases}
(\bar{V}-\bar{\gamma} e_1) \cdot \grad_{\mathcal{A}} \bar{V}  - \bar{\mu} \Delta_{\mathcal{A}} \bar{V} + \grad_{\mathcal{A}} \bar{p}  = \bar{f} \circ \mathfrak{F}  & \text{in } \Omega_H \\
 \Div_{\mathcal{A}}{\bar{V}}=0 & \text{in } \Omega_H \\
 (\bar{p}I-  \bar{\mu} \mathbb{D}_{\mathcal{A}} \bar{V}) \mathcal{N} = (\bar{g} \bar{\eta} -\bar{\sigma} \mathcal{H}(\bar{\eta}) )\mathcal{N}  + \bar{\Xi}\circ \mathfrak{F} \mathcal{N} & \text{on } \Sigma_{H} \\
  \bar{V}\cdot \mathcal{N} + \bar{\gamma} \pd_1 \bar{\eta} = 0 &\text{on } \Sigma_{H} \\
 \bar{V} \cdot e_{n} =0 &\text{on } \Sigma_0 \\
 (\bar{\mu} \mathbb{D}_{\mathcal{A}} \bar{V} e_{n} )'=\bar{\alpha} \bar{V}' &\text{on } \Sigma_0.
\end{cases}
\end{equation}
Here, in addition to $\bar{\eta}$, the unknowns are the velocity field $\bar{V} : \Omega_H \to \R^n$ and the pressure $\bar{p}: \Omega_H \to \R$.  In these equations and in what follows we employ the standard notation of $'$ to indicate the first $d=n-1$ components of a vector in $\R^{n}$.  The vector field $\bar{f}: \R^n \to \R^n$ is a given applied force, and $\bar{\Xi} : \R^{n} \to \R^{n \times n}_{\text{sym}}$ is a given applied  symmetric stress tensor. $\mathcal{N}=(-\grad'\bar{\eta},1)$ is the non-unit normal vector. The surface tension coefficient is $\Bar{\sig}\ge0$ and $\mathcal{H}$ refers to the following mean curvature
\begin{equation}
    \mathcal{H}(\Bar{\eta})=\m{div}'\tp{\grad'\Bar{\eta}\tp{1+|\grad'\Bar{\eta}|^2}^{-1/2}}.
\end{equation}
The map $\mathcal{A} = (\grad \mathcal{F})^{-\m{t}}: \Omega_H \to \R^{n \times n}$ is 
\begin{equation} 
 \mathcal{A}(x,y) =  
\begin{pmatrix}
 I_{d \times d} & - y \grad' \bar{\eta}(x) /  \bar{\eta}(x) \\
0_{1 \times d} & H/\bar{\eta}(x)
\end{pmatrix},
\end{equation}
and the associated $\mathcal{A}-$differential operators are given via 
\begin{equation}
(\grad_{\mathcal{A}} \bar{p})_i  = \sum_{j=1}^n \mathcal{A}_{ij} \pd_j \bar{p}, \;   
\Div_{\mathcal{A}} \bar{V} = \sum_{i,j=1}^n \mathcal{A}_{ij}\pd_j \bar{V}_i, 
\text{ and } (\Delta_{\mathcal{A}} \bar{V})_i = \sum_{j=1}^n \sum_{k=1}^n \sum_{m=1}^n \mathcal{A}_{jk}\pd_k \tp{\mathcal{A}_{jm} \pd_m \bar{V}_i } 
\end{equation}
and (for a given vector field  $X$)   
\begin{equation}
(X \cdot \grad_{\mathcal{A}} \bar{V})_i = \sum_{j,k=1}^n X_j \mathcal{A}_{jk} \pd_k \bar{V}_i \text{ and } (\mathbb{D}_{\mathcal{A}} \bar{V})_{ij} =\sum_{k=1}^n \tp{ \mathcal{A}_{ik} \pd_k \bar{V}_j + \mathcal{A}_{jk} \pd_k \bar{V}_i }.
\end{equation}

%%%%%%%%%%%%%%%%%%%%%%%%%%%%%%%%%%%%%%%%%%%%%%%%%
%%%%%%%%%%%%%%%%%%%%%%%%%%%%%%%%%%%%%%%%%%%%%%%%%
\subsubsection{Scaling  }
%%%%%%%%%%%%%%%%%%%%%%%%%%%%%%%%%%%%%%%%%%%%%%%%% 
%%%%%%%%%%%%%%%%%%%%%%%%%%%%%%%%%%%%%%%%%%%%%%%%%

We introduce the scale $\ep>0$ and introduce the rescaled functions
\begin{equation}
\begin{array}{ccc}
   \bar{V}'(x,y) = \ep u(\ep x,y),  & \bar{V}_n(x,y) = \ep^2 v(\ep x,y), &
  \bar{p}(x,y) = \ep^2 p(\ep x, y), \\
  \bar{\eta}(x) = \eta(\ep x),   & 
   \bar{f}'(x,y) = \ep^3 f'(\ep x,y), &  \bar{f}_n(x,y) = \ep^{4} f_n(\ep x,y). 
\end{array}    
\end{equation}
This handles everything except for $\bar{\Xi}$, which we further decompose as 
\begin{equation}
\bar{\Xi} = 
\begin{pmatrix}
\bar{\Xi}' & \bar{\xi} \\
\bar{\xi} & \bar{\Xi}_{nn}
\end{pmatrix}
\text{ for }
\bar{\Xi}' : \R^{n} \to \R^{d \times d} \text{ and }\bar{\xi} : \R^{n} \to \R^d,
\end{equation}
and then rescale according to
\begin{equation}
\bar{\Xi}'(x,y)  = \ep^2 \Xi'(\ep x,y), \;
\bar{\xi}(x,y)   = \ep^3 \xi(\ep x,y), \;
\bar{\Xi}_{nn}(x,y)  = \ep^2 \Xi_{nn}(\ep x,y).
\end{equation}
Finally, we rescale the physical parameters via
\begin{equation}
     \bar{\gamma} = \ep \gamma, \;
 \bar{g} = \ep^2 g, \; 
 \bar{\mu} =  \mu, \;
 \bar{\sigma} = \sigma, \;
 \bar{\alpha} = \ep^2 \alpha.
\end{equation}

Plugging in these rescaled quantities yields the following. The momentum equation's horizontal part is
\begin{multline}\label{scaled_mom_h}
\ep^3 f'(\cdot,y \eta/H) = \ep^3 \grad' p -  \ep^3 (y/\eta) \grad'  \eta \pd_n p + 
\ep^3( u -  \gamma e_1) \cdot   \grad' u  +  \ep^3\tp{ (-y/\eta)\grad'\eta \cdot (u - \gamma e_1) (H/\eta )v} \pd_n v \\
- \mu \tp{\ep^3 \Delta' u + \ep ((H/\eta)^2 +\ep^2 (y/\eta)^2 \tabs{\grad' \eta}^2)  \pd_n^2 u- \ep^3 2(y/\eta) \grad \eta\cdot\grad \pd_n u 
- \ep^3 (y/\eta^2)( \eta  \Delta' \eta - 2\tabs{\grad'\eta}^2) \pd_n u },
\end{multline}
and its vertical part is
\begin{multline}\label{scaled_mom_v}
\ep^4 f_n (\cdot,y \eta/H) =  \ep^2 (H/\eta) \pd_n p + 
\ep^4 (u - \gamma e_1) \cdot \grad' v  
+ \ep^4 \tp{(-y/\eta)  \grad'\eta \cdot ( u -  \gamma e_1) +  (H/\eta) v }\pd_n v  \\
- \mu \tp{\ep^4 \Delta' v + \ep^2 ((H/\eta)^2 + \ep^2 (y/\eta)^2 \tabs{\grad' \eta}^2) \pd_n^2 v 
- \ep^4 2(y/\eta) \grad \eta\cdot\grad\pd_n v - \ep^4 (y/\eta^2)( \eta  \Delta'\eta -  2\tabs{\grad'\eta}^2) \pd_n v } . 
\end{multline}
The balance of mass becomes
\begin{equation}\label{scaled_bom}
\ep^2 \Div' u - \ep^2 (y/\eta)  \grad' \eta \cdot  \pd_n u + \ep^2 (H/\eta) \pd_n v  = 0.
\end{equation}
The dynamic boundary condition's normal part gives 
\begin{multline}\label{scaled_dyn_norm}
 \ep^2 p - \ep^2 g \eta +  \sigma \bp{ \frac{\ep^2 \Delta' \eta}{(1+ \ep^2 \tabs{\grad'\eta}^2)^{1/2}}  - \frac{\ep^4 (\grad')^2\eta \grad'\eta \cdot \grad'\eta  }{(1+\ep^2\tabs{\grad'\eta}^2)^{3/2}}  } 
- \mu \tp{ 
\ep^4 \mathbb{D}' u \grad'\eta \cdot \grad'\eta  - 2  \ep^4 \grad'\eta \cdot \grad' v 
\\
- \ep^4 2(y/\eta)\tabs{\grad' \eta }^2 \grad'\eta \cdot  \pd_n u 
- \ep^2 (H/\eta) \grad'\eta  \cdot \pd_n u
+  2 \ep^4 (y/\eta)  \tabs{\grad' \eta}^2 \pd_n v 
+ 2\ep^2 (H/\eta)\pd_n v
} \tp{1+\ep^2 \tabs{\grad'\eta}^2}^{-1} 
\\  
- 
\tp{    \ep^4 \Xi'(\cdot,\eta) \grad'\eta \cdot \grad'\eta -  \ep^4 2 \xi(\cdot,\eta) \cdot \grad'\eta    + \ep^2 \Xi_{nn} (\cdot,\eta) }
\tp{1+ \ep^2 \tabs{\grad'\eta}^2}^{-1} 
=0,
\end{multline}
while its tangential part yields (suppressing higher order terms)
\begin{multline}\label{scaled_dyn_tan}
\ep \mu(H/\eta) \pd_n u + \ep^3 \mu \tp{- \mathbb{D}' u \grad' \eta + 2(H/\eta) \pd_n u \tabs{\grad' \eta}^2      +  \grad' v   + (H/\eta) \partial_n v \grad' \eta  - (H/\eta) \grad' \eta \cdot \partial_n u \grad' \eta  } \\
+\ep^3\tp{- \Xi'(\cdot,\eta) \grad' \eta  +  \xi(\cdot,\eta) +   \Xi_{nn}(\cdot,\eta) \grad' \eta} = O(\ep^4)   
\end{multline}
and 
\begin{equation}
     \ep^2 \mu (H/\eta) \grad' \eta \cdot \pd_n u  = O(\ep^4).
\end{equation}
The kinematic boundary condition becomes
\begin{equation}\label{scaled_kin}
\ep^2 \gamma \pd_1 \eta - \ep^2 \grad' \eta \cdot u + \ep^2 v = 0. 
\end{equation}
The Navier-slip boundary conditions on the bottom (where $y=0$) read
\begin{equation}\label{scaled_navier}
 v = 0 
\text{ and } \ep \mu (H/\eta)\pd_n u = \ep^3 \alpha  u.
\end{equation}

%%%%%%%%%%%%%%%%%%%%%%%%%%%%%%%%%%%%%%%%%%%%%%%%%
%%%%%%%%%%%%%%%%%%%%%%%%%%%%%%%%%%%%%%%%%%%%%%%%%
\subsubsection{Formal expansion and equations of motion }
%%%%%%%%%%%%%%%%%%%%%%%%%%%%%%%%%%%%%%%%%%%%%%%%% 
%%%%%%%%%%%%%%%%%%%%%%%%%%%%%%%%%%%%%%%%%%%%%%%%%

We now make the assumption that for $F \in \{\eta,u,v,p,f,\Xi\}$ we have a formal expansion in powers of $\ep^2$:  $F = \sum_{k=0}^\infty \ep^{2k} F_k$. These expansions can then be plugged into the equations \eqref{scaled_mom_h}--\eqref{scaled_navier} to derive a hierarchy of equations.  As it turns out, for our purposes here we only need to track $\eta_0,u_0, v_0,p_0,f_0,\Xi_0$ and $u_1$.  In order to simplify notation we will drop the zero subscripts and denote $u_1$ by $\bar{u}$.

We begin by recording the lowest order terms.  The order $\ep^2$ terms in the equations \eqref{scaled_mom_h}--\eqref{scaled_bom} yield the following in $\Omega_H$:
\begin{equation}\label{lowest_mom_1}
 (H/\eta)^2 \pd_n^2 u =0, \;
 (H/\eta)^2 \pd_n^2 v + (H/\eta) \pd_n p =0,
 \text{ and }
 \Div' u +  (H/\eta) \pd_n v - (y/\eta) \grad' \eta \cdot \pd_n u = 0.
\end{equation}
On $\Sigma_H$ the $\ep^2$ part of \eqref{scaled_dyn_norm} and \eqref{scaled_dyn_tan} yield
\begin{equation}\label{lowest_dyn_n}
  p -  g \eta +  \sigma \Delta' \eta  -  \Xi_{nn}(\cdot,\eta)
- \mu(H/\eta)\tp{ -    \grad'\eta  \cdot \pd_n u
+   2\pd_n v}
=0,
\text{ and } 
\mu (H/\eta) \pd_n u =0.\end{equation}
The $\ep^2$ terms in \eqref{scaled_kin} and \eqref{scaled_navier} show that 
\begin{equation}\label{lowest_kin_ns}
 \gamma \pd_1 \eta -  \grad' \eta \cdot u +  v = 0 \text{ on } \Sigma_H, \text{ and } v=0 \text{ and }\mu(H/\eta) \pd_n u =0 \text{ on } \Sigma_0.
\end{equation}

The equation \eqref{lowest_mom_1}, together with \eqref{lowest_dyn_n} and \eqref{lowest_kin_ns} imply that $\pd_n u=0$, or in other words, $u = u(x)$.  We then integrate the third equation in \eqref{lowest_mom_1} vertically and use \eqref{lowest_kin_ns} to compute $(H/\eta) v  = -y  \Div' u$. In turn, this implies that $\pd_n^2 v =0$, and so \eqref{lowest_mom_1} implies that $\pd_n p =0$, so 
\begin{equation}\label{zero_p}
 p = p(\cdot,H) = g \eta - \sigma \Delta' \eta + \Xi_{nn}(\cdot,\eta) + 2\mu (H/\eta)\pd_n u = g \eta - \sigma \Delta' \eta + \Xi_{nn}(\cdot,\eta) - 2 \mu \Div' u.
\end{equation}
Finally, plugging into \eqref{lowest_kin_ns} shows that $\gamma \pd_1 \eta - \grad' \eta \cdot u = -v(\cdot,H) = \eta \Div' u$, and hence 
\begin{equation}\label{zero_eta}
\Div'( \eta(u-\gamma e_1)) =  \gamma \pd_1 \eta - \Div'(\eta u) =0.
\end{equation}
This is all we learn at the lowest level of the expansion in $\ep$, so we now proceed to the next.  

At order $\ep^3$ \eqref{scaled_mom_h} gives an equation for $u$, coupled to $\bar{u}$:
\begin{multline}\label{first_horiz}
f'(\cdot,y \eta/H) =
( u -  \gamma e_1) \cdot   \grad' u   
- \mu  \Delta' u   - \mu (H/\eta)^2 \pd_n^2 \bar{u} 
+ \grad' p  
\\
=
( u -  \gamma e_1) \cdot   \grad' u   
- \mu  \Delta' u   - \mu (H/\eta)^2 \pd_n^2 \bar{u} 
+ \grad'(g \eta - \sigma \Delta' \eta + \Xi_{nn}(\cdot,\eta) - 2\mu \Div' u). \end{multline}
At order $\ep^3$ \eqref{scaled_dyn_tan} yields
\begin{multline}\label{first_tangent}
   \Xi'(\cdot,\eta) \grad' \eta  -  \xi(\cdot,\eta) -   \Xi_{nn}(\cdot,\eta) \grad' \eta  
   = 
\mu (H/\eta) \pd_n \bar{u} + \mu \tp{- \mathbb{D}' u \grad' \eta + (H/\eta)\pd_n v \grad' \eta   +  \grad' v     } \\
=
 \mu (H/\eta) \pd_n \bar{u} - \mu  \mathbb{D}' u \grad' \eta  -  \mu \grad' ( \eta \Div'{u}  ) -\mu (\Div'{u}) \grad' \eta,
\end{multline} and \eqref{scaled_navier} yields $\mu (H/\eta)\pd_n \bar{u} =  \alpha  u$, which together show that
\begin{multline}\label{first_integrated}
 \int_0^H \mu (H/\eta)^2 \pd_n^2 \bar{u}(\cdot,y) \;\m{d}y =  (H/\eta)\tp{ 
 \Xi'(\cdot,\eta) \grad' \eta  -  \xi(\cdot,\eta) -   \Xi_{nn}(\cdot,\eta) \grad' \eta \\
+ \mu  \mathbb{D}' u \grad' \eta  +  \mu \grad' (\eta \Div'{u}  ) +\mu (\Div'{u}) \grad' \eta - \alpha u }. 
\end{multline}
Thus, upon averaging \eqref{first_horiz} over $y \in (0,H)$ and using \eqref{first_integrated}, we eliminate $\bar{u}$ and get 
\begin{multline}\label{second_integrated}
 \frac{1}{H} \int_0^H f'(\cdot,y \eta/H) \;\m{d}y    =( u -  \gamma e_1) \cdot   \grad' u   
- \mu  \Delta' u   - \tp{  \mu  \mathbb{D}' u \grad' \eta  +  \mu \grad' (\eta \Div'{u}  ) +\mu (\Div'{u}) \grad' \eta - \alpha u  }/\eta\\
-\tp{
\Xi'(\cdot,\eta) \grad' \eta  -  \xi(\cdot,\eta) -   \Xi_{nn}(\cdot,\eta) \grad' \eta
}/\eta
+ \grad'(g \eta - \sigma \Delta' \eta + \Xi_{nn} (\cdot,\eta)- 2\mu \Div' u).
\end{multline}
Note that
\begin{equation}
    \frac{1}{H}  \int_0^H f'(\cdot,y \eta/H) \;\m{d}y  =  \frac{1}{\eta} \int_0^\eta f'(\cdot,y) \;\m{d}y, 
\end{equation}
so \eqref{second_integrated} can be rewritten as
\begin{multline}
  \int_0^\eta f'(\cdot,y )\;\m{d}y    =\eta ( u -  \gamma e_1) \cdot   \grad' u   
- \mu \eta\Delta' u   - \tp{  \mu  \mathbb{D}' u \grad' \eta  +  \mu \grad' (\eta \Div'{u}  ) + \mu (\Div'{u}) \grad' \eta - \alpha u  } \\
- \tp{
\Xi'(\cdot,\eta) \grad' \eta  -  \xi(\cdot,\eta) -   \Xi_{nn}(\cdot,\eta) \grad' \eta
}
+ \eta \grad'(g \eta - \sigma \Delta' \eta + \Xi_{nn} (\cdot,\eta)- 2\mu \Div' u).
\end{multline}
Upon computing
\begin{multline}
    \Div'( \eta \mathbb{D}' u + 2 \eta (\Div' u) I  ) = \mathbb{D}' u \grad' \eta + \eta \Delta' u + 2 \Div'{u} \grad' \eta + 3 \eta \grad' \Div' u \\
    = \mathbb{D}' u \grad' \eta + \eta \Delta' u + \grad'(\eta \Div' u) + \Div'{u} \grad' \eta + 2 \eta \grad' \Div' u,
\end{multline}
we may rewrite the prior equation as
\begin{multline}\label{final_mom}
\eta ( u -  \gamma e_1) \cdot   \grad' u   
  + \alpha  u    
+ \eta \grad'(g \eta - \sigma \Delta' \eta)
- \mu \Div'( \eta \mathbb{D}' u + 2 \eta (\Div' u) I  ) \\
=  \int_0^\eta f'(\cdot,y )\;\m{d}y + 
\Xi'(\cdot,\eta) \grad' \eta  -  \xi(\cdot,\eta) -   \Xi_{nn}(\cdot,\eta) \grad' \eta
- \eta \grad\tp{  \Xi_{nn}(\cdot,\eta) }.
\end{multline}

We now combine \eqref{zero_eta} and \eqref{final_mom} and drop the primes on the differential operators to arrive at the forced-traveling shallow water system:
\begin{equation}\label{derived traveling shallow water system}
\begin{cases}
 \gamma \pd_1 \eta - \Div(\eta u) =0, \\
 \eta ( u -  \gamma e_1) \cdot   \grad u   
  + \alpha  u    
 + \eta \grad(g \eta - \sigma \Delta \eta)
- \mu\Div( \eta \mathbb{D} u + 2 \eta (\Div u) I  ) \\
\quad =\int_0^\eta f'(\cdot,y )\;\m{d}y + 
\Xi'(\cdot,\eta) \grad \eta  -  \xi(\cdot,\eta) -   \Xi_{nn}(\cdot,\eta) \grad \eta
- \eta \grad\tp{ \Xi_{nn}(\cdot,\eta) }.
\end{cases}
\end{equation}

% _+__+_ -_+__+_ -_+__+_ -_+__+_ -_+__+_ -_+__+_ -_+__+_ -_+__+_ -_+__+_ -_+__+_ -_+__+_ -_+__+_ -_+__+_ -
\subsection{Dissipation calculation}\label{appendix on dissipation calculation}
% _+__+_ -_+__+_ -_+__+_ -_+__+_ -_+__+_ -_+__+_ -_+__+_ -_+__+_ -_+__+_ -_+__+_ -_+__+_ -_+__+_ -_+__+_ -

The goal of this subsection is to justify the role of the forcing term $\mathcal{F}$ in the generation of nontrivial traveling wave solutions to the shallow water system.

\begin{prop}[Power-dissipation of the shallow water system]\label{prop on PD calculation}
    Suppose that $v\in H^2(\R^d;\R^d)$, $\eta\in\mathcal{H}^{3}(\R^d)$, and $\mathcal{F}\in L^2(\R^d;\R^d)$ obey system~\eqref{traveling wave formulation of the equation} for some choice of $\gam\in\R^+$, $\mu,\sig\in[0,\infty)$. Then identity~\eqref{power dissipation identity} is valid.
\end{prop}
\begin{proof}
    The strategy is to test the second equation in~\eqref{traveling wave formulation of the equation} with $v$ and integrate by parts.  The term involving $\int_{\R^d} \mathcal{F} \cdot v$ is left alone, but we recompute the rest.  First, for the dissipation term we readily obtain
    \begin{equation}\label{calculation for the dissipation term}
        \int_{\R^d}\tp{v-\mu^2\grad\cdot((1+\eta)\S v)}\cdot v=\int_{\R^d}|v|^2+\mu^2\bp{\f12|\mathbb{D}^0v|^2+2\bp{1+\f{1}{d}}|\grad\cdot v|^2}.
    \end{equation}
    Next, consider the advective derivative term:
    \begin{equation}\label{calculation for the advective derivative}
        \int_{\R^d}(1+\eta)\tp{(v-\gam e_1)\cdot\grad v}\cdot v=-\f12\int_{\R^d}\grad\cdot((1+\eta)(v-\gam e_1))|v|^2=0,
    \end{equation}
    where the final equality follows via applying the first equation in~\eqref{traveling wave formulation of the equation}. Finally, we consider the gravity capillary term. Fix a frequency cut-off parameter $\kappa\in(0,1)$ and split $\eta=\Uppi^{\kappa}_{\m{L}}\eta+\Uppi^\kappa_{\m{H}}\eta$ according to~\eqref{notation for the Fourier projection operators}. We then have the identities
    \begin{equation}\label{calculation for the grav cap 1}
        \int_{\R^d}(1+\eta)(I-\sig^2\Delta)\grad\eta\cdot v=\int_{\R^d}(1+\eta)(I-\sig^2\Delta)\grad\Uppi^\kappa_{\m{L}}\eta\cdot v-\int_{\R^d}(1-\sig^2\Delta)\Uppi^\kappa_{\m{H}}\eta\grad\cdot((1+\eta)v)
    \end{equation}
    and, thanks again to the first equation in~\eqref{traveling wave formulation of the equation} and orthogonality considerations, 
    \begin{equation}\label{calculation for the grav cap 2}
    \int_{\R^d}(1-\sig^2\Delta)\Uppi^\kappa_{\m{H}}\eta\grad\cdot((1+\eta)v)=\gam\int_{\R^d}(1-\sig^2\Delta)\Uppi^\kappa_{\m{H}}\eta\pd_1\Uppi^{\kappa}_{\m{H}}\eta=0.
    \end{equation}
    We combine~\eqref{calculation for the grav cap 1} and~\eqref{calculation for the grav cap 2} and then take the limit as $\kappa\to0$. The first term on the right hand side of~\eqref{calculation for the grav cap 1} vanishes and we are left with
    \begin{equation}\label{calculation for the gravity capillary term}
        \int_{\R^d}(1+\eta)(I-\sig^2\Delta)\grad\eta\cdot v=0.
    \end{equation}
    By synthesizing~\eqref{calculation for the dissipation term}, \eqref{calculation for the advective derivative}, and~\eqref{calculation for the gravity capillary term}, we obtain~\eqref{power dissipation identity}.
\end{proof}

\begin{coro}[On the uniqueness of the trivial solution]\label{corollary on uniqueness of the trivial solution}
    Suppose that $v\in H^2(\R^d;\R^d)$, $\eta\in\mathcal{H}^{3}(\R^d)$, and $\mathcal{F}\in L^2(\R^d;\R^d)$ obey system~\eqref{traveling wave formulation of the equation} for some choice of $\gam\in\R^+$, $\mu,\sig\in[0,\infty)$. If  $\mathcal{F}=0$ and $1+\eta\ge0$ a.e., then $v=0$ and $\eta=0$.
\end{coro}
\begin{proof}
    We apply Proposition~\ref{prop on PD calculation} and invoke the additional hypotheses to see that $\int_{\R^d}|v|^2\le 0$, and hence $v=0$. We then use the second equation in~\eqref{traveling wave formulation of the equation} to deduce that $(I-\sig^2\Delta)\grad\eta=0$,
    which implies that $\eta$ is a constant. By hypothesis, $\eta\in\mathcal{H}^3(\R^d)$, and this space does not contain nontrivial constants; therefore, $\eta=0$ as well.
\end{proof}
% _+__+_ -_+__+_ -_+__+_ -_+__+_ -_+__+_ -_+__+_ -_+__+_ -_+__+_ -_+__+_ -_+__+_ -_+__+_ -_+__+_ -_+__+_ -
\section*{Statements and Declarations}
% _+__+_ -_+__+_ -_+__+_ -_+__+_ -_+__+_ -_+__+_ -_+__+_ -_+__+_ -_+__+_ -_+__+_ -_+__+_ -_+__+_ -_+__+_ -

\small{\textbf{Conflict of interest:} There does not exist conflict of interest in this document.}

\vspace{0.25cm}

\small{\textbf{Data availability:} We do not analyze or generate any datasets, because our work proceeds within a theoretical and mathematical approach.}

% _+__+_ -_+__+_ -_+__+_ -_+__+_ -_+__+_ -_+__+_ -_+__+_ -_+__+_ -_+__+_ -_+__+_ -_+__+_ -_+__+_ -_+__+_ -

% _+__+_ -_+__+_ -_+__+_ -_+__+_ -_+__+_ -_+__+_ -_+__+_ -_+__+_ -_+__+_ -_+__+_ -_+__+_ -_+__+_ -_+__+_ -
\bibliographystyle{abbrv}
% _+__+_ -_+__+_ -_+__+_ -_+__+_ -_+__+_ -_+__+_ -_+__+_ -_+__+_ -_+__+_ -_+__+_ -_+__+_ -_+__+_ -_+__+_ -
\bibliography{bib.bib}
% _+__+_ -_+__+_ -_+__+_ -_+__+_ -_+__+_ -_+__+_ -_+__+_ -_+__+_ -_+__+_ -_+__+_ -_+__+_ -_+__+_ -_+__+_ -
\end{document}